\documentclass[11pt]{amsart}

\usepackage{epigamath}

\usepackage[english]{babel}

\numberwithin{equation}{section}

\usepackage[shortlabels,inline]{enumitem}

\usepackage[utf8]{inputenc}
\usepackage[T1]{fontenc}
\usepackage{amsmath, amssymb, amsthm}
\usepackage{mathtools}

\usepackage{microtype}
    \setlist{nosep}
\usepackage{booktabs, multirow, longtable} 
\usepackage{pgf,tikz}
\usepackage{siunitx}
\usepackage{quiver}
\usepackage{graphicx}
\usepackage{faktor}

\usepackage[all]{xy}
\usepackage{relsize}

\usetikzlibrary{arrows,decorations.pathmorphing,decorations.pathreplacing,backgrounds,positioning,fit,petri,calc,cd}
\usetikzlibrary{hobby}

\usepackage[tableposition=top]{caption}
\usepackage[capitalise]{cleveref}

\usepackage{lscape}
\usepackage{supertabular}
\usepackage{blkarray}

\usepackage{framed}
\usepackage{multirow}

\RequirePackage[normalem]{ulem}

\usepackage{framed}

\usepackage{url}

\definecolor{backcolour}{gray}{.9}

\theoremstyle{plain}
 \newtheorem{theorem}{Theorem}[section]
 \newtheorem{proposition}[theorem]{Proposition}
 \newtheorem{lemma}[theorem]{Lemma}
 \newtheorem{corollary}[theorem]{Corollary}

\theoremstyle{definition}
 \newtheorem{definition}[theorem]{Definition}
 \newtheorem{notation}[theorem]{Notation}
 \newtheorem{assumption}[theorem]{Assumption}

\theoremstyle{remark}
 \newtheorem{remark}[theorem]{Remark}
 \newtheorem{example}[theorem]{Example}
\newtheorem*{rem*}{Remark}
 
\newcommand{\IA}{\mathbb A}

\newcommand{\IC}{\mathbb C}

\newcommand{\IH}{\mathbb H}

\newcommand{\IN}{\mathbb N}
\newcommand{\IP}{\mathbb P}
\newcommand{\IQ}{\mathbb Q}

\newcommand{\IZ}{\mathbb Z}

\newcommand{\bQ}{\mathbb Q}

\newcommand{\Bl}{\mathrm{Bl}}

\newcommand{\id}{\mathrm{id}}

\newcommand{\tr}{\mathrm{tr}}

\newcommand{\K}{\mathrm{K}}

\newcommand{\ga}{{g_{\alpha}}}

\DeclareMathOperator{\Cl}{Cl}
\DeclareMathOperator{\Pic}{Pic}

\DeclareMathOperator{\Stab}{Stab}         
\DeclareMathOperator{\codim}{codim}         
\DeclareMathOperator{\Aut}{Aut}         
\DeclareMathOperator{\GL}{GL}           
\DeclareMathOperator{\ord}{ord}         
\DeclareMathOperator{\rank}{rk}       
\DeclareMathOperator{\SL}{SL}           
\DeclareMathOperator{\Fix}{Fix}

\DeclareMathOperator{\Exc}{Exc}

\DeclareMathOperator{\Sing}{\mathrm{Sing}}

\renewcommand{\Im}{\operatorname{Im}}

\newcommand{\revi}[1]{{\textcolor{black}{#1}}}
\newcommand{\rev}[1]{{\textcolor{black}{#1}}}

\newcommand{\corre}[1]{{\textcolor{black}{#1}}}
\newcommand{\new}[1]{{\textcolor{black}{#1}}}
\newcommand{\gio}[1]{{\textcolor{black}{#1}}}

\newcommand{\sslash}{\mathbin{/\mkern-4.5mu/}}

\newcommand{\lra}{\longrightarrow}

\DeclareRobustCommand\longhookrightarrow
    {\lhook\joinrel\relbar\joinrel\rightarrow}
\DeclareRobustCommand\longtwoheadrightarrow
    {\relbar\joinrel\twoheadrightarrow}

\newcommand{\reg}{\mathrm{reg}}

\newcommand{\symp}{\mathrm{symp}}
\newcommand{\sing}{\mathrm{sing}}
\newcommand{\bir}{\mathrm{bir.}}
\DeclareMathOperator{\PSL}{PSL}

\newcommand{\GroupNames}{\href{https://people.maths.bris.ac.uk/~matyd/GroupNames/index500.html}{GroupNames}}

\setlist[enumerate,1]{label={\rm(\arabic*)}, ref={\rm\arabic*}}

\newcommand{\supth}[1]{\ensuremath{#1^{\mathrm{th}}}}

\EpigaVolumeYear{9}{2025} \EpigaArticleNr{14} \ReceivedOn{February 14, 2024}
\InFinalFormOn{October 21, 2024}
\AcceptedOn{December 5, 2024}

\title{Terminalizations of quotients of compact hyperk\"{a}hler manifolds by induced symplectic automorphisms}
\titlemark{Terminalizations of symplectic quotients}

\author{Valeria Bertini}
\address{Dipartimento di Matematica dell'Universit\`a di Genova (DIMA), Via Dodecaneso, 35, 16146 Genova, Italy}
\email{bertini@dima.unige.it}

\author{Annalisa Grossi}
\address{Universit\'e Paris-Saclay, CNRS, Laboratoire de Math\'ematiques d'Orsay, Rue Michel Magat, B\^at. 307, 91405 Orsay, France}
\email{annalisa.grossi@universite-paris-saclay.fr, annalisa.grossi3@unibo.it}
\curraddr{Alma Mater studiorum Universit\`a di Bologna, Dipartimento di Matematica,
  Piazza di Porta San Donato 5, Bologna, 40126 Italia}

\author{Mirko Mauri} \author{Enrica Mazzon}
\address{Universit\'e Paris Cit\'e and Sorbonne Universit\'e, CNRS, IMJ-PRG, F-75013 Paris, France}
\email{mauri@imj-prg.fr, mazzon@imj-prg.fr}

\authormark{V.~Bertini, A.~Grossi, M.~Mauri, and E.~Mazzon}

\AbstractInEnglish{Terminalizations of symplectic quotients are sources of new deformation types of irreducible symplectic varieties. We classify all terminalizations of quotients of Hilbert schemes of K3 surfaces or of generalized Kummer varieties, by finite groups of symplectic automorphisms induced from the underlying K3 or abelian surface. We determine their second Betti number and the fundamental group of their regular locus. In the Kummer case, we prove that the terminalizations have quotient singularities and determine the singularities of their universal quasi-\'{e}tale cover. In particular, we obtain at least \rev{eight} new deformation types of irreducible symplectic varieties of dimension~$4$. Finally, we compare our deformation types with those in papers by Fu--Menet and by Menet. The smooth terminalizations are only three and of K$3^{[n]}$ type, and surprisingly they all appeared in different places in the literature.}


\acknowledgement{This project has received funding from the European Research Council (ERC) under the European
Union's Horizon 2020 (H2020) research and innovation
programme (ERC-2020-SyG-854361-HyperK). VB was supported by the PRIN Project 2020 \textit{Curves, Ricci ﬂat varieties and their interactions} and by Portuguese national funds through FCT (project EXPL/MAT-PUR/1162/2021) and through CMUP (project UIDB/00144/2020). AG was supported by
the ERC under H2020 research and innovation programme (ERC-2020-SyG-854361-HyperK) and by Funded by the European Union - NextGenerationEU under the National Recovery and Resilience Plan (PNRR) - Mission 4 Education and research - Component 2 From research to business - Investment 1.1 Notice Prin 2022 - DD N. 104 del 2/2/2022, from title ``Symplectic varieties: their interplay with Fano manifolds and derived categories,'' proposal code 2022PEKYBJ – CUP J53D23003840006. MM was supported by the University of Michigan, the Hausdorff Institute of Mathematics in Bonn, the Institute of Science and Technology Austria, \'{E}cole Polytechnique in France, Universit\'e Paris Cit\'e and Sorbonne Universit\'e. This project has received funding from the H2020 research and innovation programme
under the Marie Skłodowska-Curie grant agreement No.~101034413.
EM was supported by the collaborative research center SFB 1085 \emph{Higher Invariants - Interactions between Arithmetic Geometry and Global Analysis} (funded by the Deutsche Forschungsgemeinschaft) and by Universit\'e Paris Cit\'e and Sorbonne Universit\'e.}

\begin{document}


\removeabove{0.8cm}
\removebetween{0.8cm}
\removebelow{0.8cm}

\maketitle

\begin{prelims}

\DisplayAbstractInEnglish

\bigskip

\DisplayKeyWords

\medskip

\DisplayMSCclass







\end{prelims}


\newpage

\setcounter{tocdepth}{1}

\tableofcontents


\section{Introduction}

\subsection{Irreducible symplectic varieties}
Irreducible symplectic varieties play a key role in the classification of varieties with Kodaira dimension zero. In the last decades, fundamental results about their birational geometry, algebraic cycles and moduli theory have been proved; see for instance \url{https://www.erc-hyperk.org/papers}
for a list of the latest advances in the field. The importance of irreducible symplectic  varieties rests on the celebrated Beauville--Bogomolov decomposition, proved in increasing degree of generality in \cite{B1983, GrebKebekusKovacsPeternell, DruelGuenancia, Druel2018, Guenancia2016, GrebGuenanciaKebekus, HoringPeternell, Campana2021, BGL2020}: Any compact K\"{a}hler space with numerically trivial canonical class and klt singularities admits a quasi-\'{e}tale cover\footnote{A quasi-\'{e}tale cover is a finite morphism \'{e}tale in codimension~$1$.} which can be written as the product of complex tori, strict Calabi--Yau varieties or irreducible symplectic varieties. 

It is expected that the number of deformation types of irreducible symplectic varieties is finite in each dimension; see \cref{rem:finitenessIHS}. Therefore, it is natural to ask whether it is possible to even classify irreducible symplectic varieties, at least in low dimension. Despite active research in the field, irreducible symplectic varieties (especially smooth ones) are notoriously difficult to construct. At the moment, in the smooth case, there are in each dimension at most three known deformation types of irreducible symplectic manifolds, see \cite{B1983, OG6,OG10}, namely those of
\begin{itemize}
    \item Hilbert schemes $S^{[n]}$ of $n$ points on a K3 surface $S$, 
    \item generalized Kummer varieties $\K_n(A)$ associated to an abelian surface $A$, 
    \item two sporadic examples built by O'Grady in dimensions $6$ and $10$. 
\end{itemize}
Dropping the smoothness assumption, we can generate more examples. For instance, there are, in \cite{GM22} alone, at least 29 distinct deformation types of singular $4$-dimensional irreducible symplectic orbifolds. The implicit hope is that while studying singular symplectic varieties, one may find some of them admitting a symplectic resolution, so ideally new smooth examples. Historically this is indeed how the O'Grady examples in dimensions $6$ and $10$ were discovered.

All known deformation types of irreducible symplectic varieties arise in the following ways:
\begin{itemize}
    \item moduli spaces of semistable sheaves on K3 or abelian surfaces \cite{PR2023}, 
    \item compactifications of Lagrangian fibrations, see \cite{MT2007, AS2015, M2016, SS2022, BCGPSV, LLX2024}, 
    \item terminalizations $p \colon Y \to X/G$ of symplectic quotients of a symplectic variety $X$ by a finite group $G$, see \cite{Fujiki1983, Fu-Menet, GM22}, 
    \[
\begin{tikzcd}[scale=1]
    & X \arrow[d, "q"]\\
   Y\arrow[r, "p"] & X/G\rlap{.}
\end{tikzcd}
\]
\end{itemize}
See also the survey \cite{Perego2020}. 

The purpose of this paper is to study systematically terminalizations of quotients of known irreducible symplectic manifolds. In particular, we complete part of the classification program designed by Menet in \cite[Section~1.3]{GM22}.

\subsection{Criteria for an efficient classification of terminalizations} For a sensible and efficient
description of the terminalizations above, some reductions and assumptions are in order. We first propose to restrict to the case of \begin{quote} \centering \emph{projective $\bQ$-factorial terminalizations $Y$ of symplectic quotients $X/G$\\ with simply connected regular locus $Y^{\reg}$.\footnote{Some authors call irreducible symplectic varieties with quotient singularities and simply connected regular locus  \emph{irreducible orbifolds}. We avoid this convention as it competes with the now well-established definition of irreducible symplectic varieties and it may cause confusion: An irreducible symplectic orbifold whose regular locus has nontrivial fundamental group would not be an irreducible orbifold!}}
\end{quote} 
Although the combination of quotients and birational modifications of $X/G$ is a source of many more irreducible symplectic varieties, they should be considered redundant as we explain in Section~\ref{sec:criteria}. Concretely, the reduction above requires that the candidate $G$-actions  satisfy the following conditions; see Section~\ref{subsection sympectic}, Section~\ref{sec:terminalization} and \cref{prop:fundamentalgroup} for the equivalence.

\begin{assumption}\label{assumption:codim2}
     The following equivalent conditions 
     hold:
     \begin{enumerate}
    \item\label{a:c2-1} $X/G$ has strictly canonical singularities. 
    \item\label{a:c2-2} The singular locus of $X/G$ has codimension~$2$.
    \item \label{condition:item} An element of $G$ fixes a codimension~$2$ subvariety in $X$.\footnote{If $X$ has $\mathbb{Q}$-factorial singularities, Assumption~\ref{assumption:codim2} is equivalent to the following condition:
\begin{enumerate}
    \item[(4)] The $\bQ$-factorial terminalization of $X/G$ is a nontrivial morphism.
\end{enumerate}}
\end{enumerate}
\end{assumption}

\begin{assumption}\label{ass:groupaction}
    The finite group $G$ acts on $X$ in such a way that the automorphisms whose fixed locus in~$X$ has codimension two generate the entire group $G$.
\end{assumption}

In this paper, we study the case of $X$ being a known irreducible symplectic manifold. In view of \revi{Assumption~\ref{assumption:codim2},} we can rule out the case of manifolds of O'Grady type as explained in \cref{rmk:finitesymplgroup}. Without loss of generality, we can then restrict to the case of Hilbert schemes or generalized Kummer varieties. 

Finite groups of symplectic automorphisms of them have been extensively studied in the literature; see \cref{rmk:finitesymplgroup}. However, the lattice-theoretic information of these classifications seems insufficient to prescribe the geometry and the intersection theory of the fixed loci, and ultimately the geometry and singularities of~$Y$. In order to maintain control over the geometry of the fixed loci, in this paper we assume the following.

\begin{assumption}\label{assumptioninduced}
    The finite group $G$ acts on $S^{[n]}$ or $\K_{n}(A)$ via symplectic automorphisms induced by automorphisms of the underlying K3 or abelian surface $S$ or $A$, respectively.
\end{assumption}

While Assumptions~\ref{assumption:codim2} and~\ref{ass:groupaction} are necessary to obtain an efficient classification (and should be required even in future works on the subject), Assumption~\ref{assumptioninduced} should be considered primarily as a technical requirement. Indeed, not all symplectic automorphisms with fixed loci of codimension~$2$ (so satisfying Assumption~\ref{assumption:codim2}) are induced. Consider for instance the example of a non-induced automorphism of order~$3$ on a variety of $\K3^{[2]}$ type \gio{in \cite[Example 17(iv)]{Nam2001}; \textit{cf.} also \cite[Section~3]{kawatani2009birational}.} 

There are certainly other classes of automorphisms whose fixed loci may be controlled effectively. For example, to also keep into account the Namikawa--Kawatani automorphism above, it would be interesting to also classify  quotients of Fano varieties of lines on cubic fourfolds induced by automorphisms of the underlying cubic fourfold, or automorphisms of EPW sextics, or the more challenging automorphisms of moduli spaces of semistable sheaves induced by automorphisms of the underlying surface. We plan to tackle some of these other cases in the near future and include them in the classification program of \cite[Section~1.3]{GM22}.

\subsection{Classification results}
We first show that the geometric Assumptions~\ref{assumption:codim2} and~\ref{assumptioninduced} impose group-theoretic constraints on $G$ and on the dimension of $X$.

\begin{theorem}\label{thm:canonicalsingularities} Let $G$ be a finite group of induced symplectic automorphisms acting on $X \simeq S^{[m]}$ or $\K_n(A)$. Then $X/G$ has strictly canonical singularities if and only if one of the following holds:
\begin{itemize}
    \item $m=2$ or $n=2$, and $G$ contains an involution.
    \item $n=2$, and $G$ contains a special type of automorphisms 
      of order~$3$ as in \cref{lemma:inducedinvolution}\,\eqref{l:i-2}.
    \item $n=3$, and $G$ contains a special type of involutions as in \cref{lemma:inducedinvolution}\,\eqref{l:i-1}. 
\end{itemize}
In particular, $X$ is isomorphic to $S^{[2]}$, $\K_2(A)$ or $\K_{3}(A)$. 
\end{theorem}
\begin{proof} This follows from Lemmas~\ref{lemma:codime2} and~\ref{lemma:inducedinvolution}.
\end{proof}

\begin{theorem}[\textit{cf.} \cref{prop:blowup}]\label{thm:blowuponetime}
    \new{Away from the dissident locus $($see \cref{def:dissident}\,$)$, a terminalization of\, $X/G$ as in \cref{thm:canonicalsingularities} is isomorphic to the blowup of the reduced singular locus.}
\end{theorem}

It is open whether \cref{thm:blowuponetime} holds unconditionally without Assumption~\ref{assumptioninduced}. 

\begin{theorem}[Second and third Betti numbers, \textit{cf.} \cref{formulaBetti}, \rev{\cref{remark:binarydihedral}} and \cref{prop:IH3}]
Let $G$ be a finite group of induced symplectic automorphisms of\, \(X=S^{[n]}\) or $\K_n(A)$. Let $q \colon X \to X/G$ be the quotient map, $p \colon Y \to X/G$ be a terminalization of\, $X/G$, and $\Sigma$ be the singular locus of\, $X/G$. Denote by
\begin{itemize}

\item $F_g \subset X$ the $($unique$)$ component of the fixed locus of $g \in G$ of codimension~$2$, if any,  
\item $L$ a lattice isomorphic to $H^2(X,\IZ)$, 
    \item $N_2$ the number of components $q(F_g)$ in $\Sigma$ with $\ord(g)=2$,
    \item \(N_3\) the number of components $q(F_g)$ in $\Sigma$ with $\ord(g)=3$.
\end{itemize} 
Then the following topological identities hold:
  \begin{align*}      b_2(Y)=\rank\left(L^G\right)+N_2+2N_3 - \epsilon,\\     
    IH^3(Y, \IQ) \simeq H^3(X, \IQ)^{G}, 
     \end{align*}
     where $IH^*(Y, \IQ)$ stands for the intersection cohomology of\, $Y$ with rational coefficients, and $\epsilon$ equals $1$ if $X=K_{2}(A)$ and $G_{\circ} \simeq BD_{12}$ $($cf. \cref{sec:notation}\,$)$, or $0$ otherwise.
\end{theorem}

\begin{theorem}[\textit{cf.} Tables~\ref{table Hilb2(K3)},~\ref{table n=2} and~\ref{table n=3}]
    For any action of a finite group of symplectic automorphisms of\, $X\simeq S^{[2]}, \K_2(A)$ or $\K_{3}(A)$ induced by the underlying K3 or abelian surface, the second Betti number and fundamental group of the regular locus of a \new{projective} terminalization $Y$ of the quotient $X/G$ \revi{are listed in Tables~\ref{table Hilb2(K3)},~\ref{table n=2} and~\ref{table n=3}.}
\end{theorem}

If $X\simeq S^{[2]}$, the topological invariants of $Y$ depend only on the abstract isomorphism type of $G$, while in the Kummer case, they rely on the actual action of the group $G$ and neither on the abstract isomorphism type nor on the induced action in cohomology; see \cref{ex:differentquotient}. In any case, we find a group-theoretic description of these topological invariants depending on the embedding of $G$ in the automorphism group of the underlying surface; see \cref{formulaBetti2} and \cref{cor:fundgroupKA}.

\begin{theorem}[\textit{cf.} Theorem~\ref{thm:S2new}]
    Let $G$ be a finite group of induced symplectic automorphisms acting on $S^{[2]}$ and $Y$ a projective terminalization of $S^{[2]}/G$ with simply connected regular locus. There are at least five new deformation classes of such irreducible symplectic varieties $Y$. In particular, they are not deformation equivalent to any terminalization of quotients of Kummer fourfolds by groups of induced symplectic automorphisms, or a Fujiki fourfolds appearing in \cite[Theorem 1.11]{GM22} $($cf. \cref{defb:Fujiki variety}\,$)$.
\end{theorem} 

\begin{theorem}[\textit{cf.} Table~\ref{table sing n=2}]\label{thm:class2}
    Let $G$ be a finite group of induced symplectic automorphisms acting on $\K_2(A)$ and $Y$ a \new{projective} terminalization of\, $\K_2(A)/G$ with simply connected regular locus. The Betti numbers, Chern classes and singularities of\, $Y$ are listed in Table~\ref{table sing n=2}. 
    
   In particular, there exist at least three new deformation classes of irreducible symplectic orbifolds of dimension~$4$. All other terminalizations are deformation equivalent to Fujiki varieties, with the exception of\, $G=C_2$ $($called a Nikulin orbifold\,$)$ and possibly $G=BT_{24}$.
\end{theorem}

\begin{theorem}[\textit{cf.} Table~\ref{table sing n=3} and \cref{lem:quotientK3}]\label{thm:class3}
    Let $G$ be a finite group of induced symplectic automorphisms acting on $\K_3(A)$ and $Y$ a \new{projective} terminalization of\, $\K_3(A)/G$ with simply connected regular locus. \revi{The second Betti number and the singularities of\, $Y$ are listed in Table~\ref{table sing n=3}. In particular, $Y$ is deformation equivalent to one of the Fujiki sixfolds appearing in \cite[Section~6]{GM22}.}
\end{theorem}

\begin{corollary}[\textit{cf.} \cref{cor:quotientsingularitiestext}]\label{cor:quotientsingularities}
    Any \new{projective} terminalization of a quotient of\, $\K_2(A)$ or $\K_3(A)$ by a finite group of induced symplectic automorphisms has quotient singularities.
\end{corollary}

If instead $X \simeq S^{[2]}$, the configurations of the singularities and topological invariants have already been studied in \cite{GM22} for so-called \emph{admissible}
symplectic groups. We have been informed that Menet is working on non-admissible group actions. 

It is natural to ask whether some of the previous terminalizations are smooth. We show that this happens only in three cases, and quite surprisingly they already appeared in the literature, scattered over three different places.

\begin{theorem}[Smooth terminalizations] \label{thm:smooth term}
Let \(G\) be a nontrivial finite group of induced symplectic automorphisms of\, \(X=S^{[n]}\) or \(\K_n(A)\). The quotient \(X/G\) admits a smooth terminalization if and only if
\begin{enumerate}[itemsep=2pt]
    \item $X=S^{[2]}$ and $G \simeq C_2^{4}$, see \emph{\cite[Proposition 14.5]{Fujiki1983}}, 

    \item $X=\K_2(A)$ and $G \simeq C_3^{3}$
    acting nontrivially on $H^2(A, \IZ)$, see \emph{\cite[Theorem 4.2]{kawatani2009birational}}, 

    \item $X=\K_3(A)$ and $G \simeq C_2^{5}$, see \emph{\cite[Theorem 1.1]{Floccari2022}}.
\end{enumerate}
All three terminalizations are birational to an irreducible symplectic manifold of\, $\K3^{[n]}$ type.
\end{theorem}
\begin{proof}
    This follows by direct inspection of our tables; see \cref{prop:K3_2 terminalization nonsmooth}, \cref{table sing n=2} and \cref{prop:K3A terminalization nonsmooth}. See also \cref{rmk:groupaction}. 
\end{proof}

\subsection{Second Betti numbers}
The study of terminalizations of quotients of symplectic manifolds goes back to the work of Fujiki. Nowadays, Fujiki varieties are
terminalizations of certain quotients of squares of K3 surfaces; see \cref{defb:Fujiki variety}. Their classification was initiated by Fujiki in \cite{Fujiki1983} and recently completed in \cite{GM22}. Terminalizations of cyclic quotients of $\K_2(A)$ and $S^{[2]}$ have also been studied  by Fu and Menet in \cite{Fu-Menet}. There, their interest was not to provide an exhaustive classification of all possible terminalizations arising, but rather to realize examples of irreducible symplectic fourfolds with different second Betti numbers.  

In the following table, we compare the second Betti numbers of irreducible symplectic fourfolds constructed in \cite{Fu-Menet}, \cite{GM22} and in the present paper. 

\small
\begin{longtable}{c|c|c|c|c|c|c|c|c|c|c|c|c|c|c|c|c|c|c|c|c|c}
\label{table Betti 4folds}
\revi{$b_2$}
 & 3
 & 4
 & 5
 & 6
 & 7
 & 8
 & 9
 & 10
 & 11
 & 12
 & 13
 & 14
 & 15
 & 16
 & 17
 & 18
 & 19
 & 20
 & 21
 & 22
 & 23
 \\\hline 
 \cite{Fu-Menet}
 & \(\circ\)
 &
 & \(\circ\)
 & \(\circ\)
 & \(\bullet\)
 & \(\bullet\)
 &
 & \(\bullet\)
 & \(\circ\)
 &
 &
 & \(\bullet\)
 &
 & \(\bullet\)
 & 
 &
 &
 &
 &
 &
 & \(\bullet\)
 \\ \hline
 \cite{GM22}
 &
 & \(\bullet\)
 & \(\bullet\)
 & \(\bullet\)
 & \(\bullet\)
 & \(\bullet\)
 &
 & \(\bullet\)
 & \(\bullet\)
 &
 &
 & \(\bullet\)
 &
 & \(\bullet\)
 &
 &
 &
 &
 &
 &
 & \(\bullet\)
 \\ \hline
 Present paper
 &
 &
 & \(\bullet\)
 & \(\bullet\)
 & \(\bullet\)
 & \(\bullet\)
 & 
 & \(\bullet\)
 & \(\bullet\)
 &
 &
 & \(\bullet\)
 & 
 & \(\bullet\)
 &
 &
 &
 &
 &
 &
 & \(\bullet\) \vspace{0.4 cm}\\
 \caption{The first row lists all possible second Betti numbers of irreducible symplectic fourfolds. A circle in the table corresponds to known examples of such fourfolds: The column gives their second Betti number, and the row indicates a reference in the literature where the examples appear. Black circles correspond to examples with simply connected regular locus. White circles mean that the regular locus of all examples with a fixed Betti number (column) in a given reference (row) is not simply connected.} \vspace{-0.5 cm}
 \end{longtable}

\normalsize
Observing the table, it is natural to ask the following:  
\begin{enumerate}
    \item \emph{Is there an irreducible  symplectic variety $Y$ with $b_2(Y)=3$ and $\pi_1(Y^{\reg})=1$?} A nontrivial terminalization of a quotient of a symplectic variety will always have at least $b_2 \geq 4$. In fact, the example of \cite{Fu-Menet} with $b_2=3$ is a quotient of a variety of $\K3^{[2]}$ type by an automorphism of order $11$, but its regular locus is not simply connected. 
    \item \emph{Is there an irreducible symplectic variety $Y$ with $b_2(Y)=4$ and non-quotient singularities?} In \cite{GM22}, Menet exhibits some Fujiki orbifolds of dimension~$4$ with the smallest Betti number possible, namely $b_2=4$. This Betti number cannot be realized by terminalizing the quotient of $\K_2(A)$ or $S^{[2]}$ by induced symplectic automorphisms (see Tables~\ref{table Hilb2(K3)} and~\ref{table n=2}), while examples in dimension $6$ appear in Table~\ref{table n=3}. It would be interesting to find an example with \new{($\IQ$-factorial terminal)} non-quotient singularities since at the moment the global Torelli theorem is not known in this case; see \cite[Theorem 1.1]{BL2022} and \cite{Menet2020}.
    \item \emph{Are there examples of irreducible symplectic fourfolds, possibly with simply connected regular locus, with $b_2=9,12,13,15$ or $16< b_2 < 23$.}
    \item \emph{Is there a conceptual explanation for the gap $16 < b_2 < 23$?} Note that the only examples with $b_2=23$ and $b_2=16$ are, respectively, $S^{[2]}$ and a Nikulin orbifold, \textit{i.e.}, a terminalization of a quotient of $S^{[2]}$ by an involution. Further, a $4$-dimensional irreducible symplectic orbifold with $b_2=23$ is necessarily smooth by \cite[Theorem 1.3]{Fu-Menet}. 
\end{enumerate}
\medskip

Terminalizations of sixfolds are less studied in the literature. In \cref{table Betti 6folds}, we summarize the second Betti numbers of irreducible symplectic sixfolds constructed in this paper as terminalizations of quotients of~$\K_3(A)$.  

\small
\begin{longtable}{c|c|c|c|c|c|c|c|c|c|c|c|c|c|c|c|c|c|c|c|c}
\label{table Betti 6folds}
 3
 & 4
 & 5
 & 6
 & 7
 & 8
 & 9
 & 10
 & 11
 & 12
 & 13
 & 14
 & 15
 & 16
 & 17
 & 18
 & 19
 & 20
 & 21
 & 22
 & 23
 \\\hline 
 
 & \(\circ\)
 & \(\circ\)
 & \(\circ\)
 & \(\circ\)
 & \(\bullet\)
 & \(\bullet\)
 & 
 & \(\bullet\)
 & 
 &
 &
 & \(\bullet\)
 & 
 & 
 & 
 &
 &
 &
 &
 & \(\bullet\)\vspace{0.3 cm}\\
 \caption{Second Betti numbers of terminalizations of quotients of \(\K_3(A)\), \revi{in the same notation as in Table~\ref{table Betti 4folds}}.}
 \end{longtable}
\vspace{-0.5 cm}
\normalsize

\subsection{General results on terminalizations} In the perspective of producing examples of irreducible symplectic varieties, we prove new results about terminalizations of symplectic varieties of independent interest.

\begin{proposition}[Terminalization of symplectic varieties, \textit{cf.} \cref{prop:terminalization}]
Let $f \colon Y \to X$ be a proper birational \revi{morphism onto} a symplectic variety $X$. Let $X^{\circ}$ be the complement of the dissident locus $($see \cref{def:dissident}\,$)$ and $f^{\circ} \colon Y^{\circ} \to X^{\circ}$ be the unique symplectic resolution of\, $X^{\circ}$. 

Then $f$ is a terminalization of\, $X$ if and only if $f$ is a compactification of $f^{\circ} \colon Y^{\circ} \to X^{\circ}$ and we have $\codim\left(Y\setminus Y^{\circ}\right)\geq2$. 
\end{proposition}

\begin{proposition}[Terminalization of symplectic varieties with only quotient singularities, \textit{cf.} \cref{prop:Qfactorialterminalization}]
Let $f \colon Y \to X$ be a projective $\IQ$-factorial terminalization of a complex symplectic variety $X$ with only quotient singularities. Suppose that the divisor $E|_{f^{-1}(U)}$ is irreducible for any prime exceptional divisor $E \subset Y$ and for any connected open set $U \subseteq X$ in the Euclidean topology. Assume further that $U$ is a connected Euclidean neighborhood of some $x \in X$, and let $T \to U$ be a projective terminalization of\, $U$. Then, 
up to shrinking $U$ to a smaller neighborhood of $x$, $Y_{U}\coloneqq f^{-1}(U)$ admits a locally trivial deformation to $T$. Furthermore, $Y_U$ and $T$ are locally analytically $\IQ$-factorial over $U$ $($see \cref{defn:Q-factorial}\,$)$, and they have the same singularities $($in the sense of \cref{prop:singulartermina}\,$)$. 
\end{proposition}

\begin{rem*}
A deformation $\pi \colon \mathcal{X} \to S$ is called \emph{locally trivial} if for any $x \in \mathcal{X}$, there exist analytic neighborhoods
$\mathcal{U} \subset \mathcal{X}$ of $x$ and $S_{0} \subset S$ of $\pi(x)$ such that $\mathcal{U} \simeq S_0 \times U$, where $U \coloneqq \pi^{-1}(\pi(x)) \cap \mathcal{U}$; see \textit{e.g.} \cite[Section~1.2.1]{Sernesi2006} or \cite[Definition~4.1]{BL2022}.
\end{rem*}

\begin{proposition}[Fundamental group of terminalizations, \textit{cf.} \cref{prop:fundamentalgroup}]
Let $X$ be a simply connected smooth symplectic variety endowed with an action of a finite group $G$ of symplectic automorphisms. Let
$p \colon Y \to X/G$ be a terminalization of the quotient. The fundamental group of the regular locus of\, $Y$ is 
\[\pi_1(Y^{\reg}) \simeq G/N,\]
where $N \triangleleft G$ is the normal subgroup generated by elements $\gamma \in G$ whose fixed locus in $X$ has codimension~$2$. The universal quasi-\'{e}tale cover of\, $Y$ is a terminalization of the quotient $X/N$.
\end{proposition}

\subsection{Outline}
\begin{itemize}
    \item In Section~\ref{sec:symplectic}, we recall the definition of (irreducible) symplectic variety and describe properties of their terminalizations.
    \item In Section~\ref{sec:criteria}, we motivate the criteria of the classification and comment in particular on Assumptions~\ref{assumption:codim2} and~\ref{ass:groupaction}.
    \item In Section~\ref{sec:termexplicit}, we specialize the previous results to the case of quotients of irreducible symplectic manifolds $S^{[n]}$ or $\K_n(A)$ by induced automorphism groups. For this purpose, in Section~\ref{sec:codim2 fixed loci}, we show that the codimension~$2$ fixed loci of induced symplectic automorphisms are subject to severe constraints.  
    \item In Sections~\ref{sec:secondbetti},~\ref{sec:third} and~\ref{sec:fundgroup}, we provide group-theoretic formulas for the second and third Betti numbers of~$Y$ and the fundamental group of the regular locus of $Y$.
    \item We list the second Betti number and fundamental group of the regular locus of all terminalizations of induced symplectic quotients of $S^{[2]}$, $\K_2(A)$ and $\K_3(A)$, respectively in Table~\ref{table Hilb2(K3)} (Section~\ref{sec:classificationK3}), Table~\ref{table n=2} (Section~\ref{sec:terminalizationKnKn}) and Table~\ref{table n=3} (Section~\ref{sec:terminalizationKnKn}). 
    \item In \cref{table sing n=2} (Section~\ref{sec:terminalizationfour}), we list Betti numbers, Chern classes and singularities of the terminalizations of quotients $\K_2(A)/G$ with simply connected regular locus. The analysis of the singularities is carried on in Section~\ref{sec:sing computation n=2}.
    
    \item In Section~\ref{sec:terminalization6} \corre{and in \cref{table sing n=3} (Section~\ref{sec:terminalizationfour})}, we describe the singularities of the terminalizations of quotients $\K_3(A)/G$ with simply connected regular locus.
    \item In Section~\ref{sec:birational orbifolds}, we determine whether terminalizations with the same topological invariants are actually deformation equivalent.
\end{itemize}

\subsection{Acknowledgements} This project started during the online session of the \emph{Interactive Workshop \& Hausdorff School} held at University of Bonn in September 2021, as part of the events of the ERC Sinergy Grant HyperK. We warmly thank Benjamin Bakker for suggesting the problem and inspiring the authors. Armando Capasso and Olivier Debarre took part in the online session of the workshop, and we thank them for all their suggestions and for comments on earlier versions of this paper. We also thank Gwyn Bellamy, Simon Brandhorst,  Igor Dolgachev, Salvatore Floccari, Osamu Fujino, Daniel Huybrechts, Micha\l{} Kapustka, Gr\'egoire Menet, Giovanni Mongardi, Yoshinori Namikawa, Travis Schedler for useful conversations and valuable email exchanges. We are grateful to the Simons Center for Geometry and Physics in Stony Brook for offering perfect working conditions to pursue our project on the occasion of the \emph{Workshop ``Hyperk\"ahler quotients, singularities, and quivers''} held in January 2023. 

\section{Notation}\label{sec:notation}

\begin{itemize}
    \item Denote by $S^{(n)}\coloneqq S^{n}/S_{n}$ the $n$-fold symmetric product of the surface $S$. A point in $S^{(n)}$ is an unordered $n$-tuple $[(x_1, \ldots, x_{n})]$ or the formal sum $x_1 + \dots + x_n$ with $x_i \in S$.
    \item The Hilbert--Chow morphism 
    \[\epsilon \colon S^{[n]}\lra S^{(n)}, \quad \xi \longmapsto \sum_{x \in \xi} \mathrm{length}(\mathcal{O}_{\xi, x}) \cdot x,\]
    sends any subscheme $\xi$ of length $n$ in the surface $S$ to its weighted support. It is a symplectic resolution of the symmetric product $S^{(n)}$.
    \item The generalized Kummer variety $\K_{n}(A)$ is the fiber over $0$ of the composition
    \[A^{[n+1]} \overset{\epsilon}\lra A^{(n+1)} \overset{s}\lra A,\]
    where $\epsilon$ is the Hilbert--Chow morphism and $s$ is the summation map. We denote by $A^{(n+1)}_0$ the fiber over $0$ of $s$. The restriction $\epsilon \colon \K_{n}(A) \to A^{(n+1)}_0$ is a symplectic resolution.

    \item Let $A= \IC^2/\Lambda$ be a complex torus with period lattice $\Lambda \coloneqq H_1(A, \IZ)$. 
    The group of symplectic automorphisms of $A$ is 
    \[A \rtimes \SL(\Lambda),\]
    where $A$ acts on itself by translation and $\SL(\Lambda) \subset \SL(2,\IC)$ is the group of linear automorphisms of the universal cover $\IC^2$ with determinant~$1$ and preserving the period lattice $\Lambda$. The group of induced symplectic automorphisms of $\K_{n}(A)$ is
    \[A[n+1]\rtimes \SL(\Lambda).\]
    Denote by $\tau_{\alpha}$ the automorphism on \(\K_{n}(A)\) induced by the translation $\alpha \in A[n+1]$.

    \item Given a group $G \subseteq A[n+1] \rtimes \SL(\Lambda)$, we write \[G_{\tr} \coloneqq G \cap A[n+1]\] for the normal subgroup of translations in $G$, and we set
    \begin{equation} \label{equ:G_0}  
    G_{\circ} \coloneqq G/{G_{\tr}} = \Im(\pi\colon G \longhookrightarrow A[n+1] \rtimes \SL(\Lambda) \lra \SL(\Lambda)).
    \end{equation}
\item We use the notation  
    \begin{align*}
    C_n & \hspace{5mm} \text{cyclic group of order } n, \\
    S_n & \hspace{5mm} \text{symmetric group of degree } n, \\
    A_n & \hspace{5mm} \text{alternating group of degree } n, \\
    D_n & \hspace{5mm} \text{dihedral group of degree } n, \\
    Q_8 & \hspace{5mm} \text{quaternion group}, \\
    BD_{12} & \hspace{5mm} \text{binary dihedral group of order }12, \\
    BT_{24} & \hspace{5mm} \text{binary tetrahedral group of order }24.  
    \end{align*}
    \item Let $G$ be a group acting on a normal variety $X$, and let \(q:X\to X/G\) be the quotient map.  For any $x \in X$ and $g \in G$, we write 
    \begin{align*}
    G_x& \coloneqq \{ g \in G \mid  g(x)=x\},  \\
    \Fix (g) & \coloneqq \{x \in X \mid  g(x)=x\}.
    \end{align*}
    The isotropy (group) of a point $z$ in $X/G$ is the stabilizer of a point of the orbit $q^{-1}(z)$, up to conjugation. 

\item Assume that $G$ acts on a smooth complex algebraic variety $X$ of dimension $n$ and fixes a point $x$. Then an analytic (or \'{e}tale) neighborhood of \rev{$q(x)$ in $X/G$} is isomorphic to the linear quotient $T_{x}X/G$ of its tangent space. If $G$ is cyclic of order $k$, then the action of $G$ on $T_{x}X \simeq \IA^n$ can be diagonalized and written as
\[(x_1, \ldots, x_{n}) \longmapsto \left(\xi^{m_1}_{k}x_1, \ldots, \xi^{m_n}_{k}x_n\right),\]
where $\xi_{k}$ is a primitive $\supth{k}$ root of unity and the integers $m_i$ are called weights of the action. We usually abbreviate the quotient by this action as 
\[\IA^n/\tfrac{1}{k}(m_1, \ldots, m_n).\]

\begin{definition} \label{defn:sing type ai}
    Let $X$ be an algebraic variety of dimension $2n$. We denote by $a_k=a_k(X)$ the number of singularities of analytic type
    \[\IA^{2n}/\tfrac{1}{k}(1,-1, \ldots, 1,-1).\]
\end{definition}
\end{itemize}

\section{Symplectic varieties and terminalizations}\label{sec:symplectic}

\subsection{Symplectic varieties}\label{subsection sympectic}
We refer to \cite{KollarMori1998} for the standard terminology in birational geometry. If $X$ is a normal variety and $j \colon X^{\reg} \hookrightarrow X$ is the inclusion of the regular locus, then $\Omega^{[p]}_{X}\coloneqq j_{*}\Omega^{p}_{X^{\reg}}$ is the sheaf of reflexive $p$-forms on~$X$. 
\begin{definition}
  Let $X$ be a normal variety. A reflexive 2-form
  \[
  \omega_X \in H^0\left(X, \Omega^{[2]}_{X}\right)=H^0\left(X^{\reg}, \Omega^{2}_{X^{\reg}}\right)
  \]
  is \emph{symplectic} if its restriction to the regular locus of $X$, denoted by $X^{\reg}$, is closed non-degenerate.
\end{definition}
\begin{definition}
A normal variety $X$ is \emph{symplectic}, or equivalently has \emph{symplectic singularities}, if \begin{itemize}
\item it admits a symplectic form $\omega_X \in H^0(X, \Omega^{[2]}_{X})$, 
\item it has rational singularities.
\end{itemize} 
By \cite[Corollary 1.8]{KS2021}, this means that a holomorphic
symplectic form $\omega_{X^{\reg}}$ on $X^{\reg}$ extends to a (possibly degenerate) holomorphic 2-form $\omega_{\widetilde{X}}$ on a resolution $\widetilde{X} \to X$. We say that $X$ admits a \emph{symplectic resolution} if $\omega_{\widetilde{X}}$ is non-degenerate.
\end{definition}

\begin{proposition}\label{prop:termcanonical}
  A symplectic variety has Gorenstein canonical singularities, and it is terminal if and only if the singular locus has codimension at least $4$.
\end{proposition}
\begin{proof}
See \textit{e.g.}~\cite[Claim 2.3.1]{Kol13} and \cite[Corollary 1]{Namikawa2001}.
\end{proof}

\begin{definition}
Let $X$ be a variety with canonical singularities. A \emph{terminalization} of $X$ is a \new{proper} 
birational morphism $f \colon Y \to X$ such that $Y$ has terminal singularities and $f^*K_X = K_Y$. 
\end{definition}
\noindent A terminalization of $X$ exists by \cite[Corollary 1.4.3]{BCHM}, and it can be chosen projective, $\IQ$-factorial (see \cref{defn:Q-factorial})  and equivariant with respect to a group $G$ acting on $X$. Further, it is also unique up to isomorphism in codimension~$1$; see \cite[Corollary 3.54]{KollarMori1998}. 

\subsection{Terminalization of symplectic varieties}\label{sec:terminalization}
\begin{proposition}\label{thm:semismall}
Let $f\colon Y \to X$ be a crepant birational modification of a symplectic variety $X$, \textit{e.g.}~a terminalization of\, $X$. Then $f$ is semismall; \textit{i.e.},  $\dim (Y \times_{X} Y)  \leq \dim X$. 
\end{proposition}
\begin{proof} See \textit{e.g.}~\cite[Lemma 2.11]{Kal06}, \cite[Proposition 2.14]{Losev2022} and \cite[Proposition 2.15]{Tighe2022}.
\end{proof}

\begin{definition}\label{def:dissident}
    Let $X$ be a symplectic variety. Let $X^{\circ}$ be the largest open set of points $x$ in $X$ such that either $X$ is smooth at $x$, or the formal completion $\widehat{X}_{x}$ admits a decomposition $\widehat{M}_{x}\times \widehat{W}_{x}$, where $M_{x}$ is a smooth scheme and $W$ is a canonical surface singularity. The complement $X \setminus X^{\circ}$ is called the \emph{dissident locus}.
\end{definition}

In other words, $X^{\circ}$ is the union of the strata of dimension $\dim{X}$ and $\dim{X}-2$ of the stratification of~$X$ constructed in \cite[Theorem~2.3]{Kal06}. Note that $X^{\circ}$ admits a unique symplectic resolution $f^{\circ} \colon Y^{\circ} \to X^{\circ}$ obtained by repeatedly blowing up the singular locus of $X^{\circ}$ (or of its blowup), as in the surface case.

\begin{proposition}\label{prop:terminalization}
Let $f \colon Y \to X$ be a proper birational modification of a symplectic variety $X$. Then $f$ is a terminalization of\, $X$ if and only if $f$ is a normal compactification of $f^{\circ} \colon Y^{\circ} \to X^{\circ}$ and $\codim\left(Y\setminus Y^{\circ}\right)\geq2$.
\end{proposition}
\begin{proof}
Suppose that $f$ is a terminalization of $X$. By the uniqueness of minimal surface resolution, $f$ must restrict to $f^{\circ}$ over $X^{\circ}$. Now, if $f$ extracts a divisor $E \subseteq Y\setminus Y^{\circ}$, then 
\begin{align*}
\dim\left(E\times_{X}E\right) &= 2 \dim E - \dim f(E) \geq 2 \dim X - 2 - \dim (X \setminus X^{\circ}) \\
&\geq  \dim X -2 +4 > \dim X,
\end{align*}
which contradicts \cref{thm:semismall}. Hence, $\codim\left(Y\setminus Y^{\circ}\right)\geq2$. 

Conversely, a compactification of $f^{\circ} \colon Y^{\circ} \to X^{\circ}$ such that ${\codim\left(Y\setminus Y^{\circ}\right)\geq2}$ is isomorphic in codimension~$1$ to a terminalization of $X$, so it is terminal.
\end{proof}

\begin{definition}\label{defn:Q-factorial}
A normal algebraic or analytic variety $X$ is \emph{$\IQ$-factorial} if for every Weil divisor $D$ on~$X$, there is an $m \in \IN$ such that $mD$ is Cartier. A normal complex analytic variety $X$ is \emph{locally analytically $\IQ$-factorial}  if
every open set $U \subseteq X$ in the Euclidean topology is $\IQ$-factorial.

Let $f \colon Y \to X$ be a proper morphism 
of normal complex varieties; then 
$Y$ is \emph{locally analytically $\IQ$-factorial over $X$} if $Y_{U} \coloneqq f^{-1}(U)$ is $\IQ$-factorial for any open set $U \subseteq X$ in the Euclidean topology.
\end{definition}
\begin{lemma}\label{lem:Qfact}
Let $f \colon Y \to X$ be a proper birational morphism of normal complex algebraic or analytic varieties with exceptional divisor $E= \sum_i E_i$. Suppose that
\begin{itemize}
\item[$(\dagger)$]  the divisors $E_i|_{f^{-1}(U)}$ are irreducible for any connected open set $U \subseteq X$ in the Euclidean topology.
\end{itemize} If\, $X$ is locally analytically $\IQ$-factorial and $Y$ is $\IQ$-factorial, then $Y$ is locally analytically $\IQ$-factorial.  
\end{lemma}

\begin{proof}
By assumption $(\dagger)$, any prime exceptional divisor of $f|_{Y_U} \colon Y_U \coloneqq f^{-1}(U) \to U$ is $E_i|_{Y_U}$ \new{for some $i$.} Then
by the localization formula, we have
\[\bigoplus_{i}\IQ \,E_{i}|_{Y_U} \lra \Cl(Y_U)_{\IQ} \lra \Cl(Y_{U} \setminus E)_{\IQ} \lra 0.\] 
Since $X$ is locally analytically $\IQ$-factorial, then $\Cl(Y_{U} \setminus E)_{\IQ}\simeq \Cl(U \setminus f(E))_{\IQ} \simeq \Cl(U)_{\IQ} \simeq \Pic(U)_{\IQ}$. Since $Y$ is $\IQ$-factorial, a multiple of $E_i|_{Y_U}$ is Cartier. We conclude that $\Cl(Y_U)_{\IQ}$ is generated by Cartier divisors; \textit{i.e.}, $Y$ is locally analytically $\IQ$-factorial over $X$.\end{proof}

\begin{proposition}\label{prop:singulartermina}
Let $f \colon Y \to X$ be a projective $\IQ$-factorial terminalization of a complex symplectic variety $X$ with exceptional divisor $E= \sum_i E_i$. Suppose that 
\begin{enumerate}
\item \label{item:locQfact} $X$ is locally analytically $\IQ$-factorial, 
\item \label{item:locGmaction} the formal completion $\hat{X}_x$ of\, $X$ at any singular point $x\in X$ admits a $\mathbb{G}_{m}$-action with only positive weights on the maximal ideal of\, $\mathcal{O}_{\hat{X}_x}$ and on the local symplectic form, 
\item the divisors $E_i|_{f^{-1}(U)}$ are irreducible for any connected open set $U \subseteq X$ in the Euclidean topology.
\end{enumerate} 
Assume further that $U$ is a connected Euclidean neighborhood of some $x \in X$, and let $T \to U$ be a projective terminalization of\, $U$ which is locally analytically $\IQ$-factorial over $U$. Then, up to  shrinking $U$ to a smaller neighborhood of $x$ if necessary, $Y_{U}\coloneqq f^{-1}(U)$ admits a locally trivial deformation to $T$. In particular, $Y_U$ and $T$ have the same singularities; \textit{i.e.}, for any $t \in T$,  there exist a $y \in Y_{U}$ and a formal isomorphism $\hat{T}_t \simeq \hat{Y}_{U, y}$.
\end{proposition}
\begin{proof}
  We closely follow \cite{Namikawa2008}.  For any $x \in X$, there exists a pointed affine symplectic scheme $(Z,z)$,
  \begin{enumerate*}[(i)]\item with a good $\mathbb{G}_m$-action fixing $z$, \item algebraizing $\hat{X}_x$, \textit{i.e.}, $\hat{X}_x \simeq \hat{Z}_z$, and \item only depending on $\hat{X}_x$ and the weights of the $\mathbb{G}_m$-action; see \cite[Lemma A.2 and the proof of Lemma 22]{Namikawa2008}.\end{enumerate*} The local $\mathbb{G}_m$-action on $\hat{X}_x$ lifts to $\hat{Y}_x \coloneqq \hat{X}_x \times_X Y$ and linearizes a $f|_{\hat{Y}_x}$-ample line bundle; see \cite[Steps 1 and 2 of Proposition~A.7, Lemma A.8]{Namikawa2008}. 

Now, by \cite[Proposition A.5]{Namikawa2008}, there exists a $\mathbb{G}_m$-equivariant projective morphism ${g \colon W\!=\!W(Y_U) \to Z}$ such that $\hat{Y}_x \simeq \hat{W}_z \coloneqq \hat{Z}_z \times_Z W$. By Artin's approximation \cite[Corollary 2.4]{Artin1969}, there exists an analytic open neighborhood $z \in V \subset Z$ such that, up to shrinking $U$, the following diagram commutes: 
\[
\xymatrix{
Y \ar[d]_{f} & Y_U \ar[r]^{\simeq} \ar[d]_{f_U} \ar@{_{(}->}[l] & W_V \ar[d]^{g_V} \ar@{^{(}->}[r] & W \ar[d]^{g} \\
X & U \ar[r]^{\simeq} \ar@{_{(}->}[l] & V \ar@{^{(}->}[r] & Z\rlap{.}
}
\]
Since $Y_{U} $ is a $\IQ$-factorial terminalization of $U$ by \cref{lem:Qfact}, and since the $\mathbb{G}_m$-action retracts $W$ into $W_{V}$, we conclude that $W_{V}$ and $W$ are terminal and $\IQ$-factorial too.
Applying the same construction to $T \to U$, we obtain two projective $\IQ$-factorial terminalizations of $Z$
\[W(Y_U) \overset{g}\lra Z \overset{\;\;g'}\longleftarrow W(T).\]
The result then follows from \cite[Corollary 25]{Namikawa2008}.
\end{proof}

\begin{corollary}\label{prop:Qfactorialterminalization}
Let $f \colon Y \to X$ be a projective $\IQ$-factorial terminalization of a complex symplectic variety $X$ with only quotient singularities. Suppose that the divisor $E|_{f^{-1}(U)}$ is irreducible for any prime exceptional divisor $E \subset Y$ and for any connected open set $U \subseteq X$ in the Euclidean topology. \revi{Assume further that $U$ is a connected Euclidean neighborhood of some $x \in X$, and let $T \to U$ be a projective terminalization of\, $U$. Then, 
up to shrinking $U$ to a smaller neighborhood of $x$, $Y_{U}\coloneqq f^{-1}(U)$ admits a locally trivial deformation to $T$. Furthermore, $Y_U$ and $T$ are locally analytically $\IQ$-factorial over $U$, and they have the same singularities in the sense of \cref{prop:singulartermina}. }
\end{corollary}
\begin{proof}
If $X$ has quotient singularities, then assumptions \eqref{item:locQfact} and \eqref{item:locGmaction} in \cref{prop:singulartermina} hold. \revi{The local analytic $\IQ$-factoriality of $Y$ and $T$ follows from \cref{lem:Qfact}.}
\end{proof} 

\subsection{Irreducible symplectic varieties}
\begin{definition}
A symplectic compact K\"{a}hler\footnote{We refer to \cite[Section~3.3, p. 346]{G1962} or \cite[Section~2.3]{BL2022} for the notion of (possibly singular) compact K\"{a}hler space.} variety $(X, \omega_X)$ is an \emph{irreducible $($holomorphic$)$ symplectic variety} (IHS variety for short) if for any finite quasi-\'{e}tale cover $g \colon X' \to X$, the exterior algebra of reflexive forms $H^0(X',\Omega^{[\bullet]}_{X'})$ on $X'$ is generated by the reflexive pullback $g^*\omega_X$.
\end{definition}

\begin{remark}[Finiteness results for IHS varieties]\label{rem:finitenessIHS}
    It is expected that the number of deformation types of irreducible symplectic varieties is finite in each dimension. For instance, in any given dimension, there are only finitely many diffeomorphism types of irreducible symplectic manifolds with isomorphic Beauville--Bogomolov lattice $(H^2, q)$; see \cite[Theorem 4.3]{Huy2003} and the refinement \cite[Theorem 4.4]{Kamenova2018}. Topological bounds are known in low dimension: The second Betti number of a $4$-dimensional irreducible symplectic orbifold is at most $23$ by \cite{guan.betti, Fu-Menet}, and in the smooth fourfold case is either at most $8$ or $23$ by \cite{guan.betti}, and conjecturally only $5,6,7$ or $23$ according to \cite[Corollary 1.3]{BS2022} (\textit{cf.} \cite[Theorem 9.3]{DHMV2022}); see also the recent survey \cite{BD2022}. It is expected that the same bound of at most $23$ holds for smooth sixfolds too; see \cite{Kurnosov2016, Sawon2022} for partial results. A conjectural bound for the second Betti numbers of irreducible symplectic manifolds in arbitrary dimension is proposed in \cite{KimLaza2020}.
\end{remark}

\begin{proposition}\label{prop:birationaldeformation}
Birational \new{projective} $\IQ$-factorial terminal IHS varieties are deformation equivalent. In particular, a projective $\IQ$-factorial terminalization of an IHS variety is unique up to deformation.
\end{proposition}
\proof
See \cite[Corollary 6.17]{BL2022}.
\endproof

\begin{definition}
An automorphism \(\varphi \colon X \to X\) of a symplectic variety $(X, \omega_X)$ is  \textit{symplectic} if \(\varphi^*\omega_X=\omega_X\).
\end{definition}

\begin{lemma} 
Let $G$ be a symplectic group acting on a symplectic variety $X$. Any irreducible components of the fixed locus of\, $G$ has even dimension.
\end{lemma}
\begin{proof}
   If $X$ is smooth, the fixed locus $\Fix(G)$ is smooth and  symplectic as its tangent bundle is the $G$-fixed part of the tangent bundle of $X$, thus symplectic and even-dimensional: 
   \[T_{\Fix(G)}=\left(\left(T_X\right)|_{\Fix(G)}\right)^G.\]
   In general, we stratify $X$ into smooth $G$-invariant locally closed symplectic subsets as in \cite[Theorem~2.3]{Kal06} and reduce to the argument in the smooth case.
\end{proof}

\begin{proposition}\label{prop: terminaliz and quotients of IHS}
The sets of symplectic varieties and of IHS varieties are both closed under \new{projective} terminalizations or finite quotients by symplectic groups. 
\end{proposition}
\proof
The terminalization of a symplectic variety is symplectic by construction. Let $q\colon X \to X/G$ be a symplectic quotient of a symplectic variety. Any $G$-invariant symplectic form descends to $X/G$, and $X/G$ has rational singularities by \cite[Proposition 5.13]{KollarMori1998}. Hence, $X/G$ is symplectic.

Now suppose  that $X$ is an IHS variety. The statement for \new{projective} terminalization is proved in \cite[Proposition 12]{Schwald2020}.\footnote{The notion of primitive symplectic in \textit{loc.\ cit.}~stands for IHS varieties.} The statement for symplectic quotients $q\colon (X, \omega_X) \to X/G$ is proved in \cite[Lemma~2.2]{Matsushita2015} when $X$ is smooth. The argument in the singular case is essentially identical. We give a 
self-contained proof for completeness. Let $g \colon X' \to X/G$ be a quasi-\'{e}tale cover of $X/G$ and $Z$ be the normalization of an irreducible component of $X \times_{X/G} X'$ that maps surjectively onto $X$ and $X'$. All maps in the commutative square
\[
\begin{tikzcd}[scale=1]
   Z \arrow[d, "\tilde{g}"'] \arrow[r, "q'"] & X' \arrow[d, "g"]\\
   X \arrow[r, "q"'] & X/G
\end{tikzcd}
\]
are quasi-\'{e}tale. Since $X$ is an IHS variety, the algebra $H^0(Z, \Omega^{[\bullet]}_{Z})$ is generated by the pullback $\omega_{Z} \coloneqq \tilde{g}^*(\omega_X)$. Since $q'$ is quasi-\'{e}tale and $\tilde{g}^*(\omega_X)$ descends to a symplectic form $\omega_{X'}$ on $X'$, the inequalities
\[\dim \langle \omega_{Z} \rangle =  \dim H^0\left({Z}, \Omega^{[\bullet]}_{{Z}}\right) \geq \dim H^0\left(X', \Omega^{[\bullet]}_{X'}\right) \geq \dim \langle \omega_{X'} \rangle\]
are identities, and so $H^0(X', \Omega^{[\bullet]}_{X'})$ is generated by $\omega_{X'}$. Hence, $X/G$ is an IHS variety.
\endproof

\begin{corollary}
Let $G$ be a finite symplectic group acting on a symplectic variety $X$ or an IHS variety. A \new{projective} terminalization of the quotient $X/G$ is symplectic or an IHS variety, respectively.
\end{corollary}

\begin{remark}\label{rmk:finitesymplgroup}
   \revi{ Finite groups of symplectic automorphisms of known irreducible symplectic manifolds have been extensively studied in the literature. Particularly relevant for the present paper are the classifications of finite symplectic groups acting on $\K3^{[2]}$ by H\"{o}hn and
Mason in \cite{HM} (see also the preliminary results in \cite{Mon2013}), and on $\K_2(A)$ by Mongardi, Tari and Wandel in \cite[Section~5]{Mon-Tar-Wan}. See also \cite{Mongardi2016, KMO2023}.}

In view of Assumption~\ref{assumption:codim2}, O'Grady examples instead are less interesting for our classification purposes. Indeed, the only symplectic automorphism of an irreducible symplectic manifold of O'Grady 10 type is the identity by \cite{Giovenzana.Grossi.Onorati.Veniani.Symplecticrigidity}. On the other hand, all symplectic automorphisms of an irreducible symplectic manifold of O'Grady 6 type act trivially on the second cohomology group by \cite{Grossi.Onorati.Veniani:sp.bir.transf.OG6}, 
and their fixed loci have codimension at least $4$ by \cite[Section 6]{Mongardi.Wandel:OG.trivial.action}, so they do not satisfy Assumption~\ref{assumption:codim2}.
\end{remark}

\section{Remarks on the criteria of the classification}\label{sec:criteria}
Given an irreducible symplectic variety $X$, it is possible to obtain new such varieties by taking quotients, birational contractions or terminalizations. When one classifies birational modifications of symplectic quotients, in order to avoid redundancy, it is convenient to restrict to projective $\bQ$-factorial terminalizations~$Y$ of symplectic quotients $X/G$ by a finite group $G$, with simply connected regular locus $Y^{\reg}$. If $X$ has $\mathbb{Q}$-factorial singularities, this amounts to Assumptions~\ref{assumption:codim2} and~\ref{ass:groupaction}.
 
Symplectic quotients $X/G$ of irreducible symplectic varieties $X$ are irreducible symplectic varieties themselves; see \cref{prop: terminaliz and quotients of IHS}. However, they are less interesting from a classification viewpoint as the geometry of $X/G$ can be essentially recovered from the $G$-equivariant geometry of $X$. For instance, we have the following: 
\begin{itemize}
    \item The rational cohomology of $X/G$ is isomorphic to the $G$-invariant part of the rational cohomology of~$X$:
    \[H^*(X/G, \IQ)=H^*(X, \IQ)^G.\]
    Note, however, that the relation between the integral cohomology $H^*(X/G, \IZ)$ and $H^*(X, \IZ)$ is more subtle, and even for a single involution, their connection is not trivial; see for instance \cite{Kapfer-Menet}.
    \item The fundamental group of the regular locus of $X/G$ is an extension \revi{of $G$ by $\pi_1(X^{\reg})$:}
    \[1\lra\pi_1(X^{\reg}) \lra \pi_1((X/G)^{\reg}) \lra G\lra 1.\]
    In particular, if $X^{\reg}$ is simply connected, then $\pi_1((X/G)^{\reg}) \simeq G$ 
    \item The deformations of $X/G$ are the deformations of $X$ preserving the group action.
\end{itemize}  
\medskip

More conceptually, the building blocks of the Beauville--Bogomolov decomposition are defined only up to quasi-\'{e}tale cover. Nonetheless, for the symplectic factors of the decomposition, one can actually choose a distinguished representative in the class of all quasi-\'{e}tale covers of a fixed symplectic factor, namely its universal quasi-\'{e}tale cover. In other words, the irreducible symplectic factors $Y$ in the Beauville--Bogomolov decomposition can always be chosen so that the regular locus $Y^{\reg}$ is algebraically simply connected, \textit{i.e.}, the algebraic fundamental group $\hat{\pi}_1(Y^{\reg})$ of the regular locus is trivial. This is indeed possible since the algebraic fundamental group of an irreducible symplectic variety is known to be finite by \cite[Corollary~13.2]{GrebGuenanciaKebekus}.\footnote{Actually, the same is expected to hold for the topological fundamental group: This is implicit 
in \cite[Section~13]{GrebGuenanciaKebekus}, explicitly conjectured in \cite[Conjecture 3]{Wang2022} and proved by Engel, Filipazzi, Greer, Mauri and Svaldi in \cite{FMS} under the assumption that $Y$ admits a Lagrangian fibration.} The conclusion is that we should only classify irreducible symplectic varieties $Y$ with $\hat{\pi}_1(Y^{\reg})=1$ (conjecturally, $\pi_1(Y^{\reg})=1$) as all other irreducible symplectic varieties are quasi-\'{e}tale quotients of them. 
\medskip

Birational transformations of $X/G$ are also potential new sources of irreducible symplectic varieties. However, to preserve the non-degeneracy of the symplectic form, one is only allowed to extract divisors with discrepancy zero. Any such birational modification is dominated by a $\bQ$-factorial terminalization $Y$, see \cite[Corollary 1.4.3]{BCHM}, and moreover it can be recovered by the Mori theory of $Y$ itself. Therefore, it is superfluous to study symplectic birational modifications of $X/G$ other than its $\bQ$-factorial terminalizations. Actually, it has always been clear in the literature that $\bQ$-factorial and terminal irreducible symplectic varieties form a particularly agreeable class of varieties for their well-behaved birational and deformation theory: 
\begin{itemize}
\item Birational projective $\bQ$-factorial terminal irreducible symplectic varieties are deformation equivalent. In particular, any two projective $\bQ$-factorial terminalizations of the same irreducible symplectic variety are deformation equivalent. See \cite[Corollary 6.17]{BL2022}.
\item Deformations of projective $\bQ$-factorial terminal irreducible symplectic varieties are equisingular. In particular, the Betti numbers and the fundamental group of the regular locus are deformation invariants. See \cite{Namikawa2006}.
\item The global Torelli theorem holds for $\bQ$-factorial K\"{a}hler terminal irreducible symplectic varieties with $b_2 \geq 5$. See \cite[Theorem 1.1]{BL2022}.\footnote{It is expected that the assumption on the Betti number can be removed, and this is known if the irreducible symplectic variety has only quotient singularities; see \cite{Menet2020}.}
\end{itemize} 
\medskip

Also note  that if Assumption~\ref{assumption:codim2} holds, then \revi{the cohomology class of each exceptional divisor of a terminalization $Y \to X/G$ in $H^{2}(Y, \mathbb{Z})$ remains of type $(1,1)$ only along a divisor in the Kuranishi family of $Y$.} This implies that the general deformation of $Y$ cannot come from the quotient-terminalization construction and should be considered as an honestly new deformation type of irreducible symplectic variety. 
\medskip

Finally, observe that the projectivity \new{condition can always be achieved by \cite[Corollary 1.4.3]{BCHM}. The projectivity of terminalizations obtained by gluing local terminalization is discussed in Sections~\ref{sec:terminalization} and~\ref{remark:projectiveterminalization}.} 

\section{Induced symplectic automorphisms and terminalizations} \label{sec: strategy}

In this section, we show that terminalizations of quotients of $S^{[n]}$ or $K_n(A)$ by induced symplectic automorphism groups can be obtained via a single explicit blowup \new{away from the dissident locus;} see \cref{prop:blowup}.

\subsection{Induced automorphism}
Let $X$ be either a Hilbert scheme $S^{[n]}$ of a K3 surface $S$, or a generalized Kummer variety $\K_n(A)$ associated to an abelian surface $A$.

\begin{definition}
An automorphism $\phi \colon S \to S$ of a K3 surface $S$ induces an automorphism of $S^{[n]}$. We call such an automorphism \new{of $S^{[n]}$} \emph{induced}.
\end{definition}

\begin{definition}
An automorphism $\phi \colon A \to A$ of the abelian surface $A$ (not necessarily fixing the origin) induces an automorphism $\phi^{(n+1)} \colon A^{(n+1)} \to A^{(n+1)}$ of its symmetric product $A^{(n+1)}$. If $\phi^{(n+1)}$ preserves $A^{(n+1)}_0$, then it lifts to an automorphism of $\K_{n}(A)$. We call such an automorphism \new{of $\K_{n}(A)$} \emph{induced}.
\end{definition}

Note that an induced automorphism on $S^{[n]}$ or $\K_n(A)$ is symplectic if and only if the underlying automorphism of $S$ or $A$ is symplectic.

\subsection{Codimension 2 fixed loci of induced authomorphisms} \label{sec:codim2 fixed loci}

Let $G$ be a finite group of induced symplectic automorphisms of $X=\K_{n}(A)$ or $S^{[n]}$. In this section, we describe the connected components of codimension~$2$ fixed by automorphisms $g \in G$. The importance of these loci lies in the fact that their images in $X/G$ are the centers of blowups giving a terminalization $Y \to X/G$. Their geometry is severely constrained: We show that they occur only if the orders of $g$ and $n$ are either $2$ or $3$. Further, these fixed components are all of $\K3$ or $\K3^{[2]}$ type; see also \cite[Theorems~1.0.2 and 1.0.4]{KMO2023}.

Denote the Hilbert--Chow morphism by $\epsilon \colon S^{[n]}\to S^{(n)}$ or $\epsilon \colon \K_n(A) \to A^{(n+1)}_0$ as in Section~\ref{sec:notation}.

\begin{lemma} \label{lem:morphism generically finite}
Let \(G\) be a finite group of symplectic automorphisms of\, \(X\). Let $F \subset X$ be a subvariety of codimension~$2$ fixed by an element of the group. 
Then the restriction $\epsilon\vert_{F}$ is generically finite.
\end{lemma}

\proof
If $F \not\subseteq \Exc(\epsilon)$, then $\epsilon_{|F}$ is birational, so generically finite. Then suppose  $F \subseteq \Exc(\epsilon)$. If $F= \epsilon^{-1}(\epsilon(F))$, then $F$ is uniruled, which is impossible as $F$ is the fixed locus of a symplectic automorphism and hence~$F$ is symplectic so not uniruled. It follows that $F \subsetneq \epsilon^{-1}(\epsilon(F))$ and $\dim( \epsilon^{-1}(z) \cap F) < \dim ( \epsilon^{-1}(z))$  for a general $z \in \epsilon(F)$. By \cite[Lemma 2.11]{Kal06} or \cref{thm:semismall}, the morphism $\epsilon$ is semismall; \textit{i.e.}, $\dim(X)= \dim\left( X\times_{\epsilon(X)} X\right) $.  We get
\begin{align*}
2n
& = \dim\left( X\times_{\epsilon(X)} X\right) \\
& \geq \dim\left(\epsilon^{-1}(\epsilon(F)) \times_{\epsilon(F)} \epsilon^{-1}(\epsilon(F))\right) \geq \dim (\epsilon(F)) + 2 \dim\left(\epsilon^{-1}(z)\right)\\
& > \dim(\epsilon(F)) + 2 \dim \left(\epsilon^{-1}(z) \cap F \right)\\
& = \dim(F) + \dim\left(\epsilon^{-1}(z)\cap F\right) = 2n-2 + \dim \left(\epsilon^{-1}(z)\cap F \right),
\end{align*} 
 and so $ \dim\left(\epsilon^{-1}(z)\cap F\right) \leq 1$. If $\dim\left(\epsilon^{-1}(z)\cap F\right) =0$, then $\epsilon_{|F}$ is generically finite. Otherwise, $\dim(\epsilon(F))$ is odd, which is impossible since $\epsilon(F)$ is symplectic as fixed locus of a symplectic automorphism of $\epsilon(X)$.
\endproof

\begin{lemma}[Order of automorphisms with large fixed locus]\label{lemma:codime2}
    Let $F$ be a subvariety of codimension~$2$ fixed by an induced symplectic automorphism $g \colon X \to X$. Then
\begin{itemize}
    \item $\ord(g)=2$ and $X=S^{[2]}$, $\K_2(A)$ or $\K_{3}(A)$, or
    \item $\ord(g)=3$ and $X=\K_2(A)$.
\end{itemize}
\end{lemma}

\begin{proof}
Let $G$ be the cyclic group $\langle g\rangle$ acting indifferently on the surface $M \coloneqq S$ or $A$, on the singular symplectic variety $X_{\sing} \coloneqq S^{(n)}$ or $A^{(n)}_0$, or on the symplectic resolution $X\coloneqq S^{[n]}$ or $\K_{n-1}(A)$. Stratify the surface $M$ according to length of $G$-orbits; \textit{i.e.}, 
\[M= \bigsqcup^{\ord(g)}_{i=1} M_i \quad \text{with} \quad M_i \coloneqq \{ m \in M \colon |G m|=i\}.\]
The locus $M_{\ord(g)}$ is open and dense, while $M_i$ with $i < \ord(g)$ consists of at most finitely many points as $g$ is symplectic.
A $g$-fixed point $z$ of $X_{\sing}$ is a union of orbits for the action of $G$ on $M$; \textit{i.e.}, $z=\{G m_1, \ldots, Gm_r\}$ for some $m_j \in M$ such that $\sum^r_{j=1}|G m_j|=n$.

For any partition $\underline{n} =\sum^{\ord(g)}_{i=1} i \cdot n_i$ of $n$, define $F_{\underline{n}} \subseteq X_{\sing}$ as the union of the $g$-fixed subvarieties whose general points $z=\{G m_1, \ldots, Gm_r\}$ satisfy $|\{ j \in [r] \colon |Gm_j|=i\}|=n_i$.  The finite morphism
  \begin{align*}
    \prod^{\ord(g)}_{i=1} M^{n_i}_i \lra M^{(n)}, \quad
(m_i) \longmapsto \{G m_i\},
\end{align*}
contains $F_{\underline{n}}$ in its image, and so 
\[\dim F_{\underline{n}} \leq  \sum^{\ord(g)}_{i=1} n_i\dim M_{i} =  n_{\ord(g)} \dim M_{\ord(g)} = 2 n_{\ord(g)}\]
since $\dim M_{i}=0$ for $i \neq \ord(g)$ and $\dim M_{\ord(g)}=\dim M = 2$.
In particular, we obtain
\[
\codim_{X_{\sing}}\left(F_{\underline{n}}\right) \geq 2n - 2 n_{\ord(g)}-2 \epsilon,
\]
where $\epsilon=0$ or $1$ if $M=S$ or $A$, \rev{respectively}. Finally, the equation $\codim_{X_{\sing}}(F_{\underline n})=2$ admits a solution only in the cases listed in \cref{lemma:codime2}. By \cref{lem:morphism generically finite}, the solutions for $X_{\sing}$ are solutions for the symplectic resolution $X$ too. 
\end{proof}

\begin{remark}\label{rem fix locus involutions on Hilb} 
Note that on \(X=S^{[2]}\), any induced involution $g$ fixes a locus of codimension~$2$, namely  the strict transform of $\{[(x,g(x))]| x\in S\}$ in $S^{(2)}$. Thus, for \(X=S^{[2]}\) the condition on $g$ in \cref{lemma:codime2} is actually necessary and sufficient. See also \cite[Theorems 1.1 and 1.2]{KMO2022}. 
\end{remark}

In the Kummer case, the automorphisms fixing a locus of codimension $2$ (see \cref{lemma:codime2} above) admit a particularly explicit description.

\begin{lemma}[Induced involutions and automorphisms of order $3$ on $\K_n(A)$] \label{lemma:inducedinvolution} \leavevmode
\begin{enumerate}
    \item\label{l:i-1} An induced symplectic involution $g$ of\, \(\K_{n}(A)\), with $n=2$ or $3$, fixes a subvariety $F$ of codimension~$2$ if and only if  \[g = \tau_\alpha (-\id) \in A \rtimes \SL(\Lambda)\] with
    $\alpha \in A[3]$ if $n=2$, or $\alpha \in A[2]$ if $n=3$. 
    \item\label{l:i-2} An induced symplectic automorphism $g$ of order $3$ 
      of\, \(\K_{2}(A)\) fixes a subvariety $F$ of codimension~$2$ if and only \[g = \tau_\alpha \circ T \in A \rtimes \SL(\Lambda)\] with $T^3=1$, $T \neq \id$ and $T(\alpha)=\alpha$ $($equivalently,  $T$ is a linear automorphism of order $3$ commuting with $\tau_{\alpha})$.
\end{enumerate}
\end{lemma}
\begin{proof}
Since $-\id$ is the only involution of $\SL(2, \IC)$, any induced involution $g$ of $\K_2(A)$ is of the form \[\tau_\alpha (-\id) \in A[3] \rtimes (-\id),\] and it fixes the strict transform of the surface $\{[(x,g(x),-x-g(x))] \mid x \in A \} \subset A^{(3)}_0$; see also \cite[Theorem~7.5]{Kapfer-Menet}.

If $n=3$, induced involutions of $\K_3(A)$  are of the form either
\[\tau_\alpha (-\id) \in A[4] \rtimes (-\id) \quad \text{or} \quad \tau_{\alpha} \in A[2],\] 
but the only involutions $g$ that fix a fourfold in $A^{(4)}_0$, namely $\{[(x,g(x),y, g(y))] \mid x, y\in A \}$, are those of the form $\tau_\alpha (-\id) \in A[2] \rtimes (-\id)$.

An order $3$ automorphism $g$ fixes a surface $F$ in $\K_2(A)$ if and only if it fixes the surface
\[
\epsilon(F) = \left\{\left[\left(x,g(x),g^2(x)\right)\right]\mid x \in A\right\}
  \]
  in $A^{(3)}_0$ by \cref{lem:morphism generically finite}. This occurs if and only if $g$ satisfies $1+g+g^2=0$, \textit{i.e.}, if and only if $T \in \SL(\Lambda)$ has minimal polynomial $1+t+t^2$, \textit{i.e.},  $T^3=\id$,  $T \neq \id$ and $T(\alpha)=\alpha$.
\end{proof}

\begin{remark}\label{rmk:codim2unique}
\cref{lemma:codime2,lemma:inducedinvolution} imply that an induced symplectic automorphism of $S^{[n]}$ or $\K_n(A)$ fixes at most one subvariety of codimension~$2$.
When it exists, such a  subvariety $F$ is a crepant resolution of $\epsilon(F)$ and is isomorphic to a K3 surface or a Hilbert square of a K3 surface. 
\end{remark}

\renewcommand*{\arraystretch}{1.1}
\begin{longtable}{c | c | c | l |l}
\caption{Codimension $2$ subvarieties $F$ fixed by an induced symplectic automorphism $g$ of $X$. We denote by $S_2$ and $S_3$ the minimal resolutions of $A/g$.}
\label{table:subvarieties F}\\
 $\ord(g)$ & $X$ & $F$ & $\epsilon(F)$ & $g$
 \\\hline
 $2$ & $S^{[2]}$ & $S$ & $[(x, g(x))]$ & any involution\\
 $2$ & $\K_2(A)$ & $S_2$ & $[(x, -x +\alpha, -\alpha)]$  & $\tau_{\alpha}(-\id)$\text{ with }$\alpha \in A[3]$\\
 $3$ & $\K_2(A)$ & $S_3$ & $[(x, T(x)+\alpha, T^2(x)-\alpha)]$ & $\tau_{\alpha} \circ T$ \text{ with }$T^3=1$, $T \neq \id$, $T(\alpha)=\alpha$\\
 $2$ & $\K_3(A)$ & $S_2^{[2]}$ & $[(x, -x+\alpha, y, -y+\alpha)]$ & $\tau_{\alpha}(-\id) \text{ with }\alpha \in A[2]$ 
\end{longtable}

\subsection{Terminalizations via explicit blowups}\label{sec:termexplicit}
\begin{notation}\label{Notation:gen}
    Let $G$ be a finite group of induced symplectic automorphisms of \(X=S^{[n]}\) or $\K_n(A)$.  Let $q \colon X \to X/G$ be the quotient map, $p \colon Y \to X/G$ be a terminalization of $X/G$, and $\Sigma$ be the singular locus of~$X/G$. Denote by $F_g \subset X$  the (unique by \cref{rmk:codim2unique}) component of the fixed locus of $g \in G$ of codimension~$2$, if any: 
    \[\begin{tikzcd}
	& {X \supset F_g \hspace{100pt}} \\
	Y & {X/G \supset \Sigma\coloneqq \Sing(X/G) \supseteq q(F_{g})}.
	\arrow["p", from=2-1, to=2-2]
	\arrow["q", shift right=24, from=1-2, to=2-2]
\end{tikzcd}\]
\end{notation}

\cref{prop:blowup} is a refinement of \cref{prop:terminalization} in our special context. It asserts that in order to obtain a terminalization of $X/G$ \new{away from the dissident locus,} it suffices to blow up once the irreducible components of the singular locus of codimension~$2$---no need to repeat the process---in the same way as a single blowup suffices to resolve the surface singularities of type $A_1$ and $A_2$. 

\begin{corollary}\label{prop:blowup}
  We use  the notation of \cref{def:dissident}. Away from the dissident locus, the terminalization $Y$ is isomorphic to the blowup of the reduced singular locus of\, \revi{X/G}; \textit{i.e.},
  \[Y^{\circ} \simeq \Bl_{\Sigma \cap (\revi{X/G})^{\circ}} (\revi{X/G})^{\circ}.\]
\end{corollary}

\begin{proof}
    By \cref{prop:termcanonical}, the quotient $X/G$ is not terminal if and only if the fixed locus of some element of~$G$ has a component of codimension~$2$. This occurs only if $\ord(g)=2$ or $3$, and in the precise cases detailed in \cref{lemma:codime2}. Geometrically, this implies that a normal slice to a general point in $q(F_g)$ is a canonical surface singularity of type $A_1$ or $A_2$, which can be resolved with a single blowup. We conclude by applying \cref{prop:terminalization}. 
\end{proof}

\begin{remark}
    Up to a small $\bQ$-factorial modification, see \cite[Corollary 1.37]{Kol13}, which is an isomorphism away from the dissident locus of $X/G$, we can suppose that $Y$ is $\bQ$-factorial too.
\end{remark}

\section{Second Betti number of a terminalization}\label{sec:secondbetti}

\begin{proposition}\label{formulaBetti}
We use Notation~\ref{Notation:gen}. Let
\begin{itemize}
\item $L$ be a lattice isomorphic to $H^2(X,\IZ)$, 
    \item $N_2$ be the number of components $q(F_g)$ in $\Sigma$ with $\ord(g)=2$, 
    \item \(N_3\) be the number of components $q(F_g)$ in $\Sigma$ with $\ord(g)=3$.
\end{itemize}  
Then the identity 
\[b_2(Y)=\rank\left(L^G\right)+N_2+2N_3\]
holds except in the case where $X=\K_2(A)$ and  $G_{\circ} \simeq BD_{12}$, treated in \cref{remark:binarydihedral}.
\end{proposition}

\begin{remark}
Let $X \coloneqq S^{[n]}$ or $\K_{2}(A)$ with $M\coloneqq S$ or $A$, respectively. Recall that 
\[H^2(X, \IZ) \simeq H^2(M, \IZ) \oplus \IZ e,\]
where $2e$ is the class of the $\epsilon$-exceptional divisor. Since the group $G$ of induced automorphisms preserves the $\epsilon$-exceptional divisor, we obtain
\[H^2(X/G, \IQ) \simeq H^2(X, \IQ)^G \simeq H^2(M, \IQ)^G \oplus \IQ e.\]
We conclude that
\[\rank\left(L^G\right)=\rank\left(H^2(M)^G\right) +1.\]
\end{remark}

\begin{proof}[Proof of Proposition~\ref{formulaBetti}]
The blowup formula (or the decomposition theorem) gives
\[H^2(Y, \IQ)=I\!H^2(Y, \IQ) \simeq H^2(X/G, \IQ) \oplus \bigoplus_i H^0(E_i, \IQ),\]
where the sum runs over all the $p$-exceptional divisors $E_i$. 
Then, it suffices  to compute the $p$-exceptional divisors.  As shown in Section~\ref{sec:termexplicit}, the terminalization $p \colon Y \to X/G$ extracts an exceptional prime divisor for each $q(F_g)$ with transversal $A_1$ singularities, and at most two exceptional prime divisors for each $q(F_g)$ with transversal $A_2$ singularities. 

The latter case occurs only if $g$ has order $3$ and \(X=\K_{2}(A)\). Suppose  that this is indeed the case, and denote simply by $F$ the $g$-fixed surface. Consider the blowup of $\K_{2}(A)/G$ along $q(F)$, 
\[p_{q(F)} \colon \Bl_{q(F)}(\K_{2}(A)/G) \lra \K_{2}(A)/G.\]
A neighborhood $U$ of a general point in $q(F)$ is locally analytically isomorphic to the product $\IA^{2} \times (\IA^2/C_3)$. In particular, the restriction of $p_{q(F)}$ over $U$ extracts two exceptional prime divisors. Globally, these may be contained in two distinct $p_{q(F)}$-exceptional prime divisors of $\Bl_{q(F)}(\K_{2}(A)/G)$, or be two branches of the same non-normal $p_{q(F)}$-exceptional divisor. The latter case occurs only when $G_{\circ} \simeq BD_{12}$, as explained in \cref{Prop:excorder3}. We conclude that if $G_{\circ} \not \simeq BD_{12}$,  the terminalization \(p\colon Y\to X/G\) extracts exactly two exceptional prime divisors for each \(q(F_g)\) with transversal \(A_2\) singularities, whence the statement.
\end{proof}

\begin{lemma}\label{Prop:excorder3}
Suppose that a finite group $G$ of induced symplectic automorphisms of\, $\K_{2}(A)$ contains an element~$g$ of order $3$ fixing a surface $F$. Then the blowup $p' \colon Y' \to \K_{2}(A)/G$ of\, $\K_{2}(A)/G$ along $q(F)$ extracts two exceptional prime divisors, unless $g$ is contained in a subgroup of $G$ which is isomorphic to the binary dihedral group $BD_{12}$ and  splits the quotient $G \to G_{\circ} \simeq BD_{12} \subset \SL(\Lambda)$ $($cf.~\cref{sec:notation}\,$)$. In this case, the exceptional divisor of $p'$ is irreducible.
\end{lemma}

\begin{proof}
The $G$-orbit of $F$, denoted by $G \cdot F$, consists of $r$ irreducible components 
\[\new{F \coloneqq }F_g, F_{g_1}, \ldots, F_{g_j}=F_{h^{-1}_j g h_j}=h^{-1}_j(F_g), \ldots, F_{g_{r-1}},\]
where $g_j \coloneqq h^{-1}_j g h_j$ for some $h_j \in G$. Then consider  the blowup of $\K_{2}(A)$ along $G \cdot F$, 
\[p_1 \colon X_1 \coloneqq \Bl_{G \cdot F} \K_{2}(A) \lra \K_{2}(A), \]
with exceptional divisors $\widetilde{E}_0\coloneqq p^{-1}_1(F), \widetilde{E}_1 \ldots, \widetilde{E}_{r-1}$.

Denote by $\xi_3$ a primitive third root of unity, and let $Z \subset \K_{2}(A)$ be the locus of points in $\K_{2}(A)$ whose stabilizer is neither trivial nor conjugate to $\langle g \rangle$. 
The normal bundle of $F_{g_j}$ in $\K_{2}(A)$ splits into the sum of two $\langle g_j \rangle$-equivariant line bundles: 
\[N_{F_{g_j}/\K_{2}(A)} \simeq L_{\xi_3} \oplus L_{\bar{\xi}_3},\]
where $g_j$ acts on $L_{\xi_3}$ or $L_{\bar{\xi}_3}$ by scaling by $\xi_3$ and $\bar{\xi}_3$, respectively. Therefore, away from $Z$, $\widetilde{E}_j$ is $\langle g_j \rangle$-equivariantly isomorphic to $\IP\big(N_{F_{g_j}/\K_{2}(A)}\big) \simeq \IP\left(L_{\xi_3} \oplus L_{\bar{\xi}_3}\right)$ with two $\langle g_j \rangle$-fixed sections of $\widetilde{E}_j \to F_{g_j}$, denoted by 
\[s_{\xi_3, j} \coloneqq \Im\left(\IP\left(L_{\xi_3}\right) \hookrightarrow \IP\left(L_{\xi_3} \oplus L_{\bar{\xi}_3}\right)\right) \quad \text{and} \quad s_{\bar{\xi}_3, j} \coloneqq \Im\left(\IP\left(L_{\bar{\xi}_3}\right) \hookrightarrow \IP\left(L_{\xi_3} \oplus L_{\bar{\xi}_3}\right)\right).\]
Let $p_2 \colon X_2 \to X_1$ be the simultaneous blowup of the closure of the sections $s_{\xi_3, j}$ and $s_{\bar{\xi}_3, j}$ for $j=0, \ldots, r-1$, with exceptional divisors $\widetilde{E}_{\xi_3, j}$ and $\widetilde{E}_{\bar{\xi}_3, j}$. Since the centers of the blowups $p_1$ and $p_2$ are $G$-invariant, $X_2$ inherits an action by $G$. Further, away from $q(Z)$, the quotient $X_2/G$ is isomorphic to the blowup $p_3 \colon Y_1 \to Y'$ along the double locus of the $p'$-exceptional locus $E$; it suffices to verify it in the local model $\IA^{2} \times (\IA^2/C_3)$. Therefore, away from $Z$, there exists a commutative diagram 
\begin{equation}\label{square:terminalizationC3}
  \begin{tikzcd}[scale=1]
   \K_2(A) \arrow[d, "q"'] & X_1 \arrow[l, "p_1"'] & X_2 \arrow[l, "p_2"'] \arrow[d, "q_1"]\\
   \K_2(A)/G & Y' \arrow[l, "p'"'] & Y_1 \simeq X_2/G\rlap{,} \arrow[l, "p_3"']
\end{tikzcd}  
\end{equation}
where the horizontal arrows are blowups and the vertical arrows are $G$-quotient maps.

The $p'$-exceptional divisor $E$ is the image under $p_3 \circ q_1$ of the distinct divisors $\widetilde{E}_{\xi_3, j}$ and $\widetilde{E}_{\bar{\xi}_3, j}$. Suppose that $E$ is irreducible. Then there exists a $\iota \in G$ such that \[\iota\left(\widetilde{E}_{\xi_3, 0}\right)=\widetilde{E}_{\bar{\xi}_3, 0}.\] 
Note that the subgroup $H \subseteq G$ generated by $g$ and $\iota$ has the following presentation: 
\[
H=\left\langle g, \iota \mid \iota g \iota^{-1} = g^2, \iota^{2k}=1 \right\rangle.
\]
Indeed, the following facts hold:
\begin{itemize}
\item The automorphism~$\iota$ preserves $F_{g}$, so it exchanges $\widetilde{E}_{\xi_3, 0}$ and $\widetilde{E}_{\bar{\xi}_3, 0}$, and thus it has even order.
\item The automorphism~$\iota$ preserves the locus 
$\epsilon(F)= \{[(x,g(x),g^2(x))] \mid x \in A \}$, so either $\iota g = g \iota$ or $\iota g = g^2 \iota$.
\item For any ${(x,v) \in L_{\xi_3} \subset N_{F/\K_2(A)}}$ with $x \in F_{g}$ and $v \in L_{\xi_3, x}$, we have 
\[g \cdot (x,v) = (g \cdot x, dg_{x}(v))=(x, \xi_3 v).\]
Hence, we obtain
\begin{align*}
    g \cdot ( \iota \cdot x, d\iota_x(v)) & = g \iota \cdot (x, v) = \iota g^m \cdot (x,v)\\
    & =\iota \cdot (x, \xi^m_3 v)=(\iota \cdot x, d\iota_{x}(\xi^m_3 v))=(\iota \cdot x, \xi^m_3 d\iota_{x}(v)),
\end{align*}
\textit{i.e.}, $\iota(L_{\xi_3}) = L_{\xi^m_3}$. Since $\iota$ must exchange $L_{\xi_3}$ and $L_{\bar{\xi}_3}$, we must have $m=2$, \textit{i.e.}, $\iota g = g^2 \iota$.
\end{itemize}

Further constraints on the subgroup $H$ are imposed by the fact that $G$ is a group of symplectic automorphisms coming from an abelian surface $A$. Indeed, since $G \subset A \rtimes \SL(\Lambda)$, the order of $\iota$ must be $2$, $4$ or~$6$. 
\begin{itemize}
    \item If $\ord(\iota)=2$, then $H \simeq S_3$ and $H$ projects isomorphically onto  $H_{\circ}\subset \SL(\Lambda)$, which gives a contradiction since no subgroup of $\SL(\Lambda)$ is isomorphic to $S_3$. 
    \item If $\ord(\iota)=4$, then $H \simeq BD_{12}$ and $H$ projects isomorphically onto $G_{\circ} \subset \SL(\Lambda)$ since $BD_{12}$ is a maximal finite subgroup of $\SL(\Lambda)$; see \cref{rem:Fujiki results on AG}.
\item If $\ord(\iota)=6$, we replace $\iota$ with $\iota^3$ and obtain the same contradiction as in the case  $\ord(\iota)=2$.
\end{itemize}
We conclude that $E$ is irreducible if and only if $g$ is contained in the subgroup $H \simeq BD_{12} \simeq G_{\circ}$.
\end{proof}

We provide an alternative group-theoretic characterization of $N_2$ and $N_3$.

\begin{proposition}\label{formulaBetti2}
In the notation of \cref{formulaBetti}, we have the following: 
\begin{enumerate}
    \item\label{fB2-1} $N_2$ is the number of conjugacy classes of involutions of\, $G$ if\, $X=S^{[2]}$ or $\K_2(A)$. 
    \item\label{fB2-2} $N_2$ is the number of conjugacy classes of involutions satisfying \cref{lemma:inducedinvolution}\,\eqref{l:i-1} if\, ${X=\K_3(A)}$. 
    \item\label{fB2-3} $N_3$ is the number of conjugacy classes of subgroups of\, $G$ of order $3$  satisfying \cref{lemma:inducedinvolution}\,\eqref{l:i-2} if\, $X=\K_2(A)$. 
    \item\label{fB2-4} $N_2=N_3=0$ in all other cases.
\end{enumerate}
\end{proposition}

\proof
Recall that the \new{pointwise} stabilizer of $F_g$ is the group generated by $g$; in symbols, 
\[
\Stab(F_g) \coloneqq \{g \in G \mid \forall \underline{x} \in F_g,\; g(\underline{x})=\underline{x}\} = \langle g \rangle.
\]
It is a standard and general fact that
$q(F_g)=q(F_{g'})$ if and only if $g' \in h^{-1} \Stab(F_g) h \text{ for some }h \in G$,
hence
$q(F_g)=q(F_{g'})$ if and only if $g' \in h^{-1} \langle g\rangle h \text{ for some }h \in G$.
Together with \cref{sec:codim2 fixed loci}, this gives the group-theoretic characterization of $N_2$ and $N_3$ of the statement.
\endproof

\begin{lemma}\label{lemma:N2}
If\, $X=\K_2(A)$, then 
\[
N_2=\begin{cases}
    1 & \text{if }\, |G| \text{ is even},\\
    0 & \text{if }\, |G| \text{ is odd}.\\
\end{cases}
\]
\end{lemma}
\begin{proof}
If $|G|$ is odd, then we have $N_2=0$. Otherwise, any two involutions in $G$, 
namely $t_{1} \coloneqq \tau_{\alpha} (- \id)$ and $t_{2} \coloneqq \tau_{\alpha'} (- \id)$, are conjugate to each other as
\((t_2t_1)^{-1}t_1(t_2t_1)=t_2\);  hence $N_2=1$.
\end{proof}

\section{Third Betti number of a terminalization}\label{sec:third}

\begin{proposition}\label{prop:IH3}
  We use Notation~\ref{Notation:gen}. The third intersection cohomology group of\, $Y$ is the $G$-invariant part of the third cohomology group of\, $X$; \textit{i.e.},
\[I\!H^3(Y, \IQ) \simeq H^3(X, \IQ)^{G}.\] 
\end{proposition}
\begin{proof}
    Let $\Sigma_2= \bigcup F_{g}$ be the components of the singular locus of $X/G$ of pure codimension $2$. The decomposition theorem for the semismall terminalization $p \colon Y \to X/G$ gives
    \begin{equation}\label{eq:decompsemismall}
        Rp_* \mathcal{IC}_{Y} = \mathcal{IC}_{X/G} \oplus \mathcal{IC}_{\Sigma_2}\left(R^2p_*\IQ_{Y}\right)[-2]\oplus \mathcal{S},
    \end{equation}
    where $\mathcal{S}$ is a summand of the decomposition theorem supported in codimension at least $4$. Over a dense open set of $\Sigma_2$, the constructible sheaf $R^2p_*\IQ_{Y}$ is a trivial local system  (of rank $1$ or $2$, more precisely $\ord(g)-1$). The normalization $\nu \colon q(F_g)^{\nu} \to q(F_g)$ and $X/G$ have quotient singularities, so their intersection complexes are trivial local systems:
    \[\mathcal{IC}_{X/G}= \IQ_{X/G}, \quad \mathcal{IC}_{q(F_{g})}= \nu_*\mathcal{IC}_{q(F_{g})^{\nu}} = \nu_*\IQ_{q(F_{g})^{\nu}}.\]
    Therefore, we can rewrite \eqref{eq:decompsemismall} as 
    \[Rp_* \mathcal{IC}_{Y} = \IQ_{X/G} \oplus \bigoplus_{q(F_g) \subseteq \Sigma_2}\nu_*\IQ_{q(F_{g})^{\nu}}^{\ord(g)-1}[-2]\oplus \mathcal{S}. \]
    Taking $H^3$, we obtain that
    \[
    I\!H^3(Y, \IQ)=H^3(X, \IQ)^{G} \oplus \bigoplus_{q(F_g) \subseteq \Sigma_2} H^1\left({q\left(F_{g}\right)^{\nu}}, \IQ\right)^{\ord(g)-1}. 
    \]
    Let $\Stab(\{F_g\}) \coloneqq \{g \in G \mid g(F_g)=F_g\}$ be the setwise stabilizer of $F_{g}$. Then the Galois quotient
    \[F_g \to F_g/\Stab(\{F_g\})=q(F_{g})^{\nu}\]
    induces the inclusion
    \[H^1\left(q\left(F_{g}\right)^{\nu}, \IQ\right)= H^1\left(F_{g}, \IQ\right)^{\Stab(\{F_g\})}\subseteq H^1\left(F_{g}, \IQ\right).\] Since in our cases $F_g$ is simply connected, we conclude that $ I\!H^3(Y, \IQ)=H^3(X, \IQ)^{G}$.   
\end{proof}

\begin{proposition}\label{prop:3bettinumberKum}
  We use Notation~\ref{Notation:gen}. Suppose further that $X = \K_n(A)$ and $G_{\circ} \neq 1$. Then \[H^3(Y, \IQ)=I\!H^3(Y, \IQ)=0.\] 
\end{proposition}
\begin{proof}
 There exists a $G$-equivariant isomorphism \[H^1(A, \IZ) \oplus H^3(A, \IZ) \simeq H^3(\K_{n}(A), \IZ)/\mathrm{Tors};\]
 see \cite[Corollary 6.3]{Kapfer-Menet} or \cite[Theorem 2.7]{OG2021}, or the classical version with rational coefficients in \cite[Theorem 7]{GS1993}. A nontrivial symplectic linear automorphism $g \in G_\circ$ acting on $T_0 A \simeq H^{0,1}(A)$ does not fix any vector, so by \cref{prop:IH3}, 
\[0=H^1(A, \IQ)^G \oplus H^3(A, \IQ)^G \simeq H^3(\K_{n}(A), \IQ)^G \simeq I\!H^3(Y, \IQ).\]
Finally, note that $H^3(Y, \IQ) \simeq I\!H^3(Y, \IQ)$ since
$Y$ has quotient singularities by \cref{cor:quotientsingularities}. 
\end{proof}

\section{Fundamental group of the regular locus of a terminalization}\label{sec:fundgroup}

\begin{proposition}\label{prop:fundamentalgroup}
Let $X$ be a simply connected smooth \new{complex} symplectic variety endowed with an action of a finite group $G$ of symplectic automorphisms. Let
$p \colon Y \to X/G$ be a terminalization of the quotient. The fundamental group of the regular locus of\, $Y$ is 
\[\pi_1(Y^{\reg}) \simeq G/N,\]
where $N \triangleleft G$ is the normal subgroup generated by the elements $\gamma \in G$ whose fixed locus in $X$ has codimension~$2$. The universal quasi-\'{e}tale cover of\, $Y$ is a terminalization of the quotient $X/N$.
\end{proposition}

\cref{prop:fundamentalgroup} is a refinement of \cite[Proposition 2.13]{GM22}.

\begin{remark}\label{rmk:indepen} The fundamental group of the regular locus of a terminalization of $X/G$ is actually independent of the choice of the given terminalization since all terminalizations of $X/G$ are isomorphic in codimension~$1$. In general, however, the fundamental group of the regular locus of a variety is not a birational invariant. For instance, the fundamental group of the regular locus of the singular Kummer surface $A^{(2)}_0$ is infinite, but its minimal resolution is simply connected.
\end{remark} 

\begin{remark}
 The subgroup \(N\) generated by elements in $G$ whose fixed locus in \(X\) admits a component of codimension~$2$ is a normal subgroup  of \(G\). Indeed, the property of an element of having a component of the fixed locus of a certain codimension is invariant up to conjugation: If $g$ fixes a locus $F$ of codimension $m$, then $hgh^{-1}$ fixes the locus $h(F) \simeq F$ of the same codimension.
 It follows that any element conjugate to a generator of \(N\) is in \(N\); hence $N$ is normal.
\end{remark}

\begin{proof}[Proof of \cref{prop:fundamentalgroup}]
The quotient map $q \colon X \to X/G$ is \'{e}tale over the regular locus of $X/G$. Therefore, we have a short exact sequence
\[1 \lra \pi_1\left(q^{-1}((X/G)^{\reg})\right)\lra \pi_1((X/G)^{\reg}) \lra G \lra 1.\]
Since $X$ is simply connected and $q$ is \'{e}tale in codimension~$1$, we have 
$\pi_1((X/G)^{\reg}) \simeq G$.
As $(X/G)^{\reg}$ can be identified with a Zariski dense open subset of $Y^{\reg}$, we have a surjective map
\[G \simeq \pi_1((X/G)^{\reg}) \longtwoheadrightarrow \pi_1(Y^{\reg}).\]

Let $F$ be a codimension~$2$ subvariety of $X$ fixed by an element of $G$. An analytic neighborhood $U$ of a general point of $q(F)$ in $X/G$ is isomorphic to \rev{an analytic open set of} $\IA^{\dim X -2} \times W$, where $W$ is the canonical surface singularity $\IA^2/\Stab(F)$. The restriction of a terminalization $p \colon Y \to X/G$ to $U$ is isomorphic to \rev{an analytic simply connected open subset $\widetilde{U}$ of} $\IA^{\dim X -2} \times \widetilde{W}$, where $\widetilde{W}$ is the unique (simply connected) minimal resolution of $W$. By inclusions, we obtain the following commutative diagram: 
\[
\begin{tikzcd}[scale=1]
   \Stab(F)= \pi_1\left(\IA^2 \setminus\{0\}/\Stab(F)\right) = \pi_1(U \cap (X/G)^{\reg}) \arrow[d] \arrow[r] &  \pi_1\left(\widetilde{U}\right)=1 \arrow[d]\\
   \pi_1((X/G)^{\reg}) \arrow[two heads, r] & \pi_1(Y^{\reg})\rlap{.}
\end{tikzcd}
\]
Therefore, there exists a surjective map 
\[G/N \longtwoheadrightarrow \pi_1(Y^{\reg}).\]
We prove that the previous surjection is invertible. Let $p_{N} \colon Y_N \to X/N$ be a \corre{$G/N$-equivariant terminalization of $X/N$.} 
We obtain the following commutative square: 
\[
\begin{tikzcd}[scale=1]
   X/N \arrow[d, "q_{1}"'] & Y_N \arrow[l, "p_N"'] \arrow[d, "q_2"]\\
   X/G & Y_N/(G/N)\rlap{,} \arrow[l, "p"']
\end{tikzcd}
\]
where the horizontal arrows are birational morphisms and the vertical arrows $G/N$-quotient maps. Let $(X/G)^{\circ}$ be the complement of the dissident locus; see  \cref{def:dissident}. By the definition of $N$, $q_1$ is \'{e}tale over $(X/G)^{\circ}$, so $p^{-1}((X/G)^{\circ})$ is a symplectic resolution of $(X/G)^{\circ}$ built via the same sequence of blowups which gives $Y_{N}$ over $(X/N)^{\circ}$. We conclude that $Y_N/(G/N)$ is a terminalization of $X/G$ by \cref{prop:terminalization}, and by \cref{rmk:indepen}, there exists a surjective morphism
\begin{equation*}\pushQED{\qed}
\pi_1(Y^{\reg}) \longtwoheadrightarrow G/N.
\qedhere \popQED
	\end{equation*}
\renewcommand{\qed}{}  
\end{proof}

\begin{corollary}\label{cor:fundgroupKA}
We use Notation~\ref{Notation:gen}. The fundamental group of the regular locus of\, $Y$ is 
\[\pi_1(Y^{\reg}) \simeq G/N,\]
where $N$ is the normal subgroup generated by all elements of
\begin{itemize}
    \item order $2$ if $X=S^{[2]}$,
    \item order $2$ and $3$ satisfying \cref{lemma:inducedinvolution}\,\eqref{l:i-1} and~\eqref{l:i-2}  if $X=\K_2(A)$, 
    \item order $2$ satisfying \cref{lemma:inducedinvolution}\,\eqref{l:i-1}  if $X=\K_3(A)$.
\end{itemize}
In all other cases, $\pi_1(Y^{\reg}) \simeq G$.
\end{corollary}

\section{Terminalizations of quotients of Hilbert schemes on \texorpdfstring{$\boldsymbol{\K3}$}{K3} surfaces}\label{sec:classificationK3} 
Symplectic actions of finite groups on \(S^{[2]}\) have been classified in \cite[Table 12]{HM}. Here we restrict to the groups $G$ of even order whose action comes from an action on the underlying K3 surface $S$ (which are marked with the label \emph{Type K3} in the fourth column of \cite[Table 12]{HM}). Since any involution gives rise to a surface with transversal $A_1$ singularities in $S^{[2]}/G$ (see \cref{rem fix locus involutions on Hilb}),
the previous conditions grant that the quotient $S^{[2]}/G$ is not terminal, as required by our criteria of classification (\textit{cf.} Section~\ref{sec:criteria}).

In Table~\ref{table Hilb2(K3)}, for any such group $G$, we list
\begin{itemize}
    \item the group ID as in \GroupNames, 
    \item an alias of $G$ as abstract group, 
    \item \(\rank \coloneqq \rank\big(H^2(S^{[2]})^G\big)=\rank\big(H^2(S)^G\big)+1\), as computed in \cite[Table 12, fifth column]{HM},  
    \item the number \(N_2\) of codimension $2$ components of the singular locus of \(S^{[2]}/G\), as computed in \cref{formulaBetti2}\eqref{fB2-1}, 
    \item \(b_2(Y)  =\rank + N_2 \), see \cref{formulaBetti} and \cref{lemma:codime2} for the fact that $N_3=0$,
    \item \(\pi_1(Y^{\reg})\simeq G/N\), where $N$ is the subgroup generated by involutions, see \cref{cor:fundgroupKA}. 
\end{itemize}
We highlight in gray the quotients whose terminalization has simply connected regular locus.
\small\renewcommand*{\arraystretch}{1.1}
\begin{longtable}{c|c|c|c|c|c}
\caption{Terminalizations of \(S^{[2]}/G\)} \\
 ID\label{table Hilb2(K3)}
 & \(G\)
 & \(\rank\)
 & $N_2$
 & \(b_2(Y)\)  
 &\(\pi_1(Y^{\reg})\) 
 \\\hline \endfirsthead
 ID
 & \(G\)
 & \(\rank\)
 & $N_2$
 & \(b_2(Y)\)  
 &\(\pi_1(Y^{\reg})\) 
 \\\hline \endhead
\rowcolor{backcolour} 2,1 & \(C_2\) & 15 & 1 & 16 & \{1\}
 \\\hline
 4,1 & \(C_4\) & 9 & 1 & 10 & \(C_2\)
 \\
 \rowcolor{backcolour} 4,2 & \(C_2^2\) & 11 & 3 & 14 & \{1\}
 \\\hline
 \rowcolor{backcolour} 6,1 & \(S_3\) & 9 & 1 & 10 & \{1\}
 \\
 6,2 & \(C_6\) & 7 & 1 & 8 & \(C_3\)
 \\\hline
 8,1 & \(C_8\) & 5 & 1 & 6 & \(C_4\)
 \\
 8,2 & \(C_2\times C_4\) & 7 & 3 & 10 & \(C_2\)
 \\
 \rowcolor{backcolour} 8,3 & \(D_4\) & 8 & 3 & 11 & \{1\}
 \\
 8,4 & \(Q_8\) & 6 & 1 & 7 & \(C_2^2\)
 \\
 \rowcolor{backcolour} 8,5 & \(C_2^3\) & 9 & 7 & 16 & \{1\}
 \\\hline
 \rowcolor{backcolour} 10,1 & \(D_5\) & 7 & 1 & 8 & \{1\}
 \\\hline
 12,1 & \(BD_{12}\) & 5 & 1 & 6 & \(S_3\)
 \\
 12,3 & \(A_4\) & 7 & 1 & 8 & \(C_3\)
 \\
\rowcolor{backcolour} 12,4 & \(D_6\) & 7 & 3 & 10 & \{1\}
 \\
 12,5 & \(C_2\times C_6\) & 5 & 3 & 8 & \(C_3\)
 \\\hline
 16,2 & \(C_4^2\) & 5 & 3 & 8 & \(C_2^2\)
 \\
 16,3 & \(C_2^2\rtimes C_4\) & 6 & 5 & 11 & \(C_2\)
 \\
 16,6 & \(M_4(2)\) & 4 & 2 & 6 & \(C_4\)
 \\
 16,8 & \(Q_8\rtimes C_2\) & 5 & 2 & 7 & \(C_2\)
 \\
 16,9 & \(Q_{16}\) & 4 & 1 & 5 & \(D_4\)
 \\
 \rowcolor{backcolour} 16,11 & \(C_2\times D_4\) & 7 & 7 & 14 & \{1\}
 \\
 16,12 & \(C_2\times Q_8\) & 5 & 3 & 8 & \(C_2^2\)
 \\
 \rowcolor{backcolour} 16,13 & \(C_4\circ D_4\) & 6 & 4 & 10 & \{1\}
 \\
 \rowcolor{backcolour} 16,14 & \(C_2^4\) & 8 & 15 & 23 & \{1\}
 \\\hline
 18,3 & \(C_3\times S_3\) & 5 & 1 & 6 & \(C_3\)
 \\
 \rowcolor{backcolour} 18,4 & \(C_3\rtimes S_3\) & 7 & 1 & 8 & \{1\}
 \\\hline
 20,3 & \(C_5\rtimes C_4\) & 5 & 1 & 6 & \(C_2\)
 \\\hline
 24,3 & \(Q_8\rtimes C_3\) & 4 & 1 & 5 & \(A_4\)
 \\
 \rowcolor{backcolour} 24,8 & \(C_3\rtimes D_4\) & 5 & 3 & 8 & \{1\}
 \\
 \rowcolor{backcolour} 24,12 & \(S_4\) & 6 & 2 & 8 & \{1\}
 \\
 24,13 & \(C_2\times A_4\) & 5 & 3 & 8 & \(C_3\)
 \\\hline
 32,6 & \(C_2^3\rtimes C_4\) & 5 & 5 & 10 & \(C_2\)
 \\
 32,7 & \(C_4 . D_4\) & 4 & 4 & 8 & \(C_2\)
 \\
 32,11 & \(C_4 \wr C_2\) & 4 & 3 & 7 & \(C_2\)
 \\
 \rowcolor{backcolour} 32,27 & \(C_2^2 \wr C_2\) & 6 & 10 & 16 & \{1\}
 \\
 32,31 & \(C_4 ._{4} D_4\) & 5 & 5 & 10 & \(C_2\)
 \\
 32,44 & \(C_8.C_2^2\) &  4 & 3 & 7 & \(C_2\)
 \\
 \rowcolor{backcolour} 32,49 & \(D_4 \circ D_4\) & 6 & 10 & 16 & \{1\}
 \\\hline
 36,9 & \(C_3^2 \rtimes C_4\) & 5 & 1 & 6 & \(C_2\)
 \\
 \rowcolor{backcolour} 36,10 & \(S_3^2\) & 5 & 3 & 8 & \{1\}
 \\
 36,11 & \(C_3 \times A_4\)& 5 & 1 & 6 & \(C_3^2\)
 \\\hline
 48,3 & \(C_4^2 \rtimes C_3\) & 5 & 1 & 6 & \(A_4\)
 \\
 \rowcolor{backcolour} 48,29  & \(Q_8 \rtimes S_3\) & 4 & 2 & 6 & \{1\}
 \\
 48,30 & \(A_4 \rtimes C_4\) & 4 & 3 & 7 & \(S_3\)
 \\
 \rowcolor{backcolour} 48,48 & \(C_2 \times S_4\) & 5 & 5 & 10 & \{1\}
 \\
 48,49 & \(C_2^2 \times A_4\)& 4 & 7 & 11 & \(C_3\)
 \\
 48,50 & \(C_2^2 \rtimes A_4\) & 6 & 5 & 11 & \(C_3\)
 \\\hline
 \rowcolor{backcolour} 60,5 & \(A_5\) & 5 & 1 & 6 & \{1\}
 \\\hline
 64,32 & \(C_2 \wr C_4\) & 4 & 6 & 10 & \(C_2\)
 \\
 64,35 & \( C_4^2 \rtimes_3 C_4\) & 4 & 4 & 8 & \(C_2^2\)
 \\
 64,136 & \(D_4 ._9 D_4\) & 4 & 6 & 10 & \(C_2\)
 \\
 \rowcolor{backcolour} 64,138 & \(C_2 \wr C_2^2\) & 5 & 9 & 14 & \{1\}
 \\
 \rowcolor{backcolour} 64,242 & \(C_2^{4} \rtimes C_2^{2}\) & 5 & 9 & 14 & \{1\}
 \\\hline
 \rowcolor{backcolour} 72,40 & \(S_3 \wr C_2\) & 4 & 3 & 7 & \{1\} 
 \\
 72,41 & \(C_3^2 \rtimes Q_8\) & 4 & 1 & 5 & \(C^2_2\)
 \\
 \rowcolor{backcolour} 72,43 & \(C_3 \rtimes S_4\) & 5 & 2 & 7 & \{1\} 
 \\\hline
 80,49 & \(C_2^{4} \rtimes C_5\) & 4 & 3 & 7 & \(C_5\)
 \\\hline 
 \rowcolor{backcolour} 96,64 & \(C_{4}^2\rtimes S_3\) & 4 & 2 & 6 & \{1\}
 \\
 96,70 & \(C_2^{4} \rtimes C_6\)& 4 & 4 & 8 & \(C_3\)
 \\
 \rowcolor{backcolour} 96,195 & \(A_4 \rtimes D_4\) &  4 & 6 & 10 & \{1\}
 \\
 96,204 & \(C_2^{3} \rtimes A_4\)& 4 & 4 & 8 & \(C_3\)
 \\
 \rowcolor{backcolour} 96,227 & \(C_2^2 \rtimes S_4\) & 5 & 5 & 10 & \{1\}
 \\\hline
 \rowcolor{backcolour} 120,34 & \(S_5\) & 4 & 2 & 6 & \{1\}
 \\\hline
 \rowcolor{backcolour} 128,931 & \(C_4^2 \rtimes _5 D_4\) & 4 & 7 & 11 & \{1\}
 \\\hline
 144,184 & \(A_4^2\) & 4 & 3 & 7 & \(C_3^2\)
 \\\hline
 \rowcolor{backcolour} 160,234 &  \(C_2^{4} \rtimes D_5\) & 4 & 4 & 8 & \{1\}
 \\\hline
 \rowcolor{backcolour} 168,42 & \(\GL_3(\mathbb{F}_2)\) & 4 & 1 & 5 & \{1\}
 \\\hline
 \rowcolor{backcolour} 192,955 & \(C_2^4 \rtimes D_6\) & 4 & 6 & 10 & \{1\}
 \\
 192,1023 & \(C_4^{2} \rtimes A_4\) & 5 & 3 & 8 & \(C_3\)
 \\ 
 \rowcolor{backcolour} 192,1493 & \(C_2^{3} \rtimes S_4\) &  4 & 6 & 10 & \{1\}
 \\\hline
 \rowcolor{backcolour} 288,1026 & \(A_4 \rtimes S_4\) & 4 & 4 & 8 & \{1\}
 \\\hline
 \rowcolor{backcolour} 360,118 & \(A_6\) & 4 & 1 & 5 & \{1\}
 \\\hline
 \rowcolor{backcolour} 384,18135 & \(F_{384}\) & 4 & 4 & 8 & \{1\}
 \\\hline
 \rowcolor{backcolour} 960,11357 & \(M_{20}\) & 4 & 2 & 6 & \{1\}
 \\
\end{longtable}

\normalsize

\begin{proposition} \label{prop:K3_2 terminalization nonsmooth}
All terminalizations in Table~\ref{table Hilb2(K3)} are singular with the exception of\, $G \simeq C_2^4$.
\end{proposition}

\begin{proof}

If the terminalization $Y$ is smooth, then \new{in particular} the quotient $S^{(2)}/G$ does not admit an isolated singularity $[(x,y)]$ with $x \neq y$.
In fact, such points lie in the locus where the birational morphism \(Y\to S^{(2)}/G\) is an isomorphism. 
    
Equivalently, for any \(g\in G\), there exists no point \((x,y)\in S^2\) such that $(g(x), g(y))=(x,y)$ and $y \notin G x$. Otherwise, such a $g$-fixed point would not be of the form $(x,\iota(x))$ for any involution $\iota \in G$, so it would not lie on a codimension $2$ component of the fixed locus of some element $g \in G$, and it certainly would give rise to a singularity of $Y$. 

Equivalently, if \(Y\) is smooth, then the following statement holds true.  

\begin{assumption}\label{assumption:false}
        For any $g \in G$, the fixed locus $\Fix(g) \subset S$ lies in a fiber of $\pi: S \to S/G$.
\end{assumption}

We deduce the following  \cref{lemma:ordersmooth,lemma:corresp}. 

\begin{lemma}\label{lemma:ordersmooth} Under Assumption~\ref{assumption:false}, $\ord(g) \cdot |\Fix(g)|$ divides $|G|$ for any $g \in G$.
\end{lemma}

\begin{proof}[Proof of \cref{lemma:ordersmooth}]
Given $g \in G$ and \(x\in \Fix(g) \subset S\), any point of the orbit $G x$ is fixed by a conjugate of $g$ since \(\Fix(hgh^{-1})=h \Fix(g)\) with \(h\in G\). Moreover, $\Fix(g) \subseteq G x$ by Assumption~\ref{assumption:false}, so the orbit \(G x\) is the disjoint union of the fixed loci of conjugates of $g$, and $|\Fix(g)|$ divides $ |G x|$. We conclude that
\begin{equation*}\pushQED{\qed}
  \ord(g) \cdot |\Fix(g)| \text{ divides } |G_x| \cdot |G x|=|G|.\qedhere \popQED
	\end{equation*}
\renewcommand{\qed}{}     
\end{proof}

\begin{lemma}\label{lemma:corresp} Under Assumption~\ref{assumption:false}, there exists a bijective correspondence  
\begin{align*}
        \{\text{conjugacy classes of involutions of }G \} & \xleftrightarrow{\hphantom{\hspace{10pt}}} \{\text{singular points in } S/G \text{ with even isotropy}\}\\
        [\iota] & \longmapsto \, \pi(\Fix(\iota)).
\end{align*}
\end{lemma}
\begin{proof}[Proof of \cref{lemma:corresp}]
The correspondence $\iota \mapsto \pi(\Fix(\iota))$ is well defined since $\Fix(\iota)$ lies in the same $\pi$-fiber by Assumption~\ref{assumption:false}. It is also independent of the representative of $[\iota]$ since $\Fix(g \iota g^{-1})=g \Fix(\iota)$. The inverse map sends a singular point $q$ to the conjugacy class of the unique involution of $G_x$ for any $x \in \pi^{-1}(q)$. The uniqueness of such an involution follows from the faithfullness of the action of $G_x$ as finite subgroup of $\SL(2, \IC)$ on the tangent space $T_x S$ (recall that there is a unique involution in $\SL(2, \IC)$). 
\end{proof}

By \cite{guan.betti}, if $Y$ is a smooth IHS fourfold, then either $3 \leq b_2(Y) \leq 8$ or $b_2(Y)=23$, and according to Table~\ref{table Hilb2(K3)}, the latter occurs only if $G \simeq C_2^4$. All terminalizations $Y$ in Table~\ref{table Hilb2(K3)} with \(b_2(Y)\le 8\) and \(\pi_1(Y^{\reg})=\{1\}\) fail to satisfy the necessary conditions for smoothness detailed in Lemmas~\ref{lemma:ordersmooth} and~\ref{lemma:corresp}. 
In order to apply these lemmas, we use the classification of the singularities of $S/G$ obtained in \cite{Xiao1996}  and the computation of the cardinality of $\Fix(g_{n}) \subset S$ for a symplectic automorphism $g_n$ on $S$ of order $n$, contained for instance in \cite[Section~5]{nikulin1980finite}:
 \begin{longtable}{c|c c c c c c c}
  \(n\) 
  & 2
  & 3
  & 4
  & 5
  & 6
  & 7
  & 8
  \\\hline
  \(| \Fix(g_n)|\) 
  & 8
  & 6
  & 4
  & 4
  & 2
  & 3
  & 2
 \end{longtable}
More precisely, we are able to exclude all the cases, as
\begin{itemize}
\item if $|G|= 160$ or $288$, \cref{lemma:corresp} fails, 
\item if $|G|=48, 96, 384$ or $960$, \cref{lemma:ordersmooth} fails since $3 \mid |G|$ but $3 \cdot 6 \nmid |G|$, 
\item for all other groups $G$, \cref{lemma:ordersmooth} fails since $2 \mid |G|$ but $2 \cdot 8 \nmid |G|$.\qedhere
\end{itemize}
\renewcommand{\qed}{}
\end{proof}

\begin{theorem}\label{thm:S2new}
    Let $G$ be a finite group of induced symplectic automorphisms acting on $S^{[2]}$, and let $Y$ be a projective terminalization of $S^{[2]}/G$ with simply connected regular locus. There are at least five new deformation classes of such irreducible symplectic varieties $Y$. In particular, they are not deformation equivalent to any terminalization of quotients of Kummer fourfolds by groups of induced symplectic automorphisms, or a Fujiki fourfolds appearing in \cite[Theorem 1.11]{GM22}.
\end{theorem}

\begin{longtable}{c|c|c}
 ID &  $G$ & $b_2(Y)$
  \\\hline
   10,1  &  \(D_5\) & 8 \\
   60,5 & \(A_5\) & 6\\
   120,34 & \(S_5\) & 6\\
   168,42 & \(\GL_3(\mathbb{F}_2)\) & 5 \\
  360,118 & \(A_6\) & 5
 \end{longtable}
 
\begin{proof}
If the projective terminalizations $Y_1 \to S^{[2]}/G_1$ and  $Y_2 \to S^{[2]}/G_2$ are deformation equivalent, then $b_2(Y_1)=b_2(Y_2)$ and  $\sqrt{|G_1|/|G_2|}$ is a rational number; see \cite[Proposition 3.21, Proof of Proposition 1.13]{GM22}. We then conclude  by direct inspections of Tables~\ref{table Hilb2(K3)} and~\ref{table sing n=2} and \cite[Section~5]{GM22}. 

Note that the terminalization $Y\to S^{[2]}/D_5$ has $b_2(Y)=8$, but it cannot be deformation equivalent to any of the new terminalizations with $b_2=8$ in \cref{table sing n=2}. Indeed, the subgroup $C_5 \triangleleft D_5$ fixes two points $z_1, z_2$ on $S$ lying in different $D_5$-orbits (\textit{cf.} \cite[Theorem 3]{Xiao1996}), and so the point $(z_1, z_2) \in S^{2}$ corresponds to an isolated singularity of $Y$ with isotropy $C_5$, but this singularity never appears in \cref{table sing n=2}. 
\end{proof}

\begin{remark}
   \rev{The terminalizations of $S^{[2]}/G$ in Table~\ref{table Hilb2(K3)} are Fujiki varieties $S(G)_{\theta}^{[2]}$ with trivial involution $\theta = \id$; see \cref{defb:Fujiki variety}. More information on their singularities is available in \cite{GM22}, provided that $G$ is an \emph{admissible} group of induced symplectic automorphisms; see \cite[Definition 1.10]{GM22}.}
\end{remark} 

 \section{Terminalizations of quotients of generalized Kummer manifolds}\label{sec:classificationKn} 
In this section, we compute the second Betti number  and the fundamental groups of the regular locus of terminalizations of quotients of $\K_2(A)$ and $\K_3(A)$ by finite groups of induced symplectic automorphisms; see Tables~\ref{table n=2} and~\ref{table n=3}, respectively.

\subsection{Symplectic automorphisms of an abelian surfaces}\label{rem:Fujiki results on AG}
Let $G$ be a finite group of symplectic automorphisms of an abelian surface $A$.
In the notation of Section~\ref{sec:notation}, the group $G\subseteq A[n+1]\rtimes \SL(\Lambda)$ fits in the short exact sequence
\[1\lra G_\textrm{tr}\lra G\lra G_\circ\lra 1.\] 
By the classification of finite subgroups of $\SL(2,\IC)$ together with \cite[Lemma 3.3]{Fujiki88}, $G_{\circ}$ is isomorphic to $\{1\}$, $C_m$ for $m \in \{2,3,4,6\}$, $Q_8$, $BD_{12}$ or $BT_{24}$. Moreover, by \cite[Remarks 3.6 and~3.12]{Fujiki88}, $(A,G_\circ)$ is deformation equivalent to one of the following: 
\begin{center}
$\begin{aligned}
& (A, 
\langle 
-\id
\rangle \simeq C_2),
&& &&& \\
& (E^2, 
\langle 
g_m
\rangle \simeq C_m)
&& \textrm{for } E=\IC/{\langle 1, \xi_m} \rangle, 
&&& g_m=\begin{psmallmatrix}
    \xi_m & 0 \\ 0 & \xi_m^{-1}
\end{psmallmatrix},
&&&& \xi_m=e^{\frac{2\pi i}{m}}, \\
& (E^2, 
\langle 
h, 
k
\rangle \simeq Q_8)
&& \textrm{for }  E=\IC/{\langle 1, i \rangle},
&&& h=\begin{psmallmatrix}
    0 & -1 \\ 1 & 0
\end{psmallmatrix},
&&&& k=g_4=\begin{psmallmatrix}
    i & 0 \\ 0 & -i
\end{psmallmatrix}, \\
& (\mathbb H/\Gamma, \langle i,j\rangle \simeq Q_8)
&& \textrm{for }  \mathbb H=\mathbb R[1,i,j,k],
&&& \Gamma=\mathbb Z[1,i,j,t],
&&&& t=\tfrac{1+i+j+k}{2}, \\
& (E^2, 
\langle 
h, 
l
\rangle \simeq BD_{12})
&& \textrm{for }  E=\IC/{\langle 1, \xi_6 \rangle}, 
&&& h=\begin{psmallmatrix}
    0 & -1 \\ 1 & 0
\end{psmallmatrix},
&&&& l=g_6=\begin{psmallmatrix}
    \xi_6 & 0 \\ 0 & \xi_6^{-1}
\end{psmallmatrix}, \\
& (\mathbb H/\Gamma, \Gamma^\times \simeq BT_{24})
&& \textrm{for }  \mathbb H=\mathbb R[1,i,j,k],
&&& \Gamma=\mathbb Z[1,i,j,t],
&&&& t=\tfrac{1+i+j+k}{2}, \\
& && 
&&& \Gamma ^\times = \langle r,t \rangle ,
&&&& r=\tfrac{1+i+j-k}{2};
\end{aligned}
$
\end{center}
\normalsize
\noindent see also \cite[Proposition 3.7, Lemma 2.6]{Fujiki88} and the surveys   \cite[Section~2.2]{Pietromonaco2022} and \cite[Appendix 2]{KMO2023}. Therefore, without loss of generality, we can assume that $G_{\circ}$ acts on $A$ as above,  and we identify actions of $G$ up to conjugation in $A[n+1] \rtimes \SL(\Lambda)$. In fact, the topological invariants we are interested in, that is,  $b_2(Y)$ and $\pi_1(Y^{\reg})$, are independent on the deformation type of the pair $(A,G)$ and invariant under conjugation in $A[n+1] \rtimes \SL(\Lambda)$. 

\begin{lemma} \label{lemma: rank H2invariant}
Let $G$ be a finite group of symplectic automorphisms of an abelian surface $A$.
Then,
$$ \rank(H^2(A)^G)=
\begin{cases}
 6 & \text{if }\, G_\circ \simeq C_2,\\
 4 & \text{if }\, G_\circ\simeq C_3, C_4, C_6, \\
 3 & \text{if }\, G_\circ\simeq Q_8, BD_{12}, BT_{24}.
\end{cases}
$$
\end{lemma}
\begin{proof}
Note that the group
\[A \rtimes (-\id) = \ker \{ A \rtimes \SL(\Lambda) \lra \SL(\Lambda) \lra \PSL(\Lambda)\}\]
acts trivially on $H^2(A)$, so $H^2(A)^G = H^2(A)^{G_{\circ}}$ and if $-\id \in G_\circ$, we have $H^2(A)^{G} = H^2(A)^{G_{\circ}/\langle -\id \rangle}$. The claim then follows from \cite[Section~6]{Fujiki88}.
\end{proof}

\subsection{Second Betti numbers and fundamental groups of terminalizations}\label{sec:terminalizationKnKn} Let $G$ be a finite group of induced symplectic automorphisms of $\K_2(A)$ or $\K_3(A)$. 
In \cref{table n=2,table n=3}, we list
\begin{itemize}
    \item the group ID of $G$ as in \GroupNames, when available (otherwise, we write NA), 
    \item an alias of $G$ as abstract group (we express $G$ as a (split or non-split) extension \revi{of \(G_\circ\) by \(G_{\text{tr}}\)} and, when available, we adopt the enumeration of extensions in \GroupNames; otherwise, we add the subscript * for unnumbered extensions), 
    \item \(\rank\coloneqq \rank\left(H^2(\K_2(A))^G\right)=\rank(H^2(A)^G)+1\), where the ranks are computed in \cref{lemma: rank H2invariant}, 
    \item the number \(N_i\) of components of codimension~$2$ of the singular locus of \(\K_n(A)/G\) with transversal \(A_{i-1}\) singularities, see \cref{formulaBetti2}, 
    \item   
    \(b_2(Y)\), see \cref{formulaBetti} and \cref{remark:binarydihedral}, 
    \item \(\pi_1(Y^{\reg})\), see \cref{cor:fundgroupKA}.  
\end{itemize}
The explicit values of \(N_i\), \(b_2(Y)\) and the groups \(\pi_1(Y^{\reg})\) can be obtained using GAP.\footnote{A GAP code containing all the calculations is available from the
authors upon request.}
We highlight in gray the quotients whose terminalization has simply connected regular locus.

\begin{example} \label{ex:differentquotient}
Let $\xi_3$ be a primitive third root of unity,  and let $E$ be an elliptic curve with complex multiplication $\xi_3 \curvearrowright E \colon x \mapsto \xi_3 \cdot x$. Consider the symplectic automorphism $g_3(x_1,x_2)=(\xi_3 x_1, \xi^{-1}_3 x_2)$. Choose \revi{$a=(a_1, a_2)$ and $b=(b_1, b_2)$ in $E^2[3]$} such that $g_3(a)=a$ and $g_3(b) \neq b$. Denote by $\tau_a, \tau_b \colon E^2 \to E^2$ the translations $\tau_a(x_1, x_2)=(x_1+a_1, x_2+a_2)$ and $\tau_b(x_1, x_2)=(x_1+b_1, x_2+b_2)$. Now, both $\tau_{a} g_3$ and $\tau_{b}  g_3$ induce the same action on $H^*(\K_2(E^2), \IZ)$. However, the quotient $\K_2(E^2)/\langle \tau_{a}  g_3\rangle$ has strictly canonical singularities, while $\K_2(E^2)/\langle \tau_{b}  g_3\rangle$ is terminal. The actions correspond to the two distinct rows for the cyclic group $C_3$ in Table~\ref{table n=2}.
\end{example} 

\begin{remark}[Group actions \textit{vs.}~abstract realization] \label{rmk:groupaction}
    In \cref{ex:differentquotient}, we pointed out that different actions of the same abstract group may lead to terminalizations with different deformation type. This explains why in the tables below, the same abstract group may appear multiple times. \textit{A priori}, the table should include all possible actions, namely all possible subgroups of $A[n+1] \rtimes G_{\circ}$. A GAP code can easily provide all of them and their relevant invariants, but to avoid redundancy,  \begin{quote} \centering
        \emph{we identify groups which are conjugate in $A[n+1] \rtimes G_\circ$\\ or give the same string of invariants $[G,G_\circ, \rank, N_i, b_2, \pi_1]$.}
        \vspace{0.1 cm}
    \end{quote} 
    Note that the latter condition leaves open the possibility that there may be terminalizations with different deformation type but same string of invariants; see for instance \cref{ex:differentdeformation}. 
    In this regard, we check the following facts:  \begin{itemize}
         \item Using a GAP code, we observe that the actions whose quotients have a smooth terminalization, \textit{i.e.}, $C_3^3$ in \cref{table n=2} and $C_2^5$ in \cref{table n=3},  are unique up to conjugation in $A[n+1] \rtimes G_\circ$. 
         \item By elementary algebraic considerations, the actions whose quotients admit a terminalization with simply connected regular locus (the most relevant according to Section~\ref{sec:criteria}) are all affine; \textit{i.e.}, the group $G$ is conjugate to the semidirect product $G_\tr \rtimes G_\circ$ by an element in $A[n+1]$; see \cref{lem:G is affine}.
    \end{itemize}
\end{remark}

\begin{remark}[Quaternion group]\label{ex:differentdeformation}

The moduli space of pairs $(A,G_\circ)$, where $A$ is an abelian surface and $G_\circ$ is a symplectic group of linear automorphisms, is connected except for $G_\circ=Q_8$, in which case it has two connected components represented by the pairs $(E^2, Q_8)$ and $(\IH/\Gamma,Q_8)$ defined in Section~\ref{rem:Fujiki results on AG}; see \cite[Remark 3.12]{Fujiki88}. The action of $Q_8$ on $E^2$ is maximal; \textit{i.e.}, it is not contained in any other finite subgroup of $\Aut(E^2)$ \new{(fixing the origin),} while the action on $\IH/\Gamma$ is the restriction of the action of $BT_{24}$ in Section~\ref{rem:Fujiki results on AG}. The induced actions on $\K_{3}(A)$ of their overgroups give rise to terminalizations with different $b_2$ and $\pi_1$; we distinguish the two cases in \cref{table sing n=2}. 

Accidentally, the terminalizations of $\K_{2}(E^2)/G$ and $\K_{2}(\IH/\Gamma)/G$, with $G_{\circ}=Q_8$, have the same $b_2$ and $\pi_1$ (because the groups $E^2[3]\rtimes Q_8$ and $\IH/\Gamma[3]\rtimes Q_8$ turn out to be abstractly isomorphic to the group of ID 648,730). Therefore, we write them only once in \cref{table n=2}. Mind, however, that the terminalizations are not deformation equivalent. For instance, the terminalizations of the quotients $\K_2(E^2)/Q_8$ and $\K_2(\IH/\Gamma)/Q_8$ have different singularities: 
     \begin{alignat*}{3}
     \Sing\left(\K_2\left(E^2, Q_8\right)\right)& : && \quad 4 Q_8+6 C_4 +29 C_2, \\
    \Sing(\K_2(\IH/\Gamma, Q_8)) &: && \quad 2 Q_8 + 9 C_4 + 28 C_2,
 \end{alignat*}
 where $m \cdot G_x$ means that the terminalization has $m$ isolated singularities with isotropy $G_x$. These singularities are computed in the same way as in Section~\ref{sec:singularitiesterminalization}; we omit the details.
 \end{remark}

\begin{remark}[Binary dihedral group]\label{remark:binarydihedral}
    If $G \simeq C^{2k}_{3} \rtimes BD_{12}$ with $k=0, 1$ or $2$, then there exists a unique conjugacy class of subgroups of order $3$ satisfying \cref{lemma:inducedinvolution}\eqref{l:i-2} and further contained in a subgroup of $G$ isomorphic to $BD_{12}$, splitting the projection  $G \simeq C^{2k}_{3} \rtimes BD_{12} \to BD_{12}$. So if $G_{\circ}\simeq BD_{12}$, by \cref{Prop:excorder3}, the formula in \cref{formulaBetti} acquires a correction term as follows: 
    \[b_2(Y)=\rank\left(L^G\right)+N_2+2N_3-1.\]
\end{remark}
\small\renewcommand*{\arraystretch}{1.1}
\begin{longtable}{c|c|c|c|c|c|c|c}
\caption{Terminalizations of \(\K_2(A)/G\)}\\
\label{table n=2}
 ID 
 & \(G\)
 & \(G_\circ\)
 & \(\rank\)
 & $N_2$
 & $N_3$
 & \(b_2(Y)\)  
 & $\pi_1(Y^{\reg})$ 
 \\\hline \endfirsthead
 ID 
 & \(G\)
 & \(G_\circ\)
 & \(\rank\)
 & $N_2$
 & $N_3$
 & \(b_2(Y)\)  
 &\(\pi_1(Y^{\reg})\) 
 \\\hline \endhead
\rowcolor{backcolour} 2,1
 & \(C_2\)
 & 
 & 
 & 1
 & 0
 & 8
 & $\{1\}$ 
 \\
 \rowcolor{backcolour} 6,1
 & \(C_3\rtimes C_2\)
 & 
 & 
 & 1
 & 0
 & 8
 & \(\{1\}\)
 \\
 \rowcolor{backcolour} 
 18,4
 & \(C_3^2\rtimes_2 C_2\)
 & 
 & 
 & 1
 & 0
 & 8
 & \(\{1\}\)
 \\
 \rowcolor{backcolour} 54,14
 & \(C_3^3\rtimes C_2\)
 & 
 & 
 & 1
 & 0
 & 8
 & \(\{1\}\)
\\
 \rowcolor{backcolour} 162,54
 & \(C_3^4\rtimes C_2\)
 & \multirow{-5}{*}{\(C_2\)}
 & \multirow{-5}{*}{7}
 & 1
 & 0
 & 8
 & \(\{1\}\)
\\ \hline
 3,1
 & \(C_3\)
 &
 &
 & 0
 & 0
 & 5
 & \(C_3\)
 \\
 \rowcolor{backcolour} 3,1
 & \(C_3\)
 & 
 & 
 & 0
 & 1
 & 7
 & \(\{1\}\)
 \\
 9,2
 & \(C_3^2\)
 & 
 & 
 & 0
 & 0
 & 5
 & \(C_3^2\)
 \\
 \rowcolor{backcolour} 9,2
 & \(C_3^2\)
 & 
 & 
 & 0
 & 3
 & 11
 & \(\{1\}\)
 \\
 27,3
 & \(C_3^2\rtimes C_3\)
 & 
 & 
 & 0
 & 0
 & 5
 & \(C_3^2\rtimes C_3\)
 \\
 27,3
 & \(C_3^2\rtimes C_3\)
 & 
 & 
 & 0
 & 1
 & 7
 & \(C_3\)
 \\
 27,5
 & \(C_3^3\)
 & 
 & 
 & 0
 & 0
 & 5
 & \(C_3^3\)
 \\
\rowcolor{backcolour}  27,5
 & \(C_3^3\)
 & 
 & 
 & 0
 & 9
 & 23
 & \(\{1\}\)
 \\
 81,12
 & \(C_3^3\rtimes_2 C_3\)
 & 
 & 
 & 0
 & 0
 & 5
 & \(C_3^3\rtimes_2 C_3\)
 \\
 81,12
 & \(C_3^3\rtimes_2 C_3\)
 & 
 & 
 & 0
 & 3
 & 11
 & \(C_3\)
 \\
 243,37
 & \(C_3^4\rtimes_1 C_3\)
 & \multirow{-11}{*}{\(C_3\)}
 & \multirow{-11}{*}{5}
 & 0
 & 1
 & 7
 & \(C_3^2\)
\\ \hline
4,1
 & \(C_4\)
 &
 &
 & 1
 & 0
 & 6
 & \(C_2\)
\\
36,9
 & \(C_3^2\rtimes C_4\)
 & 
 & 
 & 1
 & 0
 & 6
 & \(C_2\)
\\ 
324,164
 & \(C_3^4\rtimes_4 C_4\)
 & \multirow{-3}{*}{\(C_4\)}
 & \multirow{-3}{*}{5}
 & 1
 & 0
 & 6
 & \(C_2\)
 \\ \hline
 \rowcolor{backcolour} 6,2
 & \(C_6\)
 &
 &
 & 1
 & 1
 & 8
 & \(\{1\}\)
\\ 
\rowcolor{backcolour} 18,3
 & \(C_3\rtimes C_6\)
 & 
 & 
 & 1
 & 2
 & 10
 & \(\{1\}\)
\\
\rowcolor{backcolour} 54,13
 & \(C_3^2\rtimes_4 C_6\)
 & 
 & 
 & 1
 & 5
 & 16
 & \(\{1\}\)
\\
\rowcolor{backcolour} 54,5
 & \(C_3^2\rtimes C_6\)
 & 
 & 
 & 1
 & 1
 & 8
 & \(\{1\}\)
\\
\rowcolor{backcolour} 162,40
 & \(C_3^3\rtimes_4 C_6\)
 & 
 & 
 & 1
 & 2
 & 10
 & \(\{1\}\)
 \\
\rowcolor{backcolour}  486,146
 & \(C_3^4\rtimes_4 C_6\)
 & \multirow{-6}{*}{\(C_6\)}
 & \multirow{-6}{*}{5}
 & 1
 & 1
 & 8
 & \(\{1\}\)
\\ \hline
8,4
 & \(Q_8\)
 &
 &
 & 1
 & 0
 & 5
 & \(C_2^2\)
\\ 
72,41
 & \(C_3^2\rtimes Q_8\)
 & 
 & 
 & 1
 & 0
 & 5
 & \(C_2^2\)
\\ 
648,730
 & \(C_3^4\rtimes Q_8\)
 & \multirow{-3}{*}{\(Q_8\)}
 & \multirow{-3}{*}{4} 
 & 1
 & 0
 & 5
 & \(C_2^2\)
 \\ \hline
12,1
 & \(BD_{12}\)
 &
 &
 & 1
 & 1
 & \rev{6}
 & \(C_2\)
\\ 
 108,37
 & \(C_3^2\rtimes_3 BD_{12}\)
 & 
 & 
 & 1
 & 3
 & \rev{10}
 & \(C_2\)
\\ 
972,NA
 & \(C_3^4\rtimes_* BD_{12}\)
 & \multirow{-3}{*}{\(BD_{12}\)}
 & \multirow{-3}{*}{4}
 & 1
 & 1
 & \rev{6}
 & \(C_2\)
 \\ \hline
\rowcolor{backcolour} 24,3
 & \(BT_{24}\)
 &
 &
 & 1
 & 1
 & 7
 & \(\{1\}\)
\\ 
\rowcolor{backcolour} 216,153
 & \(C_3^2\rtimes BT_{24}\)
 & 
 & 
 & 1
 & 1
 & 7
 & \(\{1\}\)
\\ 
\rowcolor{backcolour} 1944,NA
 & \(C_3^4\rtimes_* BT_{24}\)
 & \multirow{-3}{*}{\(BT_{24}\)}
 & \multirow{-3}{*}{4}
 & 1
 & 1
 & 7
 & \(\{1\}\)
\\ 
\end{longtable}

\begin{longtable}{c|c|c|c|c|c|c}
\caption{Terminalizations of \(\K_3(A)/G\)}\\
\label{table n=3}
 ID
 & \(G\)
 & \(G_\circ\)
 & \(\rank\)
 & $N_2$
 & \(b_2(Y)\)  
 & $\pi_1(Y^{\reg})$
\\\hline \endfirsthead
 ID
 & \(G\)
 & \(G_\circ\)
 & \(\rank\)
 & $N_2$
 & \(b_2(Y)\)  
 &\(\pi_1(Y^{\reg})\) 
\\\hline \endhead
 2,1
 & \(C_2\)
 & \multirow{26}{*}{\(C_2\)}
 & \multirow{26}{*}{7}
 & 0
 & 7
 & \(C_2\)
 \\
\rowcolor{backcolour}  2,1
 & \(C_2\)
 & 
 & 
 & 1
 & 8
 & \(\{1\}\)
 \\
 4,2
 & \(C_2^2\)
 & 
 & 
 & 0
 & 7
 & \(C_2^2\)
 \\
\rowcolor{backcolour}  4,2
 & \(C_2^2\)
 & 
 & 
 & 2
 & 9
 & \(\{1\}\)
 \\
 8,5
 & \(C_2^3\)
 & 
 & 
 & 0
 & 7
 & \(C_2^3\)
 \\
 8,3
 & \(C_4\rtimes C_2\) 
 & 
 & 
 & 1
 & 8
 & \(C_2\)
 \\
 8,3
 & \(C_4\rtimes C_2\) 
 & 
 & 
 & 0
 & 7
 & \(C_4\rtimes C_2\) 
 \\
\rowcolor{backcolour}  8,5
 & \(C_2^3\)
 & 
 & 
 & 4
 & 11
 & \(\{1\}\)
 \\
 16,11
 & \(C_2^3\rtimes C_2\)
 & 
 & 
 & 0
 & 7
 & \(C_2^3\rtimes C_2\)
 \\
 16,11
 & \(C_2^3\rtimes C_2\)
 & 
 & 
 & 2
 & 9
 & \(C_2\)
 \\
 16,14
 & \(C_2^4\)
 & 
 & 
 & 0
 & 7
 & \(C_2^4\)
 \\
\rowcolor{backcolour}  16,14
 & \(C_2^4\)
 & 
 & 
 & 8
 & 15
 & \(\{1\}\)
 \\
 32,46 
 & \((C_2^2\times C_4)\rtimes_5 C_2\)
 & 
 & 
 & 0
 & 7
 & \((C_2^2\times C_4)\rtimes_5 C_2\)
 \\
 32,34
 & \(C_4^2\rtimes_6 C_2\)
 & 
 & 
 & 0
 & 7
 & \(C_4^2\rtimes_6 C_2\)
 \\
 32,51
 & \(C_2^5\)
 & 
 & 
 & 0
 & 7
 & \(C_2^5\)
 \\
\rowcolor{backcolour}  32,51
 & \(C_2^5\)
 & 
 & 
 & 16
 & 23
 & \(\{1\}\)
 \\
 32,46
 & \((C_2^2\times C_4)\rtimes_5 C_2\)
 & 
 & 
 & 4
 & 11
 & \(C_2\)
 \\
 32,34
 & \(C_4^{2} \rtimes_6 C_2\)
 & 
 & 
 & 1
 & 8
 & \(C_2^2\)
 \\
 64,211
 &  \((C_2 \times C_4^2) \rtimes_{11} C_2\)
 & 
 & 
 & 0
 & 7
 &  \((C_2 \times C_4^2) \rtimes_{11} C_2\)
 \\
 64,261
 & \((C_2^3\times C_4)\rtimes_7 C_2\)
 & 
 & 
 & 0
 & 7
 & \((C_2^3\times C_4)\rtimes_7 C_2\)
 \\
 64,211
 & \((C_2 \times C_4^2) \rtimes_{11} C_2\)
 & 
 & 
 &  2
 &  9
 & \(C_2^2\)
 \\
 64,261
 & \((C_2^3\times C_4)\rtimes_7 C_2\)
 & 
 & 
 & 8
 & 15
 & \(C_2\)
 \\
 128,2172
 & \((C_2^2\times C_4^2)\rtimes_{23}C_2\) 
 & 
 & 
 & 0
 & 7
 & \((C_2^2 \times C_4^{2}) \rtimes_{23} C_2\)
 \\
 128,2172
 & \((C_2^2 \times C_4^{2}) \rtimes_{23} C_2\)
 & 
 & 
 & 4
 & 11
 & \(C_2^2\)
 \\
 128,1599
 &  \(C_4^{3} \rtimes_{15} C_2\)
 & 
 & 
 & 0
 & 7
 & \(C_4^{3} \rtimes_{15} C_2\)
 \\
 128,1599
 & \(C_4^{3} \rtimes_{15} C_2\)
 & 
 & 
 & 1
 & 8
 & \(C_2^3\)
 \\
 256,29630
 & \((C_2\times C_4^3)\rtimes_* C_2\)
 & 
 & 
 & 2
 & 9
 & \(C_2^3\)
 \\
 256,29630
 & \((C_2\times C_4^3)\rtimes_* C_2\)
 & 
 & 
 & 0
 & 7
 & \((C_2\times C_4^3)\rtimes_* C_2 \)
 \\
 512,NA
 & \(C_4^4\rtimes C_2\)
 & \multirow{-6}{*}{\(C_2\)}
 & \multirow{-6}{*}{7}
 & 1
 & 8
 & \(C_2^4\)
\\ \hline
 3,1
 & \(C_3\)
 &
 &
 &
 &
 & \(C_3\)
\\
 12,3
 & \(C_2^2\rtimes C_3\)
 & 
 & 
 & 
 & 
 & \(C_2^2\rtimes C_3\)
\\
 48,3
 & \(C_4^2\rtimes C_3\)
 & 
 & 
 & 
 & 
 & \(C_4^2\rtimes C_3\)
\\
 48,50
 & \(C_2^4\rtimes_2 C_3\)
 & 
 & 
 & 
 & 
 & \(C_2^4\rtimes_2 C_3\)
\\
 192,1020
 & \((C_2^2\times C_4^2)\rtimes_3 C_3\)
 & 
 & 
 & 
 & 
 & \((C_2^2\times C_4^2)\rtimes_3 C_3\)
\\
768,1083578
 & \(C_4^4\rtimes C_3\)
  & \multirow{-6}{*}{\(C_3\)}
 & \multirow{-6}{*}{5}
 & \multirow{-6}{*}{0}
 & \multirow{-6}{*}{5}
 &  \(C_4^4\rtimes C_3\)
\\ \hline
 4,1
 &  \(C_4\)
 &
 &
 & 0
 & 5
 & \(C_4\)
 \\
 4,1
 &  \(C_4\)
 & 
 &
 & 1
 & 6
 & \(C_2\)
 \\
 8,2
 &  \(C_2\times C_4\)
 & 
 &
 & 0
 & 5
 & \(C_2\times C_4\)
 \\
 8,2
 &  \(C_2\times C_4\)
 & 
 &
 & 2
 & 7
 & \(C_2\)
 \\
 16,3
 &  \(C_2^2\rtimes C_4\)
 & 
 &
 & 3
 & 8
 & \(C_2\)
 \\
 16,10
 &  \(C_2^2\times C_4\)
 & 
 &
 & 4
 & 9
 & \(C_2\)
 \\
 16,3
 &  \(C_2^2\rtimes C_4\)
 & 
 &
 & 0
 & 5
 & \(C_2^2\rtimes C_4\)
 \\
 16,10
 &  \(C_2^2\times C_4\)
 & 
 &
 & 0
 & 5
 & \(C_2^2\times C_4\)
 \\
 32,22
 &  \(C_2^3\rtimes_2 C_4\)
 & 
 &
 & 0
 & 5
 & \(C_2^3\rtimes_2 C_4\)
 \\
 32,6
 &  \((C_2\times C_4)\rtimes C_4\)
 & 
 &
 & 2
 & 7
 & \(C_2^2\)
 \\
 32,6
 & \((C_2\times C_4)\rtimes C_4\)
 & 
 &
 & 1
 & 6
 & \(C_4\)
 \\
 32,6
 & \((C_2\times C_4)\rtimes C_4\)  
 & 
 &
 & 2
 & 7
 & \(C_4\)
 \\
 32,6
 & \((C_2\times C_4)\rtimes C_4\) 
 & 
 &
 & 0
 & 5
 &  \((C_2\times C_4)\rtimes C_4\)
 \\
 32,22
 &  \(C_2^3\rtimes_2 C_4\)
 & 
 &
 & 6
 & 11
 & \(C_2\)
 \\
 64,34
 &  \(C_4^2\rtimes C_4\)
 & 
 &
 & 0
 & 5
 & \(C_4^2\rtimes C_4\)
 \\
 64,34
 &  \(C_4^2\rtimes C_4\)
 & 
 &
 & 1
 & 6
 & \(C_2^2\rtimes C_2\)
 \\
 64,90
  & \((C_2^2\times C_4)\rtimes_3 C_4\)
 & 
 & 
 & 0
 & 5
 & \((C_2^2\times C_4)\rtimes_3 C_4\)
 \\
 64,90
 & \((C_2^2\times C_4)\rtimes_3 C_4\)
 & 
 & 
 & 2
 & 7
 & \(C_4\)
 \\
 64,90
 & \((C_2^2\times C_4)\rtimes_3 C_4\)
 & 
 & 
 & 4
 & 9
 & \(C_2\times C_2\)
 \\
 64,60
 &  \(C_2^4\rtimes_3 C_4\)
 & 
 &
 & 0
 & 5
 & \(C_2^4\rtimes_3 C_4\)
 \\
 64,60
 & \(C_2^4\rtimes_3 C_4\) 
 & 
 &
 & 10
 & 15
 & \(C_2\)
 \\
 128,856
 &  \((C_2\times C_4^2)\rtimes_9 C_4\)
 & 
 &
 & 0
 & 5
 & \((C_2\times C_4^2)\rtimes_9 C_4\)
 \\
 128,513
 &  \((C_2^3\times C_4)\rtimes_6 C_4\)
 & 
 &
 & 6
 & 11
 & \(C_2^2\)
 \\
 128,513
 &  \((C_2^3\times C_4)\rtimes_6 C_4\)
 & 
 &
 & 4
 & 9
 & \(C_4\)
 \\
 128,513
 &  \((C_2^3\times C_4)\rtimes_6 C_4\)
 & 
 &
 & 0
 & 5
 & \((C_2^3\times C_4)\rtimes_6 C_4\)
 \\
 128,856
 &  \((C_2\times C_4^2)\rtimes_9 C_4\)
 & 
 &
 & 2
 & 7
 & \(C_2^2\rtimes C_2\)
 \\
 256,5681
 &  \((C_2^2\times C_4^2)\rtimes_* C_4\)
 & 
 &
 & 3
 & 8
 & \(C_2^2\rtimes C_2\)
 \\
 256,5681
 &  \((C_2^2\times C_4^2)\rtimes_* C_4\)
 & 
 &
 & 0
 & 5
 & \((C_2^2\times C_4^2)\rtimes_* C_4\)
 \\
 256,1534
 &  \((C_2^2\times C_4^2)\rtimes_* C_4\)
 & 
 &
 & 2
 & 7
 & \(C_2\times C_4\)
 \\
 256,1534
 &  \((C_2^2\times C_4^2)\rtimes_* C_4\)
 & 
 &
 & 4
 & 9
 & \(C_2^3\)
 \\
 512,NA
 &  \((C_2\times C_4^3)\rtimes_* C_4\)
 & 
 &
 & 2
 & 7
 & \(C_2^3\rtimes C_2\)
 \\
 512,NA
 &  \((C_2\times C_4^3) \rtimes_* C_4\)
 & 
 &
 & 1
 & 6
 & \(C_2^2 \rtimes C_4\)
 \\
 1024,NA
 &  \(C_4^4\rtimes_* C_4\)
 & \multirow{-33}{*}{\(C_4\)}
 & \multirow{-33}{*}{5}
 & 1
 & 6
 & \(C_2^4\rtimes_1 C_2\)
 \\ \hline
 6,2
 & \(C_6\)
 &
 &
 & 1
 & 6
 & \(C_3\)
 \\
24,13
 & \(C_2^2\rtimes C_6\)
 & 
 & 
 & 2
 & 7
 & \(C_3\)
 \\ 
 96,72
 & \(C_4^2\rtimes_2 C_6\)
 & 
 & 
 & 1
 & 6
 & \(C_2^2\rtimes C_3\)
 \\ 
 96,229
 & \(C_2^4\rtimes_4 C_6\)
 & 
 & 
 & 6
 & 11
 & \(C_3\)
 \\ 
 384,18223
 & \((C_2^2\times C_4^2)\rtimes_* C_6\)
 & 
 & 
 & 2
 & 7
 & \(C_2^2\rtimes C_3\)
 \\ 
 1536,NA
 & \(C_4^4\rtimes_* C_6\)
 & \multirow{-6}{*}{\(C_6\)}
 & \multirow{-6}{*}{5}
 & 1
 & 6
 & \(C_2^4\rtimes_2 C_3\)
 \\ \hline
 8,4
 & \(Q_8\)
 & \multirow{6}{*}{\(Q_8\)}
 & \multirow{6}{*}{4}
 & 0
 & 4
 & \(Q_8\)
 \\
 8,4
 & \(Q_8\)
 & 
 & 
 & 1
 & 5
 & \(C_2^2\)
 \\
 16,4
 & \(C_2.Q_8\)
 & 
 & 
 & 0
 & 4
 & \(C_2.Q_8\)
 \\
 16,12
 & \(C_2\times Q_8\)
 & 
 & 
 & 0
 & 4
 & \(C_2\times Q_8\)
 \\
 16,4
 & \(C_2.Q_8\)
 & 
 & 
 & 2
 & 6
 & \(C_2\)
 \\
 16,12
 & \(C_2\times Q_8\)
 & 
 & 
 & 2
 & 6
 & \(C_2^2\)
 \\
 32,29
 & \(C_2^2\rtimes Q_8\)
 & 
 & 
 & 0
 & 4
 & \(C_2^2\rtimes Q_8\)
 \\
 32,29
 & \(C_2^2\rtimes Q_8\)
 & 
 & 
 & 4
 & 7
 & \(C_2^2\)
 \\
 64,224
 & \(C_2^3\rtimes_2 Q_8\)
 & 
 & 
 & 5
 & 9
 & \(C_2^2\)
 \\
 64,23
 & \(C_2^3._2 Q_8\)
 & 
 & 
 & 5
 & 9
 & \(C_2^2\)
 \\
 64,23
 & \(C_2^3._2 Q_8\)
 & 
 & 
 & 0
 & 4
 & \(C_2^3._2 Q_8\)
 \\
 64,224
 & \(C_2^3\rtimes_2 Q_8\)
 & 
 & 
 & 0
 & 4
 & \(C_2^3\rtimes_2 Q_8\)
 \\
 128,764
 & \((C_2^2\times C_4)\rtimes Q_8\)
 & 
 & 
 & 1
 & 5
 & \(C_2._2C_2^2\)
 \\
 128,764
 & \((C_2^2\times C_4)\rtimes Q_8\)
 & 
 & 
 & 4
 & 8
 & \(C_2^3\)
 \\
 128,761
 & \(C_2^4\rtimes_2 Q_8\)
 & 
 & 
 & 0
 & 4
 & \(C_2^4\rtimes_2 Q_8\)
 \\
 128,761
 & \(C_2^4\rtimes_2 Q_8\)
 & 
 & 
 & 7
 & 11
 & \(C_2^2\)
 \\
 256,298
 & \((C_2^3\times C_4)._*Q_8\)
 & 
 & 
 & 3
 & 7
 & \(C_2\times C_4\)
 \\
 256,25861
 & \((C_2^3\times C_4)\rtimes_* Q_8\)
 & 
 & 
 & 5
 & 9
 & \(C_2^3\)
 \\
 256,298
 & \((C_2^3\times C_4)._* Q_8\)
 & 
 & 
 & 4
 & 8
 & \(C_2._1C_2^2\)
 \\
 256,25861
 & \((C_2^3\times C_4)\rtimes_* Q_8\)
 & 
 & 
 & 2
 & 6
 & \(C_2._2C_2^2\)
 \\
 512,NA
 & \((C_2^2\times C_4^2)\rtimes_* Q_8\)
 & 
 & 
 & 3
 & 7
 & \(C_2^2\rtimes C_2^2\)
 \\
 512,NA
 & \((C_2^2\times C_4^2)\rtimes_* Q_8\)
 & 
 & 
 & 2
 & 6
 & \(C_2^2.C_2^2\)
 \\
 1024,NA
 & \((C_2\times C_4^3)\rtimes_* Q_8\)
 & 
 & 
 & 2
 & 6
 & \(C_2^3\rtimes_2 C_2^2\)
 \\
 1024,NA
 & \((C_2\times C_4^3)._* Q_8\)
 & 
 & 
 & 1
 & 5
 & \(C_2^3._1C_2^2\)
 \\
 2048,NA
 & \(C_4^4\rtimes_* Q_8\)
 & \multirow{-25}{*}{\(Q_8\)}
 & \multirow{-25}{*}{4}
 & 1
 & 5
 & \(C_2^4\rtimes_1 C_2^2\)
 \\ \hline
 8,4
 & \(Q_8\)
 &
 &
 & 0
 & 4
 & \(Q_8\)
 \\
 8,4
 & \(Q_8\)
 & 
 & 
 & 1
 & 5
 & \(C_2^2\)
 \\
 16,12
 & \(C_2\times Q_8\)
 & 
 & 
 & 2
 & 6
 & \(C_2^2\)
 \\
 16,12
 & \(C_2\times Q_8\)
 & 
 & 
 & 0
 & 4
 & \(C_2\times Q_8\)
 \\
 32,2
 & \(C_2^2._2 Q_8\)
 & 
 & 
 & 0
 & 4
 & \(C_2^2._2 Q_8\)
 \\
 32,47
 & \(C_2^2\times Q_8\)
 & 
 & 
 & 0
 & 4
 & \(C_2^2\times Q_8\)
 \\
 32,2
 & \(C_2^2._2 Q_8\)
 & 
 & 
 & 4
 & 8
 & \(C_2^2\)
 \\
 32,47
 & \(C_2^2\times Q_8\)
 & 
 & 
 & 4
 & 8
 & \(C_2^2\)
 \\
 64,74
 & \(C_2^3\rtimes_1 Q_8\)
 & 
 & 
 & 0
 & 4
 & \(C_2^3\rtimes Q_8\)
 \\
 64,74
 & \(C_2^3\rtimes_1 Q_8\)
 & 
 & 
 & 5
 & 9
 & \(C_2^2\)
 \\
 128,36
 & \(C_2^4._4 Q_8\)
 & 
 & 
 & 0
 & 4
 & \(C_2^4._4 Q_8\)
 \\
 128,1572
 & \(C_2^4\rtimes_6 Q_8\)
 & 
 & 
 & 7
 & 11
 & \(C_2^2\)
 \\
 128,36
 & \(C_2^4._4 Q_8\)
 & 
 & 
 & 7
 & 11
 & \(C_2^2\)
 \\
 128,1572
 & \(C_2^4\rtimes_6 Q_8\)
 & 
 & 
 & 0
 & 4
 & \(C_2^4\rtimes_6 Q_8\)
 \\
 256,3378
 & \((C_2^3\times C_4)\rtimes_* Q_8\)
 & 
 & 
 & 0
 & 4
 & \((C_2^3\times C_4)\rtimes_* Q_8\)
 \\
 256,3378
 & \((C_2^3\times C_4)\rtimes_* Q_8\)
 & 
 & 
 & 2
 & 6
 & \(Q_8\)
 \\
 256,3378
 & \((C_2^3\times C_4)\rtimes_* Q_8\)
 & 
 & 
 & 5
 & 9
 & \(C_2^3\)
 \\
 512,NA
 & \((C_2^2\times C_4^2)._*Q_8\)
 & 
 & 
 & 1
 & 5
 & \(C_4^2\)
 \\
 512,NA
 & \((C_2^2\times C_4^2)\rtimes_* Q_8\)
 & 
 & 
 & 1
 & 5
 & \(C_2\times Q_8\)
 \\
 512,NA
 & \((C_2^2\times C_4^2)._*Q_8\)
 & 
 & 
 & 2
 & 6
 & \(C_2^2\rtimes C_4\)
 \\
 512,NA
 & \((C_2^2\times C_4^2)\rtimes_* Q_8\)
 & 
 & 
 & 4
 & 8
 & \(C_2^4\)
 \\
 1024,NA
 & \((C_2\times C_4^3)\rtimes_* Q_8\)
 & 
 & 
 & 1
 & 5
 & \(C_4^2\rtimes_5 C_2\)
 \\
 1024,NA
 & \((C_2\times C_4^3)\rtimes_* Q_8\)
 & 
 & 
 & 2
 & 6
 & \(C_2^4\rtimes_1 C_2\)
 \\ 
 2048,NA
 & \(C_4^4\rtimes_* Q_8\)
 & \multirow{-24}{*}{\(Q_8\subset BT_{24}\)}
 & \multirow{-24}{*}{4}
 & 1
 & 5
 & \(C_2^4\rtimes_1 C_2^2\)
 \\ \hline
 12,1
 & \(BD_{12}\)
 &
 &
 & 1
 & 5
 & \(S_3\)
 \\ 
48,30
 & \(C_2^2\rtimes BD_{12}\)
 & 
 & 
 & 2
 & 6
 & \(S_3\)
 \\
 192,1495
 & \(C_2^4\rtimes_4 BD_{12}\)
 & \(BD_{12}\)
 & 
 & 5
 & 9
 & \(S_3\)
 \\
 192,185
 & \(C_4^2\rtimes BD_{12}\)
 & 
 & 
 & 1
 & 5
 & \(C_2^2\rtimes S_3\)
 \\
 768,1088649
 & \((C_2^2\times C_4^2)\rtimes BD_{12}\)
 & 
 & 
 & 2
 & 6
 & \(C_2^2\rtimes S_3\)
 \\
3072,NA
 & \(C_4^4\rtimes BD_{12}\)
 & 
 & \multirow{-6}{*}{4}
 & 1
 & 5
 & \(C_2^4\rtimes_3 S_3\)
 \\\hline
 24,3
 & \(BT_{24}\)
 &
 &
 & 1
 & 5
 & \(A_4\)
 \\
 96,3
 & \(C_2^2.BT_{24}\)
 & 
 &
 & 2
 & 6
 & \(A_4\)
 \\
 96,203
 & \(C_2^2\rtimes BT_{24}\)
 & 
 &
 & 2
 & 6
 & \(A_4\)
 \\
 384,4
 & \(C_2^4._* BT_{24}\)
 & 
 &
 & 3
 & 7
 & \(A_4\)
 \\
 384,5868
 & \(C_2^4\rtimes_* BT_{24}\)
 & 
 & 
 & 3
 & 7
 & \(A_4\)
 \\
 1536,NA
 & \((C_2^2\times C_4^2)\rtimes_* BT_{24}\)
 & 
 &
 & 2
 & 6
 & \(C_2^2\rtimes A_4\)
 \\
 1536,NA
 & \((C_2^2\times C_4^2)._* BT_{24}\)
 & 
 &
 & 1
 & 5
 & \(C_2^2.A_4\)
 \\
 6144,NA
 & \(C_4^4\rtimes_* BT_{24}\)
 & \multirow{-8}{*}{\(BT_{24}\)}
 & \multirow{-8}{*}{4}
 & 1
 & 5
 & \(C_4^2\rtimes A_4\)
 \\
\end{longtable}
\normalsize

\subsection{Technical digression: Affine actions} In this technical section, we show that all actions giving a terminalization with simply connected regular locus are affine. The possible groups arising are listed in Tables~\ref{table n=2} and~\ref{table n=3} and satisfy one of the assumptions \eqref{l:Gia-1}--\eqref{l:Gia-5} in \cref{lem:G is affine}. 

\begin{lemma}[Affine groups] \label{lem:G is affine}
    Let $G$ be a finite group of induced symplectic automorphisms of\, $\K_n(A)$. Then $G$ is conjugate by an element of $A[n+1]$ to the affine subgroup $G_\tr \rtimes G_\circ$ if
    \begin{enumerate}
        \item\label{l:Gia-1} $n=2$ and $G_\circ \simeq C_2$,
        \item\label{l:Gia-2} $n=2$ and $G_\circ \simeq C_6$,
        \item\label{l:Gia-3} $n=2$ and $G_\circ \simeq BT_{24}$,
        \item\label{l:Gia-4} $n=2$, $G_\circ \simeq C_3$ and $N_3 \neq 0$,
        \item\label{l:Gia-5} $n=3$, $G_\circ \simeq C_2$ and $N_2 \neq 0$.
    \end{enumerate}
\end{lemma} 

\begin{proof} We follow the notation of Section~\ref{rem:Fujiki results on AG}.
  It suffices to prove that, up to conjugation in $A[n+1]$, the quotient $G \to G_{\circ}$ splits. If $\{g_i\}$ are generators of $G_{\circ} \subset \SL(\Lambda)$, and $\tau_{\alpha_i} g_i$ is a lift of $g_i$ in $G \subset A[n+1] \rtimes \SL(\Lambda)$, we show that there exists an $\alpha \in A[n+1]$ such that the $\tau_{\alpha}(\tau_{\alpha_i} g_i)\tau_{-\alpha} \subset \SL(\Lambda)$
    generate $G_{\circ}$; \textit{i.e.}, $G_{\circ} \subset \tau_{\alpha} G \tau_{-\alpha}$ splits the quotient $\tau_{\alpha} G \tau_{-\alpha} \to G_{\circ}$.

  \eqref{l:Gia-1}~ Let $\tau_\alpha(-\id) \in G$ be a lift of $-\id \in C_2$. Conjugating by $\tau_\alpha$, we write
  $$\tau_\alpha (\tau_\alpha (-\id)) \tau_{-\alpha}=-\id \in \tau_\alpha G \tau_{-\alpha}.$$

\eqref{l:Gia-2}~  Let $\tau_\alpha g_6 \in G$ be a lift of $g_6 \in C_6$. Observe that $(\id - g_6)$ is an automorphism of $A[3]$, so we can pick a $\beta$ such that $\beta - g_6(\beta)= -\alpha$, which gives $$\tau_\beta (\tau_\alpha g_6) \tau_{-\beta}= g_6 \in \tau_\beta G \tau_{-\beta}.$$

\eqref{l:Gia-3}~ Let $\tau_\alpha t$ \new{ be a lift of $t$ to $G$.} 
        Observe that $(\id - t)$ is an automorphism of $A[3]$, so we can pick a $\gamma$ such that $\gamma - t(\gamma)= -\alpha$, which implies $$\tau_\gamma (\tau_\alpha t) \tau_{-\gamma}=t \in \tau_{\gamma}G \tau_{-\gamma}.$$ 
        Let $\tau_\beta r$ \new{ be a lift of $r$ to $\tau_{\gamma}G \tau_{-\gamma}$.} 
        As $t^3=-\id \in \tau_{\gamma}G \tau_{-\gamma}$ and $(\tau_\beta r)^3=\tau_{\beta + r(\beta) + r^2(\beta)}(-\id) \in \tau_{\gamma}G \tau_{-\gamma}$, we write $$\tau_{\beta + r(\beta) + r^2(\beta)} ( \tau_\beta r) \tau_{-\beta -r(\beta) - r^2(\beta)}=r \in \tau_{\gamma}G \tau_{-\gamma}.$$

        \eqref{l:Gia-4}~ As $N_3 \neq 0$, by \cref{lemma:inducedinvolution}\eqref{l:i-2}, there exists an
        $$\alpha \in \Pi_{g_3} \coloneqq \{ \alpha \in A[3] \,|\, g_3(\alpha)=\alpha\}= \textrm{Im}(\id - g_3)$$
        such that $\tau_\alpha g_3 \in G$. Pick a $\beta$ such that $\beta- g_3(\beta)=-\alpha$. Then $$\tau_\beta (\tau_\alpha g_3) \tau_{-\beta}= g_3 \in \tau_\beta G \tau_{-\beta}.$$

        \eqref{l:Gia-5}~ As $N_2 \neq 0$, by \cref{lemma:inducedinvolution}\eqref{l:i-1},  there exists an $\alpha \in A[2]$ such that $\tau_\alpha (-\id) \in G$. Pick  a $\beta$ such that $2\beta=\alpha$. Then
\begin{equation*}\pushQED{\qed}
        \tau_\beta (\tau_\alpha (-\id)) \tau_{-\beta}=-\id \in \tau_\beta G \tau_{-\beta}.\qedhere \popQED
\end{equation*}
\renewcommand{\qed}{}    
\end{proof}

\subsection{Terminalizations with simply connected regular locus}\label{sec:terminalizationfour} 

\cref{table sing n=2,table sing n=3} are refinements of \cref{table n=2,table n=3} for terminalizations $Y$ of quotients with simply connected regular locus. 
\begin{itemize}
    \item We list the group ID, the alias of $G$, the integers \(N_i\) and $b_2(Y)$ as in \cref{table n=2,table n=3}. 
      
    \item We list the numbers $a_k$ of isolated singularities in \(Y\) of analytic type \(\IA^{2n}/\tfrac{1}{k}(1,-1, \ldots, 1,-1)\); see \cref{defn:sing type ai}, the computations in \cref{sec:sing computation n=2} for $n=2$, and in \cref{prop:K3A terminalization nonsmooth} for $n=3$.

    \item If $n=3$, we list the number $s_2$ of surfaces of $Y$ with general transversal singularities of type $\tfrac{1}{2}(1,1,-1,-1)$; see \cref{prop:K3A terminalization nonsmooth}.
    \item If $n=2$, we list the topological Euler characteristic $\chi$, the fourth Betti number $b_4$ and the Chern numbers $c_4$ and $c_2^2$ of $Y$, which are functions of $b_2(Y)$ and $a_k$ as follows:\footnote{Recall that in our case,  $b_3(Y)=0$ by \cref{prop:3bettinumberKum}.}
    \begin{align*}
\quad b_4(Y)
    & = 10 b_2(Y) - b_3(Y) + 46 - a_2 - 2a_3 - 3a_4
    && \quad \text{by \cite[Proposition 3.6]{Fu-Menet}},\\
\quad \chi(Y)
    & = 12 b_2(Y)- 3b_3(Y) + 48 - a_2 - 2a_3 - 3a_4 
    && \quad \text{by \cite[Proposition 3.6]{Fu-Menet}},\\
\quad c_4(Y)
   & = \chi(Y) - \frac{a_2}{2} - \frac{2a_3}{3} -\frac{3a_4}{4}
   && \quad \text{by \cite[Theorem 2.14]{Blache1996}}, \\ 
\quad c_2^2(Y)
   & = \frac{1}{3}c_4(Y) + 720 - 240\left(\frac{a_2}{2^5} + \frac{2a_3}{27} + \frac{9a_4}{2^6}\right)
   && \quad \text{by \cite[]{Fu-Menet}}. 
\end{align*}
Note that we can apply the previous identities since our terminalizations have quotient singularities; see \cref{cor:quotientsingularities}. 

\item If $Y$ is deformation equivalent to a known IHS variety, we write the latter in the last column; this analysis follows from \cref{prop:mainsection} for \(n=2\) and \cref{thm:smooth term} for \(n=3\). The notation $\K_n(A,G)$ stands for a \new{projective} terminalization of $\K_n(A)/G$, while $S(G)^{[n]}_\theta$ is the Fujiki variety; see Notation~\ref{notationKummer} and \cref{defb:Fujiki variety}. 
The question mark in the correspondence of \(G=BT_{24}\) indicates that it shares the singularities and topological invariants of $S(S_3^2 \rtimes C_2)^{[2]}_{\id}$, but we could not decide whether the two are deformation equivalent; see also \cref{rmk:open}. Note that $\K_2(A,C_2)$ is studied in \cite{Kapfer-Menet} and also appeared in \cite{Fu-Menet}. In all other cases, we declare the deformation type to be \emph{new}. 
\end{itemize}

\small
\renewcommand{\arraystretch}{1.5}
\begin{longtable}{c|c|c|c|c|c|c|c|c|c|c|c||c}
\caption{Terminalizations of \(\K_2(A)/G\) with simply connected regular locus 
}\\
\label{table sing n=2}
 ID 
 & \(G\)
 & $N_2$
 & $N_3$
 & \(b_2\)  
 & $a_2$
 & $a_3$
 & $a_4$
 & $b_4$
 & $\chi$
 & $c_4$
 & $c_2^2$
 &
 \\\hline
 2,1
 & \(C_2\)
 & \multirow{5}{*}{1}
 & 0
 & 8
 & 36
 & 0
 & 0
 & 90
 & 108
 & 90
 & 480
 & \cite{Kapfer-Menet}
 \\
 6,1
 & \(C_3\rtimes C_2\) 
 & 
 & 0
 & 8
 & 36
 & 13
 & 0
 & 64
 & 82
 & 166/3
 & 712/3
 & new 
 \\
 18,4
 & \(C_3^2\rtimes_2 C_2\)
 & 
 & 0
 & 8
 & 36
 & 16
 & 0
 & 58
 & 76
 & 142/3
 & 544/3
 & new
 \\
 54,14
 & \(C_3^3\rtimes C_2\)
 & 
 & 0
 & 8
 & 36
 & 13
 & 0
 & 64
 & 82
 & 166/3
 & 712/3
 & $\K_2(A, S_3)$
\\
 162,54
 & \(C_3^4\rtimes C_2\)
 & 
 & 0
 & 8
 & 36
 & 0
 & 0
 & 90
 & 108
 & 90
 & 480
 & $\K_2(A, C_2)$ 
\\ \hline
 3,1
 & \(C_3\)
 & \multirow{3}{*}{0}
 & 1
 & 7
 & 0
 & 12
 & 0
 & 92
 & 108
 & 100
 & 540
 & $S(C^2_3)^{[2]}_{-\id}$
 \\
 9,2
 & \(C_3^2\)
 &
 & 3
 & 11
 & 0
 & 15
 & 0
 & 126
 & 150
 & 140
 & 500
 & $S(C_3)^{[2]}_{-\id}$
 \\ 
 27,5
 & \(C_3^3\)
 & 
 & 9
 & 23
 & 0
 & 0
 & 0
 & 276
 & 324
 & 324
 & 828
 & $S^{[2]}$\\
 \hline
 6,2
 & \(C_6\)
 & \multirow{6}{*}{1}
 & 1
 & 8
 & 28
 & 12
 & 0
 & 74
 & 92
 & 70
 & 320
 & $S(C_3 \rtimes S_3)_{\id}^{[2]}$
\\ 
18,3
 & \(C_3\rtimes C_6\)
 & 
 & 2
 & 10
 & 28
 & 12
 & 0 
 & 94
 & 116
 & 94
 & 328
 & $S(S_3)^{[2]}_{\id}$
\\
54,13
 & \(C_3^2\rtimes_4 C_6\)
 & 
 & 5
 & 16
 & 28
 & 0
 & 0
 & 178
 & 212
 & 198
 & 576
 & $S(C_2)^{[2]}_{\id}$
\\
54,5
 & \(C_3^2\rtimes C_6\)
 & 
 & 1
 & 8 
 & 28
 & 20
 & 0
 & 58
 & 76
 & 146/3
 & 512/3
 & $S(C_3 \rtimes S_3)^{[2]}_{(-\id,\id)}$
\\
162,40
 & \(C_3^3\rtimes_4 C_6\)
 & 
 & 2
 & 10
 & 28
 & 12
 & 0
 & 94
 & 116
 & 94
 & 328
 & $S(S_3)^{[2]}_{\id}$
 \\
 486,146
 & \(C_3^4\rtimes_4 C_6\)
 & 
 & 1
 & 8
 & 28
 & 12
 & 0
 & 74
 & 92
 & 70
 & 320
 & $S(C_3 \rtimes S_3)_{\id}^{[2]}$
 \\ \hline
24,3
 & \(BT_{24}\)
 & \multirow{3}{*}{1}
 & 1
 & 7
 & 20
 & 12
 & 3
 & 63
 & 79
 & 235/4
 & 275
 & $S(S_3^2 \rtimes C_2)^{[2]}_{\id}$ ?
\\ 
216,153
 & \(C_3^2\rtimes BT_{24}\)
 & 
 & 1
 & 7
 & 20
 & 16
 & 3
 & 55
 & 71
 & 577/12
 & 601/3
 & new
\\ 
1944,NA
 & \(C_3^4\rtimes BT_{24}\)
 & 
 & 1
 & 7
 & 20
 & 12
 & 3
 & 63
 & 79
 & 235/4
 & 275
 & $\K_2(A, BT_{24})$
\\ 
\end{longtable}
\renewcommand{\arraystretch}{1}
\normalsize
\small
\renewcommand{\arraystretch}{1.5}
\begin{longtable}{c|c|c|c|c|c||c}
    \caption{Terminalizations of \(\K_3(A)/G\) with simply connected regular locus} 
    \\ \label{table sing n=3}
  \text{ID}
  & \(G\)
  & $N_2$
  & \(b_2\)  
  & $a_2$
  & $s_2$
  &
  \\ \hline  
  \(2,1\) 
  & \(\langle -\id \rangle\)
  & 1
  & 8
  & 140
  & 0
  & \(S(C_2^4)^{[3]}\) 
  \\ 
  \(4,2\)
  & \(C_2\times \langle -\id \rangle\)
  & 2
  & 9
  & 112
  & 7
  &\(S(C_2^3)^{[3]}\)
  \\ 
  \(8,5\)
  & \(C_2^2\times \langle -\id \rangle\)
  & 4
  & 11
  & 64
  & 18
  & \(S(C_2^2)^{[3]}\)
  \\ 
  \(16,14\)
  & \(C_2^3\times \langle -\id \rangle\)
  & 8
  & 15
  & 0
  & 28
  & \(S(C_2)^{[3]}\)
  \\ 
  \(32,51\)
  & \(C_2^4\times \langle -\id \rangle\)
  & 16
  & 23
  & 0
  & 0
  & \(S^{[3]}\)
  \\ 
\end{longtable}
\normalsize
\renewcommand{\arraystretch}{1.0}

\subsection{Singularities of terminalizations of quotients of \(\boldsymbol{\K_3(A)}\)}\label{sec:terminalization6} We determine the singularities of the terminalizations of $\K_3(A)/G$ with simply connected regular locus.

\begin{proposition}\label{prop:K3A terminalization nonsmooth}
Let $X=\K_3(A)$ and $G = C_2^i \times \langle -\id \rangle$ for $0 \leq i \leq 4$. 
Then the singular locus of\, $Y$ consists only of $a_2$ isolated points of type
\(\IA^{6}/\tfrac{1}{2}(1,1,1,1, 1,1)\) and $s_2$ surfaces with general transversal singularities $\IA^{4}/\tfrac{1}{2}(1,1,1,1)$, where 
\begin{align*}
    a_2 & =4\left(42-7\cdot 2^i + \frac{1}{3}(2^i-1)(2^i-2)\right), \\
    s_2 & =(2^i-1)(8-2^{i-1}).
\end{align*}
\end{proposition}

\begin{proof} The fixed loci of an automorphism in $G$ are computed for instance in \cite[Lemma 2.10, Proposition~2.12]{Floccari2022}.
\begin{enumerate}
\item Any nontrivial translation $\tau_\alpha \in A[2]$ fixes eight $\K3$ surfaces $V_{\alpha, \theta}$ in $\K_3(A)$, where
$$
\epsilon(V_{\alpha, \theta})= \{ 
[(x, x+\alpha, -x + \theta, -x+\alpha + \theta)] \,|\, x \in A
\}
$$ with $\theta \in A[2]$ and $V_{\alpha, \theta}=V_{\alpha, \theta + \alpha}$. 
\item Any involution $\tau_\alpha(-\id) \in A[2] \times  \langle -\id \rangle$ fixes the fourfold $W_{\alpha}$ of \(\K3^{[2]}\)-type, where
$$
\epsilon(W_\alpha) = \{ 
[(x, -x + \alpha, y, -y + \alpha)] \,|\, x,y \in A
\},
$$ 
and 140 isolated fixed points of the form
$$
[(\varepsilon_1, \varepsilon_2, \varepsilon_3, -\varepsilon_1-\varepsilon_2 -\varepsilon_3)] 
\quad \text{ with $2\varepsilon_i=\alpha$ and the $\varepsilon_i$ pairwise disjoint.}
$$
\end{enumerate}
Observe that these fixed loci satisfy the following intersection rules: 
\begin{itemize}
\item Two fourfolds $W_\alpha$ and $W_\beta$ intersect along the surface $V_{\alpha+\beta, \alpha}=V_{\alpha + \beta, \beta}$.
\item Three fourfolds $W_\alpha, W_\beta$ and $W_\gamma$ intersect in four points of the form
$$
[(\varepsilon, \varepsilon+\alpha+\beta, \varepsilon+ \alpha + \gamma, \varepsilon+ \beta + \gamma)] \quad \text{ with $2\varepsilon=\alpha+ \beta + \gamma$.}
$$
Thus, $W_\alpha \cap W_\beta \cap W_\gamma$ consists of $4$ of the $140$ isolated points fixed by $\tau_{\alpha + \beta + \gamma}(-\id)$.
\end{itemize}

Let $z$ be an isolated point of $\Fix(\tau_{\alpha}(-\id))$.
Then one of the following cases holds:
\begin{enumerate}[label=(\roman*),ref=\roman*]
\item\label{1} $G_z=\langle \tau_\alpha(-\id) \rangle$, and $z$ corresponds to a singular point of $Y$. 
\item\label{2} $G_z=\langle \tau_\alpha(-\id), \tau_\beta(-\id) \rangle \simeq C_2^2$, and \[z \in W_\beta \cap V_{\alpha + \beta, \theta} \,\left(=W_\beta \cap W_{\theta} \cap W_{\alpha+\beta+\theta}\right)\] 
for some $\theta \neq \alpha, \beta$ in $A[2]$, equivalently $\tau_\theta(-\id) \notin G$.
Indeed, $z$ is an isolated fixed point only for the involution $\tau_\alpha(-\id)$ as can be seen
by writing $2 \varepsilon_i =\alpha$, so $z$ must lie in $W_{\beta}$. 
Locally at $q(z)$,  the terminalization $q\colon Y \to X/G$ is isomorphic to 
\[\mathrm{Tot}(T^*_{\mathbb{P}^1}) \times \IA^4/ \tfrac{1}{2}(1,1, 1,1) \lra \IA^2/\tfrac{1}{2}(1,1) \times \IA^4/\tfrac{1}{2}(1,1, 1,1); \]
thus it contains only a singular surface with general transversal singularities of type $\tfrac{1}{2}(1,1, 1,1)$. 
\item\label{3} $G_z=\langle \tau_\alpha(-\id), \tau_\beta(-\id), \tau_\gamma (-\id) \rangle \simeq C_2^3$, and \[z \in W_{\beta} \cap W_{\gamma} \cap W_{\alpha + \beta + \gamma}.\] Locally at $q(z)$, the quotient $X/G$ is isomorphic to the triple product of a canonical surface singularity of type $A_1$, which admits a symplectic resolution.  
\end{enumerate}
\small
\begin{center}
\begin{tikzpicture}
\draw[thick] plot coordinates {(0,0) (-1.4,-0.5) (0,-1) (1.4,-0.5) (0,0) (0,1.25)};
\draw[dashed] plot coordinates {(1.4,-0.5) (1.4,0.75) (0,1.25) (-1.4,0.75) (-1.4,-0.5)};

\fill (0,0) circle[radius=2pt] node[above right]{\(z\)};
\fill (0,1.25)  node[above]{\(V_{\alpha+\beta,\theta}\)};
\fill (0,-0.5) node{\scriptsize{\(W_{\beta}\)}};
\fill (-1.4,0.5) node[right]{\scriptsize{\(W_{\alpha+\beta+\theta}\)}};
\fill (0,0.5) node[right]{\hspace{4mm}\scriptsize{\(W_{\theta}\)}};

\fill (0,-1.5) node[below]{(ii) \(G_z\simeq C_2^2\)};

\end{tikzpicture} 
\hspace{2.5cm}
\begin{tikzpicture}
\draw[thick] plot coordinates {(1.4,-0.5) (0,0)};
\draw[thick] plot coordinates {(0,1.25) (-1.4,0.75) (-1.4,-0.5) (0,-1) (1.4,-0.5) (1.4,0.75) (0,1.25) (0,0) (-1.4,-0.5)};

\fill (0,0) circle[radius=2pt] node[above]{\hspace{3mm}\(z\)};
\fill (0,-0.5) node{\scriptsize{\(W_{\beta}\)}};
\fill (-1.4,0.5) node[right]{\scriptsize{\(W_{\alpha+\beta+\gamma}\)}};
\fill (0,0.5) node[right]{\hspace{4mm}\scriptsize{\(W_{\gamma}\)}};
\fill (0,1.25)  node[above]{\(V_{\alpha+\beta,\gamma}\)};
\fill (-1.4,-0.5) node[left]{\(V_{\alpha+\gamma,\beta}\)};
\fill (1.4,-0.5) node[right]{\(V_{\beta+\gamma,\beta}\)};

\fill (0,-1.5) node[below]{(iii) \(G_z\simeq C_2^3\)};

\end{tikzpicture}
\end{center}
\normalsize

As a result, $Y$ has only (quotient) singularities of the types appearing in the statement of \cref{prop:K3A terminalization nonsmooth}, with the invariants $a_2$ and $s_2$ as follows: 
\begin{align*}
    s_2 
    & = \# \text{ surfaces in}\Sing(X/G)\\
    & = (\# \text{ surfaces in }X \text{ fixed by a translation}) - (\# \text{ such surfaces lying on a fixed fourfold})\\
    & = 8 \cdot (\# \text{ nontrivial translations in }G) - (\# \text{  intersection of two fixed fourfolds})\\
    & = 8( 2^i-1) - {2^i \choose 2}, \\
    a_2
    & = \# \text{ isolated singular points of }X/G
    \\
    & = \big(\# \text{ isolated points in}\Fix(\tau_\alpha(-\id))\text{ for some }\tau_\alpha(-\id) \in G \\
    & \hphantom{= (}\; \text{and not lying} 
    \text{ on }W_{\beta}\text{ for any }\beta \in G_{\tr}\big)/(\# \text{orbits of such points})\\
    & = 2^i 
    \cdot \big( (\# \text{ points fixed by $\tau_\alpha(-\id)$ and not lying on $W_\alpha$})\\
    & \hphantom{=2^i\cdot\big(}\; - (\# \text{ such points lying on }W_{\beta}\text{ for some }\beta \in G_{\tr})
    \big)
    /2^i.
\end{align*}

The points $z \in \Fix(\tau_\alpha(-\id))$ lying on $W_\beta \setminus W_\alpha$ for $\beta \neq \alpha$ are in particular fixed by $\tau_{\alpha + \beta}$; hence they lie on one of the seven fixed surfaces $V_{\alpha+\beta, \theta}= W_{\theta} \cap W_{\alpha+\beta+\theta}$ for $\theta \neq \alpha, \beta$ in $A[2]$. Thus, for each choice of $\beta \in G_\tr \setminus \{ \alpha \}$, there are $4 \cdot 7$ such points $z \in W_{\beta} \cap W_{\theta} \cap W_{\alpha + \beta + \theta}$. 
However, note that when $\theta$ is in $G_{\tr}$, we count the same point $z$ three times. Indeed, if $\theta \in G_\tr$, then $z$ is a point of type~\eqref{3}, and it lies on three fixed fourfolds $W_\beta$, $W_\theta$ and $W_{\alpha+\beta+\theta}$. So,

\begin{align*}
    (\# 
    & \text{ isolated points fixed by $\tau_\alpha(-\id)$ and lying on }W_{\beta}\text{ for some }\beta \in G_{\tr})\\
    & = 4 \cdot 7 \cdot (\# \text{ translations in 
    $G_\tr \setminus \{ \alpha \}$}) 
    - 4 \cdot 2 \cdot 
    (\# \text{ choices of $\{\beta, \theta\}$ in $G_\tr \setminus \{ \alpha \}$})/3\\
    & = 4 \cdot 7 \cdot (2^i-1) - 4 \cdot 2 \cdot {2^i-1 \choose 2} \frac{1}{3}, 
\end{align*}
and we conclude
\begin{equation*}\pushQED{\qed}
a_2 = 140 -  4 \cdot 7 \cdot (2^i-1) + 4 \cdot 2 \cdot {2^i-1 \choose 2} \frac{1}{3}.
\qedhere \popQED
	\end{equation*}
\renewcommand{\qed}{}     
\end{proof}

\begin{corollary}[Quotient singularities]\label{cor:quotientsingularitiestext}
    Any \new{projective} terminalization of a quotient of\, $\K_2(A)$ or $\K_3(A)$ by a finite group of induced symplectic automorphisms has quotient singularities.
\end{corollary}
\begin{proof}
By direct inspection of the singularities of $K_{n}(A)/G$, where $G$ is one of the groups in Table~\ref{table sing n=2} (more precisely, because of the analysis of the local model of terminalizations in \cref{prop: sing after terminalization} and the projectivity of their gluing explained in \cref{remark:projectiveterminalization}) and \cref{prop:K3A terminalization nonsmooth}, we see that the projective terminalizations with simply connected regular locus have quotient singularities. The result in general follows since any other \new{projective} terminalization is deformation equivalent to a quotient of a terminalization with simply connected regular locus by \cref{prop:fundamentalgroup}.
\end{proof}

\section{Singularities of quotients of generalized Kummer fourfolds} \label{sec:sing computation n=2}

In this section, we analyze the singularities of the quotients $\K_2(A)/G$ (see \cref{table sing n=2}) and describe local models for their terminalizations (see Section~\ref{subsec:local models}). A result of Namikawa grants that the singularity type of global terminalizations agrees with that of the local models; see Section~\ref{remark:projectiveterminalization}. One of the difficulties, compared to similar previous investigations, is that our groups $G$ are not necessarily cyclic and may contain several translations. This implies that the intersections and combinatorics of fixed loci of all elements $g \in G$  make the analysis technically more challenging. To navigate this complexity, we display the configuration of the singularities of $\K_2(A)/G$ in some schematic pictures in Section~\ref{sec:singularitiesterminalization}: The diagrams clarify the relative position and isotropy of each stratum of the singular locus.


\revi{We pursue the desired analysis of the singularities of $\K_2(A)/G$ and their terminalizations, as follows:}
\begin{itemize}
    \item We describe the $G$-fixed locus on $\K_2(A)$; see \cref{table Gfixed}. This can be done in terms of the $G$-fixed locus of $A^3_{0}$, except where the Hilbert--Chow morphism $\epsilon \colon \K_2(A) \to A^3_{0}$ is not an isomorphism, especially on the punctual Hilbert scheme $\epsilon^{-1}(0) \simeq \IP(1,1,3)$; see Section~\ref{subsec:fixed point punctual Hilbert scheme}. 
    \item In Section~\ref{sec:singularitiesterminalization}, we provide an algorithm that extracts the configuration of the singularities of $\K_2(A)/G$ from the intersection theory and combinatorics of the $G$-fixed loci. We run this explicitly for the new deformation types that appear in \cref{table sing n=2} and represent the singularities in a diagram (see Figure~\ref{fig:configuration 216} and Section~\ref{diagram24,3}). 
    \item We provide local models for the singularities of $\K_2(A)/G$, describe the singularities of a local terminalization and show that they can be glued to a projective terminalization of $\K_2(A)/G$ by results of Namikawa; see Sections~\ref{sec:terminalization} and~\ref{remark:projectiveterminalization}.
\end{itemize}

\subsection{Projective terminalizations}\label{remark:projectiveterminalization}
Gluing together local analytic models of terminalizations may lead to a non-projective global terminalization, as in \cite[Proposition 13.3]{Fujiki1983}. One may wonder whether the local singularities of a global projective terminalization differ from that of an arbitrary local model. This is not our case. 

The only local terminalization that is not obtained by blowing up the reduced singular locus, and which may potentially affect the projectivity of the global terminalization, corresponds to a singularity with isotropy $BT_{24}$, namely singularity \eqref{item8} in \cref{prop: sing after terminalization}. Two local symplectic resolutions of such a singularity are described in \cite{LS2012}. They are blowups of local analytic Weil divisors followed by the blowup of the singular locus of the previous blowup. In particular, the exceptional locus is irreducible. We do not know whether the same sequence of birational transformations can be carried out globally on $\K_2(A)/G$, namely if the effective Weil divisor extends (at least its class in the class group does). Nevertheless, by \cref{prop:Qfactorialterminalization}, any projective terminalization of $X/G$ should be locally isomorphic to one of the two symplectic resolutions obtained by Lehn and Sorger. In fact, Bellamy showed that these are the only symplectic resolutions of such a quotient singularity; see \cite[Section~4.3]{Bellamy2016}. We conclude that, in our case, a projective terminalization of $\K_2(A)/G$ can indeed be obtained by gluing local models of terminalization, which are listed in \cref{prop: sing after terminalization}.

\subsection{Local models of some symplectic singularities and their terminalizations} \label{subsec:local models}

\begin{lemma}\label{prop: sing after terminalization}
    Let $G$ be a finite group with a faithful complex symplectic representation $V$ of dimension $4$.
    \begin{enumerate}
        \item \label{item1} If\, $G \simeq C_{k}$ for $k=2,3,4$ or $6$ and $V/G$ has an isolated $($terminal\,$)$ singularity, then
        \[V/G \simeq \IA^4/\tfrac{1}{k}(1,1,-1,-1).\]
        \item \label{item2} If\, $G \simeq C_{4}$ and $\Sing(V/G)$ is an irreducible surface generically of transversal $A_1$-singularities, then 
        \[V/G \simeq \IA^4/\tfrac{1}{4}(1,-1,2,2),\]
        and a terminalization of $V/G$ has two singularities of type $\tfrac{1}{2}(1,1,1,1)$; \textit{i.e.}, $a_2=2$. 
        \item \label{item3} If\, $G \simeq C_{6}$ and $\Sing(V/G)$ is an irreducible surface generically of transversal $A_1$-singularities, then 
        \[V/G \simeq \IA^4/\tfrac{1}{6}(1,-1,3,3),\]
        and a terminalization of $V/G$ has two singularities of type $\tfrac{1}{3}(1,1,-1,-1)$; \textit{i.e.}, $a_3=2$.
        \item \label{item4} If\, $G \simeq C_{6}$ and $\Sing(V/G)$ is an irreducible surface generically of transversal $A_2$-singularities, then 
        \[V/G \simeq \IA^4/\tfrac{1}{6}(1,-1,2,2),\]
        and a terminalization of $V/G$ has three singularities of type $\tfrac{1}{2}(1,1,1,1)$; \textit{i.e.}, $a_2=3$.
        \item \label{item4.5} If\, $G \simeq C_{6}$ and $\Sing(V/G)$ consists of two surfaces generically of transversal $A_1$- and $A_2$-singularities, respectively, then 
        \[V/G \simeq \IA^2/C_2 \times \IA^2/C_3.\]
         \item \label{item5}If\, $G = C_3 \times C_3$, then
       \[V/G \simeq \IA^2/C_3 \times \IA^2/C_3.\]
        \item \label{item6}If\, $G = S_3$ and $V/G$ has singularities in codimension $2$, then
       \[V/G \simeq \mathfrak{h} \oplus \mathfrak{h}^*/S_3,\]
        where $S_3$ acts by permutation on the hyperplane $\mathfrak{h}=\{x \in \mathbb{A}^3 \,|\,  \sum_{i}x_i=0 \}$.
         \item \label{item7}If\, $G = C_3 \times S_3= C_3^2 \rtimes C_2$ and $\Sing(V/G)$ consists of two surfaces generically of transversal $A_1$- and $A_2$-singularities, respectively,  then
       \[V/G \simeq (\mathfrak{h} \otimes \chi) \oplus (\mathfrak{h}\otimes \chi)^*/C_3 \times S_3,\]
        where $\mathfrak{h}$ is the irreducible $2$-dimensional representation lifted from $S_3$ and $\chi$ is a nontrivial character of order $3$.
        \item \label{item8} If\, $G=BT_{24}$ and $\Sing(V/G)$ is an irreducible surface generically of transversal $A_2$-singularities, then
        \[V/G \simeq \rho \oplus \rho^*/BT_{24},\]
        where $\rho$ is the $($unique up to dual\,$)$ irreducible $2$-dimensional representation of $BT_{24}$ generated by complex reflections of order $3$.
    \end{enumerate}
    The quotients $V/G$ as in \eqref{item4.5}, \eqref{item5}, \eqref{item6}, \eqref{item7} and \eqref{item8} all admit a smooth terminalization.
\end{lemma}

\begin{proof}
  The symplectic form $\omega_V$ induces a $G$-equivariant isomorphism $V \simeq V^*$, and $W$ is an irreducible subrepresentation of $V$ if and only if its dual $W^*$ is so too. Therefore, $V$ decomposes in irreducible representations in one the following ways:
\begin{itemize}
\item $\chi_1\oplus \chi_1^*\oplus \chi_2 \oplus \chi_2^*$ if and only if $G$ is abelian,
\item $\chi_1\oplus \chi_1^* \oplus \rho$ with $\rho \simeq \rho^*$ symplectic,
\item $\rho \oplus \rho^*$,
\item $V$,
\end{itemize}
where $\chi_i$ and $\rho$ are irreducible $G$-representations of dimension $1$ and $2$, respectively.

First consider  the abelian cases: (\ref{item1})--(\ref{item5}). The computation of the weights of the action is elementary. We comment on the singularities of a terminalization. In cases (\ref{item2}) and (\ref{item3}), a terminalization is obtained in the following way. Let $p_V \colon \Bl_{F}(V) \to V$ be the blowup of the plane $F \subset V$ with nontrivial stabilizer. The action of $G=C_{2k}$, with $k=2$ or $3$, lifts to $\Bl_{F}(V)$ and in particular on $p_V^{-1}(0) \simeq \IP^1$ via $[x: y] \mapsto [\xi_{2k} x : \xi^{-1}_{2k} y]$. We obtain the following diagram:
\[
\begin{tikzcd}[scale=1]
   \Bl_{F}(V) \arrow[d, "p_V"'] \arrow[r, "/C_2"]& \Bl_{F}(V)/C_{2}\arrow[r, "/C_k"]& \Bl_{F}(V)/C_{2k}\arrow[d]\\
   V \arrow[rr] & &V/C_{2k}\rlap{.}
\end{tikzcd}
\]
Since $C_2 = \langle \xi_{2k}^k \rangle$ fixes only the $p_V$-exceptional divisor, the quotient $\Bl_{F}(V)/C_{2}$ is smooth, and the residual $C_k$\nobreakdash-action fixes the points $[0:1]$ and $[1:0]$ in $p_V^{-1}(0)/C_2 \simeq \IP^1$. Hence,
the terminalization $\Bl_{F}(V)/C_k \to V/C_k$ has exactly two singular points of type $\tfrac{1}{k}(1,-1,1,-1)$. A similar argument gives a proof of (\ref{item4}) by chasing fixed points as above in a local version of diagram \eqref{square:terminalizationC3} in \cref{Prop:excorder3}. Finally, note that $\Bl_0 (\IA^2/C_2) \times \Bl_0 (\IA^2/C_3)$ and $(\Bl_0 (\IA^2/C_3))^2$ give symplectic resolutions in cases (\ref{item4.5}) and (\ref{item5}), respectively.

We are left with the non-abelian cases: (\ref{item6})--(\ref{item8}).
\begin{itemize}
    \item The only irreducible $2$-dimensional representation $\mathfrak{h}$ of $S_3$ is not symplectic; it is generated by complex reflections. So we must have $V \simeq \mathfrak{h} \oplus \mathfrak{h}^*$.
    \item The irreducible $2$-dimensional representations of $C_3 \times S_3$ are $\mathfrak{h}$, $\mathfrak{h}\otimes \chi$, $(\mathfrak{h}\otimes \chi)^*$. The representation $V$ cannot be $\mathfrak{h} \oplus \mathfrak{h}^*$; otherwise, the $(C_3 \times S_3)$-action would factor through $S_3$. We must therefore have $V \simeq (\mathfrak{h} \otimes \chi) \oplus (\mathfrak{h}\otimes \chi)^*$.
    \item $BT_{24}$ has seven irreducible representations: three characters $1, \chi, \chi^*$ lifted from $BT_{24} \twoheadrightarrow BT_{24}/Q_8=C_3$; three $2$-dimensional representations $\rho_{\symp}$, $\rho = \rho_{\symp} \times \chi$, $\rho^*$; and a $3$-dimensional representation. The reducible faithful $4$-dimensional representations of $BT_{24}$ are 
    \[1\oplus 1 \oplus \rho_{\symp}, \quad \chi \oplus \chi^* \oplus \rho_{\symp}, \quad \rho_{\symp}\oplus \rho_{\symp}, \quad \rho \oplus \rho^*.\]
    Only the last representation admits a plane with generic stabilizer exactly $C_3$. So $V \simeq \rho \oplus \rho^*$.
\end{itemize}
The quotient $V/G$ admits a smooth terminalization in cases (\ref{item6}), (\ref{item7}) and (\ref{item8}); see \cite[Corollary 1.2]{B2009} or \cite[Theorem 1]{LS2012}.
\end{proof}

\subsection{Fixed points of the punctual Hilbert scheme} \label{subsec:fixed point punctual Hilbert scheme}
Let $g$ be a symplectic automorphism of the complex torus $A$. The $g$-fixed points lying in the locus where the Hilbert--Chow morphism
$\epsilon \colon \K_2(A) \to A_0^{(3)}$ is an isomorphism are fixed points in $A_0^{(3)}$, and they can be described as triples of points in $A$ partitioned by $g$-orbits. The $g$-fixed points $z$ in the $\epsilon$-exceptional locus deserve additional analysis. 

The positive-dimensional fibers $\epsilon^{-1}(\epsilon(z))$, with their reduced structure, are isomorphic to
\begin{enumerate}
    \item $\IP(T^*_{\alpha}A) \simeq \IP^1$ if $\epsilon(z)=[(\alpha,\alpha,\beta)]$ and $g(\alpha)=\alpha$ and $g(\beta)=\beta$, or
    \item $\mathcal{H}_3 \simeq \IP(1,1,3)$ if $\epsilon(z)=[(\alpha,\alpha,\alpha)]$ and $g(\alpha)=\alpha$. 
\end{enumerate}

In the former case, $g$ acts on $T^*_{\alpha}A$ with weights $(1,-1)$, which gives the following lemma.
\begin{lemma} \label{lemma:one-dim fibers HC}
    The automorphism $g$ acting on the rational curve $\epsilon^{-1}(\alpha,\alpha,\beta) \simeq \IP(T^*_{\alpha}A)$ fixes either the whole curve if $\ord(g)=2$, or two points corresponding to the eigenlines of $g$ if $\ord(g) > 2$.
\end{lemma}

In the latter case, the fiber $\epsilon^{-1}(\epsilon(z))$ is the so-called punctual Hilbert scheme $\mathcal{H}_3$ of three points on a plane, isomorphic to the weighted projective space $\IP(1,1,3)$; see \cite[Section~IV.2, p.~76]{Briancon1977}
or \cite[Section~3]{Gordon15}. It parametrizes ideals of colength $3$ supported on a single point, say $0 \in T_0 A \simeq \IA^2_{x,y}$, namely 
\begin{itemize}
    \item the square $\mathfrak{m}^2$ of the maximal ideal $\mathfrak{m}=(x,y)$, 
    \item the curvilinear ideals $I$ of colength $3$, \textit{i.e.}, ideals containing the ideal of a smooth curve passing through the origin. In symbols, $I=(f, \mathfrak{m}^3)$, where $f \in \mathfrak{m}$ and $df \neq 0$. 
\end{itemize}
Note that if $\frac{\partial f}{\partial x} \neq 0$ and $\frac{\partial f}{\partial y} \neq 0$, we can write 
\[I=(x+c_0y+c_1y^2, \mathfrak{m}^3)=\left(\frac{1}{c_0}x+y+\frac{c_1}{c^3_0}x^2, \mathfrak{m}^3\right)\]
using the equivalences $x^2+c_0 xy \equiv 0$ and $xy+c_0 y^2 \equiv 0$ modulo $I$. This gives the transition functions of $\mathrm{Tot} \mathcal{O}_{\IP^1}(3)=\IP(1,1,3) \setminus [0:0:1] = \mathcal{H}_{3} \setminus \mathfrak{m}^2$. In particular, the zero-section of $\mathrm{Tot} \mathcal{O}_{\IP^1}(3)$, isomorphic to $\IP(T^*_{0}A) \simeq \IP^1_{[\lambda \colon \mu]}$, represents the curvilinear ideals cosupported on the lines through the origin; \textit{i.e.}, $I([\lambda \colon \mu])=(\lambda x +  \mu y, \mathfrak{m}^3)$.

\begin{lemma}\label{lemma:punctualHilbert}
    Let $V$ be a $2$-dimensional symplectic representation of the finite group $G$. Denote by $\IC(k)$ the $\IC^*$-character given by $t \cdot v = t^k v$, and let $W(k) \coloneqq W \otimes \IC(k)$ for any vector space $W$. Then the Hilbert scheme $\mathcal{H}_3$ of three points on $V$ is $G$-equivariantly isomorphic to
    \[V^*(1) \oplus (\det V^*)^{\otimes 2}(3)\sslash \IC^* \simeq \IP(1,1,3).\]
\end{lemma}
\begin{proof}
    Consider the curvilinear ideal $I=(x+c_0y+c_1y^2, \mathfrak{m}^3)$ and the matrix representation of the action of an element $g \in G$ on $V$, 
    \[g=M = \begin{pmatrix} 
a & b \\
c & d
\end{pmatrix}.\]
The ideal $g^*I=(f \circ g^{-1}, \mathfrak{m}^3)$ is generated by 
\[\left(\frac{dx-by}{\det M}
+c_0\frac{-cx+ay}{\det M}y
+c_1\frac{(-cx+ay)^2}{\det M^2}, \mathfrak{m}^3 \right) = 
\left(x
+ \frac{-b+c_0 a}{d-c_0 c}
+c_1\frac{ad-bc}{(d-c_0 c)^3}y^2, \mathfrak{m}^3 \right).\]
In the quasi-homogeneous coordinates $[x_1\colon x_2 \colon x_3]$ of $\IP(1,1,3)$, we write
\begin{equation*}
[1\colon c_0 \colon c_1] 
\longmapsto 
\left[\frac{d - c_0 c}{\det M} \colon \frac{-b+c_0 a }{\det M}\colon \frac{c_1}{\det M^2}\right],
\end{equation*}
or equivalently $[\underline{x} \colon x_3] \mapsto [(M^{-1})^t \underline{x} \colon \det((M^{-2})^t) x_3]$,  where $\underline{x}=(x_1, x_2)$.\end{proof}
Note that \new{$G$-fixed points of $\mathcal{H}_3 \simeq \IP(1,1,3)$ are $G$-invariant subspaces of $V^*\oplus (\det V^*)^{\otimes 2}$.} We obtain the following elementary corollary.

\begin{corollary}\label{cor:punctualHilbert} We use the notation of \cref{lemma:punctualHilbert}.
    A symplectic automorphism $g \in G$ fixes
    \begin{itemize}
        \item $\IP(V^*)$ and $\mathfrak{m}^2$ if $\ord(g)=2$, 
        \item two lines through $\mathfrak{m}^2$ corresponding to the eigenlines of $g$ if $\ord(g)=3$, 
        \item two points in $\IP(V^*)$ corresponding to the eigenlines of $g$ and $\mathfrak{m}^2$ if $\ord(g) > 3$. 
    \end{itemize}
\end{corollary}
Let $G$ be a group of symplectic automorphisms of the abelian variety $A$ (fixing the origin). To determine the points of $\mathcal{H}_3$ with nontrivial stabilizers, we proceed as follows:
\begin{enumerate}
    \item Note that the stabilizer of $\mathfrak{m}^2$ is $G$. 
    \item Determine the stabilizers for the action of $G$ on $\IP(T^*_{0}A)$, \textit{i.e.}, the eigenspaces of all elements $g \in \SL(T^*_0A)$. 
    \item Note that if the automorphism $g \in G$ of order $3$ fixes the point $z \in \IP(T^*_{0}A)$, then the line through $z$ and $\mathfrak{m}^2$ has generic stabilizer $C_3$.
\end{enumerate}
\small
\normalsize
\begin{figure}[h]
\centering
\begin{tikzpicture}
\draw[thick] plot coordinates {(-1.5,0) (1.5,0)} node[right]{\(F_{-\id}\)};
\draw[dashed] plot coordinates {(-1.2,0.5) (0.5,-1.5)};
\draw[dashed] plot coordinates {(1.2,0.5) (-0.5,-1.5)};

\fill (0,-0.9) circle[radius=2pt] node[right]{\(\mathfrak{m}^2\)};
\fill (0,-2.3) node[below]{\(G=C_4\)};
\end{tikzpicture} 
\hspace{1cm}
\begin{tikzpicture}
\draw[thick] plot coordinates {(-1.5,0) (1.5,0)} node[right]{\(F_{-\id}\)};
\draw[thick] plot coordinates {(-1.2,0.5) (0.5,-1.5)} node[right]{\(F_{g_3}\)};
\draw[thick] plot coordinates {(1.2,0.5) (-0.5,-1.5)};

\fill (0,-0.9) circle[radius=2pt] node[right]{\(\mathfrak{m}^2\)};
\fill (0,-2.3) node[below]{\(G=C_6\)};
\end{tikzpicture}
\hspace{1cm}
\begin{tikzpicture}
\draw[thick] plot coordinates {(-1.8,0) (1.8,0)} node[right]{\(F_{-\id}\)};
\draw[thick] plot coordinates {(-2,0.5) (0.9,-1.5)} node[right]{\(F_{g'''_3}\)};
\draw[thick] plot coordinates {(-1.7,0.5) (0.7,-1.5)};
\draw[thick] plot coordinates {(-0.9,0.5) (0.4,-1.5)} node[below]{\(F_{g''_3}\)};
\draw[thick] plot coordinates {(-0.7,0.5) (0.3,-1.5)};
\draw[thick] plot coordinates {(0.9,0.5) (-0.4,-1.5)} node[below]{\(F_{g'_3}\)};
\draw[thick] plot coordinates {(0.7,0.5) (-0.3,-1.5)};
\draw[thick] plot coordinates {(2,0.5) (-0.9,-1.5)} node[left]{\(F_{g_3}\)};
\draw[thick] plot coordinates {(1.7,0.5) (-0.7,-1.5)};

\fill (0,-0.9) circle[radius=2pt] node[right]{\hspace{2mm}\(\mathfrak{m}^2\)};
\fill (0,-2.3) node[below]{\(G=BT_{24}\)};
\end{tikzpicture}
\caption{The picture represents the loci with nontrivial stabilizer in the punctual Hilbert scheme $\mathcal{H}_3$ with respect to the action of the group $G$. We draw $\mathcal{H}_3$ as a cone with vertex $\mathfrak{m}^2$, the horizontal section is $\IP(V^*)$, and the segments from $\IP(V^*)$ to $\mathfrak{m}^2$ are lines parametrizing ideals $I=(f, \mathfrak{m}^3)$ with fixed $df$.}
\label{fig:enter-label}
\end{figure}
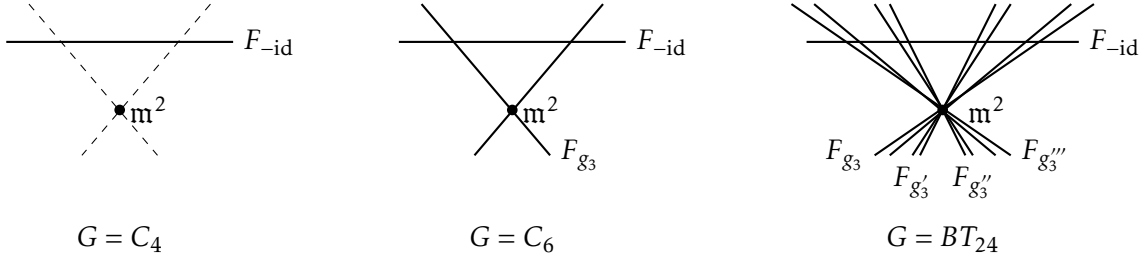

\subsection{Fixed points of $\boldsymbol{\K_2(A)}$} \label{subsec:fixed point K2(A)}

\begin{lemma} \label{lem: p with linear Gp} 
Let $G$ be a finite group with a faithful symplectic linear action on $A$. In \cref{table Gfixed}, we provide the number of surfaces and isolated points in $\Fix(G)$ in $\K_2(A)$, and their inclusion in surfaces $F_{g}$ fixed by an automorphism $g \in G$.
\end{lemma}

\renewcommand*{\arraystretch}{1.2}
\begin{table}[h]
\centering
\caption{Fixed loci of some linear actions on $\K_2(A)$}\label{table Gfixed}
\begin{tabular}{c|c|c|l}
$G$ & $G$-fixed surface & \(G\)-fixed points & Relative position of fixed loci \\\hline
$C_2$ & 1 & 36 & 36 pts $\not \in F_{-\id}$\\ \hline
$C_3$ & 1 & 12 & 12 pts $\not \in F_{g_3}$\\ \hline
\multirow{2}{*}{$C_4$} & \multirow{2}{*}{0} & \multirow{2}{*}{16} & 8 pts $\in F_{-\id}$ \\
 &  &  & 8 pts $\not\in F_{-\id}$\\ \hline
\multirow{3}{*}{$C_6$} & \multirow{3}{*}{0} & \multirow{3}{*}{12} & 2 pts $= F_{-\id} \cap F_{g_3}$ \\
 &  &  & 4 pts $\in F_{-\id} \setminus F_{g_3}$\\
  &  &  & 6 pts $\in F_{g_3} \setminus F_{-\id}$\\\hline
$BT_{24}$ & 0 & 2 & 2 pts $= \bigcap_{\ord(g)=3} F_{g}$
\end{tabular}
\end{table}

\begin{proof}
Cases $C_2$ and $C_3$ are classical; see for instance \cite[Section~1.2.1]{Tar15} and \cite[Section~5.5]{Fu-Menet}. We focus on the remaining cases. For any $G$-fixed point $z$ in $\K_2(A)$, the image $\epsilon(z)= [(x,y,-x-y)]$ in $A_0^{(3)}$ is $G$-fixed too, and $\{x,y,-x-y\}$ is a union of orbits for the action of $G$ on $A$, equivalently a union of fibers of the quotient $A \to A/G$.

    \begin{enumerate} \setcounter{enumi}{2}
        \item  If $G \simeq C_4 = \langle g_4 \rangle$, the singularities of $A/G$ are $4A_3 + 6A_1$; see \cite[Lemma 3.19]{Fujiki88} and also \cite[Proposition 2.7]{Pietromonaco2022}. We denote the point/orbit in $A$ over the singularities $4A_3$ by $0,q_1, q_2 $ and $q_3$, and the orbits over $6A_1$ are $\{x, g_4(x)\}$ for some $x \in A[2] \setminus \{0,q_1, q_2, q_3\}$. If $z$ is a $G$-fixed point in $\K_2(A)$, then one of the following holds: 
        \begin{itemize}
            \item $\epsilon(z)=[(0,0,0)]$, and $G$ fixes two points in $\epsilon^{-1}(0,0,0)$ lying in $F_{-\id}$ and the point $\mathfrak{m}^2 \not \in F_{-\id}$; see \cref{cor:punctualHilbert}.
            \item $\epsilon(z)=[(q_i, q_i, 0)] \in \epsilon(F_{-\id})$, and $G$ fixes two points in $\epsilon^{-1}(q_i, q_i, 0)$ lying on $F_{-\id}$; see \cref{lemma:one-dim fibers HC}.
            \item $\epsilon(z)=[(q_1, q_2, q_3)] \notin \epsilon(F_{-\id})$. 
            \item $\epsilon(z)=[(x,g_4(x), -x-g_4(x))] \notin \epsilon(F_{-\id})$ for some $x \in A[2] \setminus \{0,q_1, q_2, q_3\}$.
        \end{itemize}

        \item If $G \simeq \langle g_3, -\id \rangle \simeq C_6$, the singularities of $A/G$ are $A_5 + 4A_2 + 5A_1$. The point/orbit over $A_5$ is $0$, the orbits over $4A_2$ are $\{x,-x\}$ for some $x \neq 0$ with $g_3(x)=x$, and the orbits over $5A_1$ are $\{x,g_3(x),g_3^2(x)\}$ for some $x \in A[2] \setminus \{0\}$.        
        If $z$ is a $G$-fixed point in $\K_2(A)$, then one of the following holds: 
        \begin{itemize}
            \item  $\epsilon(z)=[(0,0,0)]$, and $G$ fixes two points in $\epsilon^{-1}(0,0,0)$ lying in $F_{-\id} \cap F_{g_3}$ and the point $\mathfrak{m}^2 \in F_{g_3} \setminus F_{-\id}$; see \cref{cor:punctualHilbert}. 
            
            \item $\epsilon(z)=[(x,0,-x)] \in F_{-\id} \setminus F_{g_3}$ for some $x \neq 0$ with $g_3(x)=x$.
            
            \item  $\epsilon(z)=[(x,g_3(x), {g_3}^2(x))] \in F_{g_3} \setminus F_{-\id}$ for some $x \in A[2] \setminus \{0\}$.
        \end{itemize}
        
        \item  If $G \simeq BT_{24}$, the singularities of $A/G$ are $E_6+D_4 + 4A_2 +A_1$. The point/orbit over $E_6$ is $0$, the orbit over $D_4$ is $\{q_1,q_2,q_3\}$, and all other orbits of $G$ have cardinality greater than $3$. If $z$ is a $G$-fixed point in $\K_2(A)$, then one of the following holds: 
        \begin{itemize}
            \item $\epsilon(z)=[(0,0,0)]$, and $G$ fixes $\mathfrak{m}^2 = \epsilon^{-1}(0,0,0) \cap \bigcap_{g \in G \,:\, \ord(g)=3} F_{g}$; see \cref{cor:punctualHilbert}. 
            \item $\epsilon(z)=[(q_1, q_2, q_3)] \in \bigcap_{g \in G \,:\, \ord(g)=3} F_{g}$.\qedhere
        \end{itemize} 
    \end{enumerate}
    \renewcommand{\qed}{}
\end{proof}

\begin{lemma} \label{lem:p with big stabiliser}
Let $G$ be a finite group of induced symplectic automorphisms of\, $\K_2(A)$.

    \begin{enumerate}
        \item \label{j3} If\, $G \simeq \langle \tau_\alpha \rangle \simeq C_3$, then  $G$ fixes $27$ points. 
        
        \item \label{j6} If\, $G \simeq \langle \tau_\alpha, -\id \rangle \simeq S_3$, then $G$ fixes the unique intersection point of all surfaces fixed by an involution of\,~$G$.

        \item If\, $G \simeq \langle \tau_\alpha, g_3 \rangle \simeq C_3^2$, then $G$ fixes the three intersection points between \new{a pair} 
        of surfaces fixed by an element of\,~$G$.

        \item If\, $G \simeq \langle g_3, \tau_\alpha, -\id \rangle \simeq C_3 \times S_3$, then $G$ fixes the unique intersection point of all surfaces fixed by an element of\, $G \setminus \langle g_3\rangle$. 
    \end{enumerate}
\end{lemma}

\begin{proof}
    Let $z$ be a point whose stabilizer $G_z$ contains $\tau_\alpha$ with $\alpha \neq 0$. 
    Then $z$ is of the form $[(x, x+\alpha, x-\alpha)]$ with $x \in A[3]$, and there are $27=|A[3]|/3$ such points $z$. In particular, we have the following:  
    \begin{itemize}
    \item If $G_z = \langle \tau_\alpha, -\id \rangle$, then \new{$x \in \langle \alpha \rangle$,} 
      so
      \[p=[(0, \alpha, -\alpha)] \in F_{-\id} \cap F_{\tau_{\alpha}(-\id)} \cap F_{\tau_{-\alpha}(-\id)}.\]

    \item If $G_z =\langle \tau_\alpha, g_3 \rangle$, then one of the following holds: 
      \begin{itemize}
        \item $g_3(x)=x$, \textit{i.e.}, $z\in F_{\tau_\alpha g_3} \cap F_{\tau_{-\alpha}g_3} \setminus F_{g_3}$.
        \item $g_3(x)=x-\alpha$, \textit{i.e.}, $z \in F_{g_3} \cap F_{\tau_{-\alpha} g_3} \setminus F_{\tau_\alpha g_3}$.
        \item $g_3(x)=x+\alpha$, \textit{i.e.}, $z\in F_{g_3} \cap F_{\tau_{\alpha} g_3} \setminus F_{\tau_{-\alpha} g_3}$.
    \end{itemize} 
    Note that any pair of fixed surfaces intersects in three points of the form $[(x, x+\alpha, x-\alpha)]$.

    \item If $G_z = \langle g_3, \tau_\alpha, -\id \rangle$, then combining the two cases above, we obtain
\begin{equation*}\pushQED{\qed}
      z=[(0, \alpha, -\alpha)] \in F_{-\id} \cap F_{\tau_{\alpha}(-\id)} \cap F_{\tau_{-\alpha}(-\id)} \cap F_{\tau_\alpha g_3} \cap F_{\tau_{-\alpha} g_3} \setminus F_{g_3}.
   \qedhere \popQED
	\end{equation*}  
    \end{itemize}    
\renewcommand{\qed}{}    
\end{proof}

\subsection{Singularities of symplectic quotients}\label{sec:singularitiesterminalization}
The singular locus of $\K_2(A)/G$ is stratified in
\begin{enumerate}
    \item \label{item:subvarisotropy} locally closed surfaces with isotropy $C_2$ or $C_3$, 
    \item \label{item:higherisotropypoints} points in the closure of the surfaces in \eqref{item:subvarisotropy} with isotropy strictly greater than $C_2$ or $C_3$, 
    \item remaining isolated singular points.
\end{enumerate}
The points of type \eqref{item:higherisotropypoints} are images under the quotient map $q \colon K_2(A) \to K_2(A)/G$ of 
\begin{enumerate}[label=(\ref{item:higherisotropypoints}.\arabic*)]
    \item \label{item:2.1} the intersection of surfaces $F_{g} \cap F_{h}$ fixed by some $g,h \in G$, 
    \item \label{item:2.2} fixed points in $F_{g}$ for the residual action of $N_G(\langle g \rangle)/\mathrm{ncl}(g)$, where $N_G(\langle g \rangle)$ is the normalizer of the cyclic subgroup $\langle g \rangle$ generated by $g$, and $\mathrm{ncl}(g)$ is the normal subgroup generated by $g$ in $N_G(\langle g \rangle)$.
\end{enumerate}
In order to determine the singularities of $X/G$ effectively, we use the following algorithm:
\begin{enumerate}
    \item List the possible stabilizers of points of $X$ for the action of $G$. 
    \item Determine all the points of type~\ref{item:subvarisotropy},~\ref{item:2.1} and~\ref{item:2.2}. 
\item Note that the number of remaining isolated fixed points with isotropy $m$ is 
\[
\sum_{g \in G,\ \ord(g)=m}\big((\# \text{ isolated }g\text{-fixed points}) - (\#\ g\text{-fixed points of type~\ref{item:2.1} and~\ref{item:2.2}})\big)\big/(|G|/\ord(g))
\]
\end{enumerate}

One may run the algorithm for all groups in \cref{table sing n=2}, but for brevity we make the following expository choice. 
For the terminalizations which are deformation equivalent to a Fujiki variety (see \cref{prop:mainsection}), the singularities have been already computed in \cite[Theorem 1.11] {GM22},  and we refer the reader to \textit{loc.~cit.} Here we study in detail the singularities of the new deformation types of IHS fourfolds in \cref{table sing n=2}, namely $G_{\circ}=C_2$ (see Section~\ref{sec:C2}) and $G=C_3^2\rtimes BT_{24}$ (see Section~\ref{sec:BT24}). For the only remaining case $G=BT_{24}$, for which we do not know yet if it is deformation equivalent to other Fujiki varieties (see \cref{rmk:open}), we provide the diagram of the singularities of $\K_2(A)/G$  and leave the details to the reader; see Section~\ref{diagram24,3}. 

\subsubsection{Groups with $\boldsymbol{G_0=C_2}$}\label{sec:C2}
Suppose $G_\tr \simeq C_3^{\oplus i}$ for some $i=0,\ldots,4$. Since any point $z \in \K_2(A)$ cannot be fixed by more than one translation up to multiples (\textit{i.e.}, $G_z \cap G_\tr = \{1\} $ or $\langle \tau_\alpha \rangle$), the possible nontrivial stabilizers of points in $\K_2(A)$ for the action of $G$ are 
\begin{alignat*}{2}
    \langle \tau_\beta \rangle \simeq C_3,
    & \quad \beta \in G_\tr \setminus \{0\}, \\
    \langle \tau_\alpha(-\id) \rangle \simeq C_2,
    & \quad \alpha \in G_\tr, \\
    \langle \tau_\alpha(-\id), \tau_\beta \rangle \simeq S_3,
    & \quad \alpha \in G_\tr,\ \beta \in G_\tr \setminus \{0\}.
\end{alignat*}
The singular points of $Y$ correspond to the isolated singularities of $X/G$. Indeed, as $N_2=1$ and $N_3=0$, the singular locus contains a unique irreducible component of codimension~$2$, namely $q(F_{-\id})$, with points of isotropy $C_2$ or $S_3$. By \cref{prop: sing after terminalization}\eqref{item6}, the terminalization $Y \to X/G$ is a symplectic resolution in a neighborhood of $q(F_{-\id})$. 

The singularities of $X/G$ away from $q(F_{-\id})$ are images of isolated points in $X$ fixed by elements $g \in G$. Given the list of possible stabilizers, the isolated points of the fixed locus of an involution do not lie on any surface fixed by any other involution.  We obtain that
\begin{align*}
    a_2 
    & = \# \text{ isolated singular points of }X/G\text{ with isotropy }C_2\\
    & = (\# \text{ isolated points in }X\text{ fixed by an involution in }G)/(\# \text{orbits of such points})\\
    & = (\# \text{ involutions in }G) \cdot (\# \text{ isolated points fixed by}-\id)/(|G|/2)\\
    & = 3^i \cdot 36/(2 \cdot 3^i/2)=36.
\end{align*}
On the contrary, if $g = \tau_{\beta}$ is a translation, an isolated fixed point may lie on a surface $F_{\tau_{\alpha}(-\id)}$. In that case, the point is the unique intersection of the three surfaces 
\[F_{\tau_{\alpha}(-\id)} \cap F_{\tau_{\alpha + \beta}(-\id)} \cap F_{\tau_{\alpha - \beta}(-\id)} = [(-\alpha +\beta, -\alpha-\beta, -\alpha)].\] Since there are exactly $\frac{1}{3}{3^i \choose 2}$
such points, we obtain that
\begin{alignat*}{2}
a_3 & = \# \text{ isolated singular points of }X/G\text{ with isotropy }C_3\\
    & = \big( (\# \text{ isolated points in }X \text{ fixed by a translation}) \\
    & \qquad- (\# \text{ isolated such points lying on a fixed surface}) \big) \big/(\# \text{orbits of such points})\\
    &  
    = \big((\# \text{ subgroups }\langle \tau_{\beta}\rangle \subset G) \cdot (\# \text{ isolated points in }X\text{ fixed by }\tau_{\beta})\\
    & \qquad- (\# \text{ isolated such points lying on a fixed surface}) \big) \big/(|G|/3)\\
    & = \left( \frac{3^i-1}{2} \cdot 27-\frac{1}{3}{3^i \choose 2}\right)\bigg/\frac{2 \cdot 3^i}{3}=\frac{(3^i-1)(3^{4-i}-1)}{4} \in \{0,13,16\}.
\end{alignat*}

\subsubsection{Group $\boldsymbol{C_3^2\rtimes BT_{24}}$ (ID 216,153)}\label{sec:BT24}
We first determine the possible nontrivial stabilizers of points of $\K_2(A)$ under the action of $G=C_3^2\rtimes BT_{24}$; see Figure~\ref{fig:poset Gp 216}.

We note in particular that there are no stabilizers isomorphic to $S_3$ and $Q_8$. To this end, observe that all subgroups of $G$ isomorphic to $S_3$ and $Q_8$ are conjugate, so it suffices to show that the groups $S_3 \simeq \langle \tau_\alpha, -\id \rangle$ and $Q_8 = \langle i,j,k \rangle \subset BT_{24}$ are not stabilizers of any point in $\K_2(A)$. 
\begin{itemize}
    \item In the former case, any point $[(0,\alpha, -\alpha)] \in \Fix(\langle \tau_\alpha, -\id \rangle )$ is also fixed  by $\tau_{\alpha} g_{\alpha}$, where \new{$\ga \in BT_{24}$ is the unique automorphism of order $3$ fixing the line $\langle \alpha \rangle  \in G_{\tr}$.} 
    \item  In the latter case, any point $z \in \K_2(A)$ fixed by $Q_8$ is fixed by $BT_{24}$ too. Indeed, $\epsilon(z) \in A_0^{(3)}$ is either $[(0,0,0)]$ or $[(q_1, q_2, q_3)]$, as in the proof of \cref{lem: p with linear Gp} for $BT_{24}$, and they are both fixed by $BT_{24}$. Further, the only $Q_8$-fixed point of the punctual Hilbert scheme $(0,0,0)$ is $\mathfrak{m}^2$, which is fixed by $BT_{24}$ too.
\end{itemize}

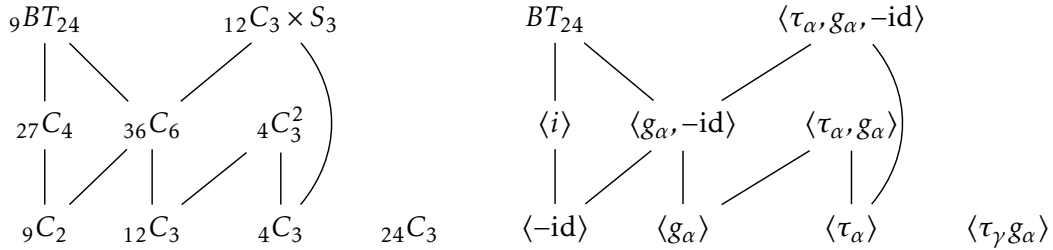
\begin{figure}[h!]
\[\begin{tikzcd}[column sep=tiny]
	{{}_{9}BT_{24}} && {{}_{12}C_3 \times S_3} &&& {BT_{24}} && {\langle \tau_\alpha, \ga, -\id \rangle} \\
	{{}_{27}C_4} & {{}_{36}C_6} & {{}_{4}C_3^2} &&& {\langle i \rangle} & {\langle \ga, -\id \rangle} & {\langle \tau_\alpha, \ga \rangle} \\
	{{}_{9}C_2} & {{}_{12}C_3} & {{}_{4}C_3} & {{}_{24}C_3} & {} & {\langle -\id \rangle} & {\langle \ga \rangle} & {\langle  \tau_\alpha  \rangle} & {\langle \tau_\gamma \ga \rangle}
	\arrow[no head, from=3-2, to=2-3]
	\arrow[no head, from=3-3, to=2-3]
	\arrow[no head, from=3-1, to=2-2]
	\arrow[no head, from=2-2, to=1-3]
	\arrow[no head, from=2-2, to=1-1]
	\arrow[no head, from=3-1, to=2-1]
	\arrow[no head, from=2-1, to=1-1]
	\arrow[no head, from=1-6, to=2-6]
	\arrow[no head, from=2-6, to=3-6]
	\arrow[no head, from=3-6, to=2-7]
	\arrow[no head, from=2-7, to=1-6]
	\arrow[no head, from=2-7, to=1-8]
	\arrow[no head, from=3-8, to=2-8]
	\arrow[no head, from=3-7, to=2-8]
	\arrow[no head, from=3-2, to=2-2]
	\arrow[curve={height=23pt}, no head, from=3-3, to=1-3]
	\arrow[no head, from=3-7, to=2-7]
	\arrow[curve={height=23pt}, no head, from=3-8, to=1-8]
\end{tikzcd}\]
    \caption{On the left, the poset of nontrivial stabilizers of points of $\K_2(A)$ under the action of $G=C_3^2\rtimes BT_{24}$, up to conjugation. The left subscript denotes the number of conjugate subgroups. On the right, we provide a representative for each conjugacy class. Note that \rev{$\alpha, \gamma \in G_{\tr}$ with $g_\alpha(\gamma)\neq \gamma$.}}
    \label{fig:poset Gp 216}
\end{figure}

As $N_2=1$ and $N_3=1$ (\textit{cf.} \cref{table n=2}), the only surfaces in the singular locus of $X/G$ are $q(F_{-\id})$ and $q(F_\ga)$. The residual groups acting on the K3 surfaces $F_{-\id}$ and $F_\ga$ are $A_4= BT_{24}/{-\id}$ and $S_3= \langle \tau_\alpha, -\id \rangle$, respectively. The singularities of the quotients $F_{-\id}/A_4$ and $F_{\ga}/S_3$ are   
\begin{alignat*}{3}
    F_{-\id}/A_4& : && \quad 6A_2 + 4A_1, \\
    F_{\ga}/S_3 &: && \quad 3A_2 + 8A_1;
\end{alignat*}
see \cite[Theorem 3, \#17, \#6]{Xiao1996}. This suffices to describe the singular points of $\K_2(A)/G$ lying on $q(F_{-\id}) \cup q(F_{g_3})$ (\textit{cf.} Figure~\ref{fig:configuration 216}):
\begin{itemize}
    \item The six points of type $A_2$ in $F_{-\id}/A_4$ correspond to
    \begin{itemize}
        \item three points in $q(F_{-\id})$ with isotropy $C_6$,
        \item two points in the intersection of $q(F_{-\id})$ and $q(F_\ga)$ with isotropy $C_6$,
        \item The point $[(\alpha, -\alpha, 0)] \subset q(F_{-\id}) \cap q(F_\ga)$, with isotropy $C_3 \times S_3$.
    \end{itemize}
  \item The four points of type $A_1$ in $F_{-\id}/A_4$ correspond to four points in $q(F_{-\id})$ with isotropy $C_4$. 
    \item The three points of type $A_2$ in $F_\ga/S_3$ corresponds to
    \begin{itemize}
    \item the point $[(\alpha, -\alpha, 0)] \subset q(F_{-\id}) \cap q(F_\ga)$, with isotropy $C_3 \times S_3$, 
    \item the point $[(x+\alpha, x-\alpha, x)] \in q(F_\ga)$, with $x \in \Pi_\ga \setminus G_\tr$ and  isotropy $C_3^2$. 
    \end{itemize}
Note that two points of type $A_2$ in $F_\ga/S_3$ are identified by the normalization map $F_\ga/S_3 \to q(F_{\ga})$. 
    \item The eight points of type $A_1$ in $F_\ga/S_3$ correspond to
    \begin{itemize}
        \item four points in $q(F_\ga)$ with isotropy $C_6$,
        \item two points in the intersection of $q(F_{-\id})$ and $q(F_\ga)$, with isotropy $C_6$,
        \item two points in $q(F_\ga)$  with isotropy $BT_{24}$.
    \end{itemize}   
\end{itemize}

\small
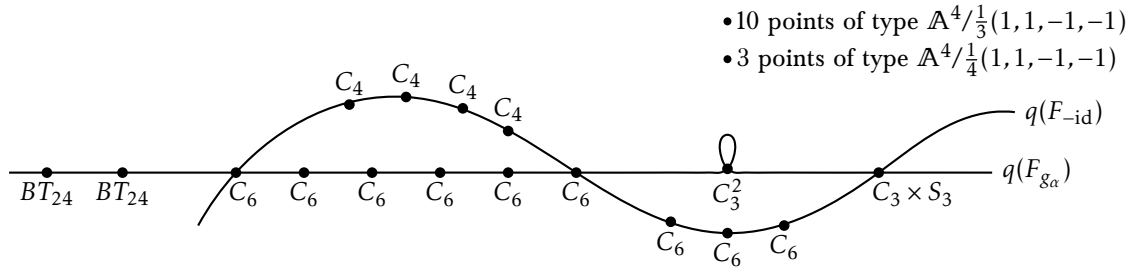
\begin{figure}
\begin{tikzpicture}
\draw[thick,hobby] plot coordinates {(-2.5,0) (0,0) (2,0) (5,0) (6,0) (6.5,0) (6.7,0) (6.8,0) (6.9,0) (7.1,0.3) (7,0.5) (6.9,0.3) (7.1,0) (7.2,0) (7.3,0) (7.5,0) (8,0) (9,0) (10.5,0)} node[right]{\(q(F_\ga)\)};
\draw[thick,hobby] plot coordinates {(0,-0.7) (0.5,0) (2.75,1) (5,0) (7,-0.8) (9,0) (10.5,0.8) (10.8,0.8) } node[right]{\(q(F_{-\id})\)};

\fill (-2,0) circle[radius=2pt] node[below]{\(BT_{24}\)};
\fill (-1,0) circle[radius=2pt] node[below]{\(BT_{24}\)};
\fill (0.5,0) circle[radius=2pt] node[below]{\hspace{2mm}\(C_6\)};
\fill (1.4,0) circle[radius=2pt] node[below]{\(C_6\)};
\fill (2.3,0) circle[radius=2pt] node[below]{\(C_6\)};
\fill (3.2,0) circle[radius=2pt] node[below]{\(C_6\)};
\fill (4.1,0) circle[radius=2pt] node[below]{\(C_6\)};
\fill (5,0) circle[radius=2pt] node[below]{\(C_6\)};
\fill (7,0.05) circle[radius=2pt] node[below]{\(C_3^2\)};
\fill (9,0) circle[radius=2pt] node[below]{\hspace{9mm}\(C_3\times S_3\)};

\fill (2,0.9) circle[radius=2pt] node[above]{\(C_4\)};
\fill (2.75,1) circle[radius=2pt] node[above]{\(C_4\)};
\fill (3.5,0.85) circle[radius=2pt] node[above]{\(C_4\)};
\fill (4.1,0.55) circle[radius=2pt] node[above]{\(C_4\)};
\fill (6.25,-0.65) circle[radius=2pt] node[below]{\(C_6\)};
\fill (7,-0.8) circle[radius=2pt] node[below]{\(C_6\)};
\fill (7.75,-0.7) circle[radius=2pt] node[below]{\(C_6\)};

\fill (7,2) circle[radius=1.5pt] node[right]{10 points of type \(\IA^4/\tfrac{1}{3}(1,1,-1,-1)\)};
\fill (7,1.5) circle[radius=1.5pt] node[right]{3 points of type \(\IA^4/\tfrac{1}{4}(1,1,-1,-1)\)};
\end{tikzpicture}
\caption{Singularities of $X/G$ for $G=C_3^2\rtimes BT_{24}$ (ID 216,153).}
\label{fig:configuration 216}
\end{figure}
\normalsize
\medskip

We are left to determine the isolated singular points in \new{$\K_2(A)/G$ according to their stabilizer type.} 
By \cref{lem: p with linear Gp}, an isolated fixed point in $\K_2(A)$ can only have stabilizer isomorphic to $C_4$, $C_3$ or $C_2$.

\textit{Isotropy} $C_4$.~ As all subgroups in $G$ isomorphic to $C_4$ are conjugate, it suffices to consider the fixed points of the linear automorphism $g_4 \in BT_{24}$. By \cref{lem: p with linear Gp} for $C_4$, the automorphism $g_4$ fixes eight points not lying on $F_{-\id}$, but two of them are fixed by the whole $BT_{24}$, and they lie on $F_\ga$; see \cref{lem: p with linear Gp} for $BT_{24}$. Thus, the number of isolated singular points with isotropy $C_4$ in $X/G$ is
\begin{equation} \label{equ:216, isolated C4}
    (8-2) \cdot (\# \text{subgroups conjugate to }\langle g_4 \rangle)/(|G|/|C_4|)=6\cdot 27/(216/4)=3.
\end{equation}

\textit{Isotropy} $C_2$.~ As all subgroups in $G$ isomorphic to $C_2$ are conjugate, it suffices to consider the points fixed by $-\id$. By \cref{lem: p with linear Gp} for $C_2$, the involution $-\id$ fixes $36$ points not lying on $F_{-\id}$, but
\begin{itemize}
    \item[-]2 of them are fixed by the whole $BT_{24}$ and lie on $F_\ga$,  
    \item[-]18 of them have stabilizer $\langle i \rangle, \langle j \rangle$ or $\langle k \rangle \simeq C_4$, 
    \item[-]$16$ of them lie on a fixed surface $F_g$ for some $g \in BT_{24}$ of order $3$.
\end{itemize}
So there are no isolated points with isotropy $C_2$ in $X/G$.

\textit{Isotropy} ${{}_{12}C_3}$.~ Consider the subgroups in $G$ conjugate to $\langle \ga \rangle$. By \cref{lem: p with linear Gp} for $C_3$, the automorphism $\ga$ fixes $12$ points not lying on $F_\ga$, but
\begin{itemize}

 \item[-]1 lies on $F_{-\id} \cap F_{\tau_\alpha \ga} \cap F_{\tau_{-\alpha} \ga}$, 
    \item[-]2 lie on $F_{\tau_\alpha \ga} \cap F_{\tau_{-\alpha} \ga}$, 
    \item[-]$9$ lie on a single surface $F_{\tau_{\beta}(-\id)}$ with $\beta \in \{0, \pm \alpha\}$.
\end{itemize}
So there are no isolated points with isotropy $C_3$ conjugate to $\langle \ga \rangle$ in $X/G$.

\textit{Isotropy} ${{}_{4}C_3}$.~ Consider the subgroups in $G$ conjugate to $\langle \tau_\alpha \rangle$. By \cref{lem:p with big stabiliser}\eqref{j3}, the translation $\tau_\alpha$ fixes $27$ isolated points of the form $z=[(x, x+\alpha, x-\alpha)]$ for $x \in A[3]$. If $x \in G_\tr$, then $z$ lies on $F_{\tau_{-x}(-\id)}$.
Thus, the number of isolated points with isotropy $C_3$ conjugate to $\langle \tau_\alpha \rangle$ in $X/G$ is 
\begin{equation} \label{equ:216, isolated C3 translation}
    (27-9) \cdot (\# \text{subgroups conjugate to }\langle \tau_{\alpha} \rangle)/(|G|/|C_3|)=(27-9)\cdot 4 /(216/3)=1.
\end{equation}

\textit{Isotropy} ${{}_{24}C_3}$.~ Consider the subgroups of $G$ conjugate to $\tau_\gamma \ga$ with $\gamma \in G_\tr \setminus \Pi_\ga$.  By \cref{lem: isolated for C3 not-linear}, the number of isolated points with isotropy conjugate to $\langle \tau_\gamma g_\alpha\rangle$ is
\begin{equation} \label{equ:216, isolated C3 not-linear}
    27 \cdot (\# \text{subgroups conjugate to }\langle \tau_\gamma \ga \rangle)/(|G|/|C_3|)=27\cdot 24 /(216/3)=9.
    \end{equation}

Finally, combining \cref{prop: sing after terminalization} and \eqref{equ:216, isolated C4}--\eqref{equ:216, isolated C3 not-linear}, we conclude that 
\begin{equation*}
 a_2= 3 \cdot 4 + 4 \cdot 2=20, \quad a_3= 10 + 2 \cdot 3 = 16, \quad a_4=3.   
\end{equation*}

\begin{lemma} \label{lem: isolated for C3 not-linear}
    There are precisely $27$ points in $\K_2(A)$ fixed by~$\tau_\gamma g_\alpha$ with $g_\alpha(\gamma) \neq \gamma$.
\end{lemma}

\begin{proof}
    Let $z \in \K_2(A)$ with $G_z=\langle \tau_\gamma g_\alpha \rangle$, and write $\epsilon(z)=[(x,y,-x-y)] \in A_0^{(3)}$, where $x$ and $y$ are fixed by~$\tau_\gamma g_\alpha$. Note that $x,y \in A[9] \cap (g_\alpha - \id)^{-1}(\gamma)$, which consists of nine elements. 
    If $x \neq y,-2y,4y$, then $z$ does not lie on the exceptional locus of $\epsilon$, and we have $(9 \cdot 6)/{3!}=9$ such points $z$. 
    Otherwise, $\epsilon(z)=[(x,x,-2x)]$ and~$\tau_\gamma g_\alpha$ fixes two points in $\epsilon^{-1}(x,x,-2x)$ by \cref{lemma:one-dim fibers HC}. \new{In total, we obtain 18 more points fixed by $\tau_\gamma g_\alpha$ lying on the exceptional locus.} 
\end{proof}

\subsubsection{Group~\(\boldsymbol{G=BT_{24}}\) (ID 24,3)}\label{diagram24,3}
\small
\begin{center}
\begin{tikzpicture}
\draw[thick,hobby] plot coordinates {(-2,0)  (10,0)} node[right]{\(F_g\)};
\draw[thick,hobby] plot coordinates { (-1,-0.8) (0,-1) (2,0) (5.5,1) (9,0) (10,-0.5) } node[below]{\(F_{-\id}\)};

\fill (-1,0) circle[radius=2pt] node[below]{\(BT_{24}\)};
\fill (-0,0) circle[radius=2pt] node[below]{\(BT_{24}\)};
\fill (2,0) circle[radius=2pt] node[below]{\hspace{1mm}\(C_6\)};
\fill (3.4,0) circle[radius=2pt] node[below]{\(C_6\)};
\fill (4.8,0) circle[radius=2pt] node[below]{\(C_6\)};
\fill (6.2,0) circle[radius=2pt] node[below]{\(C_6\)};
\fill (7.6,0) circle[radius=2pt] node[below]{\(C_6\)};
\fill (9,0) circle[radius=2pt] node[below]{\(C_6\)};
\fill (3,0.5) circle[radius=2pt] node[above]{\(C_4\)};
\fill (3.5,0.7) circle[radius=2pt] node[above]{\(C_4\)};
\fill (4,0.85) circle[radius=2pt] node[above]{\(C_4\)};
\fill (4.5,0.95) circle[radius=2pt] node[above]{\(C_4\)};
\fill (6.5,0.9) circle[radius=2pt] node[above]{\(C_6\)};
\fill (7,0.75) circle[radius=2pt] node[above]{\(C_6\)};
\fill (7.5,0.6) circle[radius=2pt] node[above]{\(C_6\)};
\fill (8,0.45) circle[radius=2pt] node[above]{\(C_6\)};

\fill (-3,1.7) circle[radius=1.5pt] node[right]{4 points of type \(\IA^4/\tfrac{1}{3}(1,1,-1,-1)\)};
\fill (-3,1.2) circle[radius=1.5pt] node[right]{3 points of type \(\IA^4/\tfrac{1}{4}(1,1,-1,-1)\)};
\end{tikzpicture}
\end{center}
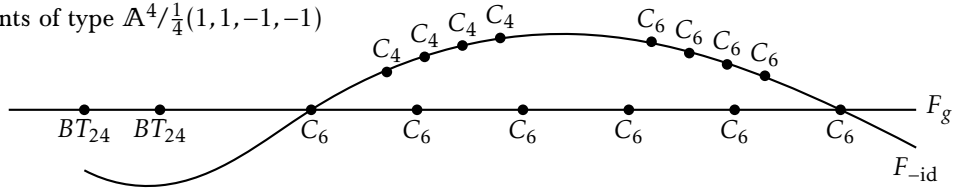
\normalsize

\section{Birational orbifolds}
\label{sec:birational orbifolds}
In this section, we show some exceptional birational maps between terminalizations of different quotients of IHS varieties; see \cref{prop:mainsection}. While the determination of Betti numbers, fundamental groups and singularities are essentially algorithmic, determining whether projective terminalizations of two different quotients $X_1/G_1$ and $X_2/G_2$ are deformation equivalent represents a subtle task. An obvious necessary condition is that their deformation invariants coincide, namely the corresponding rows in Table~\ref{table sing n=2} or in \cite[Table in Theorem 1.11]{GM22} are identical. When this is the case (with the single open exception of \cref{rmk:open}), we find an explicit birational map between $X_1/G_1$ and $X_2/G_2$, so that their terminalizations are deformation equivalent; see \cref{prop:mainsection,prop:birationaldeformation}. The idea is to write $G_1$ as an extension of $G_2$ by a normal subgroup $N_1$ and then to show that $X_2$ is birational to $X_1/N_1$, equivariantly with respect to the given $G_2$-action on $X_2$ and the residual $G_2$-action on $X_1/N_1$. As a result, we obtain
\[X_1/G_1 = (X_1/N_1)/G_2 \sim_{\bir} X_2/G_2.\] 
Even when $X_1=X_2=\K_{2}(A)$, we can still run the argument: It suffices to find an isogeny $f \colon A \to A$ such that $N_1=\ker(f)$. The ultimate goal is to merge the classification of irreducible symplectic varieties in this paper with \cite[Theorem 1.11]{GM22}, avoiding redundancy.

\begin{notation}\label{notationKummer}
    Let $a \colon G \times \K_{n}(A) \to \K_{n}(A)$ be the action of a finite group of symplectic automorphisms $G$ on $\K_{n}(A)$. A \new{projective} terminalization of the quotient $\K_{n}(A)/G$ is denoted by $\K_n(A, a)$. In the following, we always assume that the action $a$ is induced by a symplectic action on the underlying abelian surface $A$, and we simply write $\K_n(A, G)$ when the action $a$ of $G$ is clear. 
\end{notation}

\begin{definition}\label{defb:Fujiki variety}
    Let $G$ be a finite group of symplectic automorphisms of a K3 surface $S$. Let $\theta \colon G \to G$ be an involution (which may also be  the identity). The group $G$ acts on $S^n$ by 
    \[g(x_1, x_2, x_3, \ldots, x_n)=(g(x_1), \theta(g)(x_2),x_3, \ldots, x_n),\] and the symmetric group $S_n$ permutes the factors of $S^n$. 
    A \emph{Fujiki variety}, denoted by $S(G)_{\theta}^{[n]}$, is a terminalization of the quotient $S^n/\langle G, S_n\rangle$. In particular, we have
    \[S(G)_{\theta}^{[n]} \sim_{\bir} S^n/\langle G, S_n\rangle.\]
\end{definition}

\begin{proposition}\label{prop:mainsection}
The following couples or triples of  symplectic orbifolds with simply connected regular locus are deformation equivalent: \vspace{0.5em}
\begin{enumerate} \setlength\itemsep{0.5em}
    \item\label{p:m-1} $\K_2(A, C_2) \sim \K_2(A, C^4_3 \rtimes C_2)$,
    \item\label{p:m-2} $\K_2(A, S_3) \sim \K_2(A, C^3_3 \rtimes C_2)$, 
    \item\label{p:m-3} $\K_2(A, C_3) \sim S(C^2_3)^{[2]}_{-\id}$,
    \item\label{p:m-4} $\K_2(A, C^2_3) \sim S(C_3)_{-\id}^{[2]}$,
    \item\label{p:m-5} $\K_2(A, C_6) \sim \K_2(A, C^4_3 \rtimes C_6)\sim S(C_3 \rtimes S_3)_{\id}^{[2]}$,
    \item\label{p:m-6} $\K_2(A, C_3 \rtimes C_6) \sim  \K_2(A, C^3_3 \rtimes_{4} C_6) \sim S(S_3)^{[2]}_{\id}$,
    \item\label{p:m-7} $\K_2(A, C^2_3 \rtimes_{4} C_6) \sim S(C_2)^{[2]}_{\id}$
    \item\label{p:m-8} $\K_2(A, C^2_3 \rtimes C_6) \sim S(C_3 \rtimes S_3)_{(-\id,\id)}^{[2]}$,
    \item\label{p:m-9} $\K_2(A, BT_{24}) \sim \K_2(A, C^4_3 \rtimes BT_{24}) 
    $,
    \item\label{p:m-10}  $\K_3(A, C_2^i \times C_2) \sim S(C_2^{4-i})^{[3]}_{\id}$ for $0 \leq i \leq 4$.
\end{enumerate}\vspace{0.5em}
In all cases above, the group $G$ acts on \(A\)
as the affine group $G_\tr \rtimes G_\circ$ $($see \cref{lem:G is affine} and Equation~\eqref{equ:G_0}$)$.
For suitable choices of surfaces $A$ and $S$ and actions, the orbifolds in each row are actually birational.
\end{proposition}

\begin{remark}\label{rmk:open}
    The IHS orbifolds $\K_2(A, BT_{24})$ and $S(S^2_3 \rtimes C_2)^{[2]}_{\id}$ share the same Betti numbers and singularities. They could be a pair of deformation equivalent orbifolds, but the lemmas in this section are not sufficient to decide it.
\end{remark}

\begin{proof}[Proof of \cref{prop:mainsection}]
The proposition follows from Equations~\eqref{eq:birationalKA2} and~\eqref{eq:birationalKA3} below, \cref{prop:birationaldeformation}, and
\begin{itemize}
    \item[-] \cref{lem:quotientC4} for \eqref{p:m-1}, \eqref{p:m-2}, \eqref{p:m-5}, \eqref{p:m-6} and \eqref{p:m-9}, 
    \item[-] \cref{lem:K2SG} for \eqref{p:m-5}, \eqref{p:m-6} and \eqref{p:m-7}, 
    \item[-] \cref{lem:K2SII} for \eqref{p:m-3}, \eqref{p:m-4} and \eqref{p:m-8}, 
    \item[-] \cref{lem:quotientK3} for \eqref{p:m-10}. 
\end{itemize}
In order to apply \cref{lem:K2SG,lem:K2SII} in cases \eqref{p:m-3}--\eqref{p:m-8}, we first deform the pair $(A,G)$ to $(E^2, G)$, where $E$ has complex multiplication. This is possible since the moduli space of such pairs $(A,G)$ is connected by \cite[Proposition 3.7]{Fujiki88}.
\end{proof}

\begin{definition}
    Let $f\colon X \to Y$ be a morphism of algebraic varieties. An automorphism $h \colon X \to X$ descends along $f$ to an automorphism $\bar{h} \colon Y \to Y$ if the following square commutes:  
    \[
\begin{tikzcd}[scale=1]
   X \arrow[d, "f"'] \arrow[r, "h"] & X \arrow[d, "f"]\\
   Y \arrow[r, "\bar{h}"]& Y.
\end{tikzcd}
\]
\textit{Vice versa}, we say that $\bar{h}$ lifts to $h$ along $f$.
\end{definition}

\subsection{Birational orbifolds in dimension 4} 
Let $G$ be a finite group of induced symplectic automorphisms of $\K_{2}(A)$. 
By construction, we have the following birational map:
\begin{equation}\label{eq:birationalKA2}
    \K_{2}(A, G) \sim_{\bir}A^{(3)}_{0}/G \simeq A^2/(S_3 \times G),
\end{equation}
 where $G$ acts diagonally on $A^2$ and the action of $S_3$ on $A^2$ is given by 
\begin{equation}\label{eq:S3}\sigma(x,y)=(y,-x-y), \quad \tau(x,y)=(y,x).
\end{equation}

\begin{lemma}\label{lem:quotientC4}
Let $H \simeq C_3^k \subseteq A[3]$ for $0 \leq k \leq 4$, \new{acting by translation on $A$. Assume that $G_\circ$ contains $-\id$.} Then the following quotients are isomorphic:
\[
\frac{(A/H)^2}{S_3\times (A[3]/H) \rtimes G_\circ} \simeq \frac{A^2}{S_3\times H \rtimes G_\circ}, 
\]
where 
the linear symplectic group $G_\circ$ and the translation group $H$ $($respectively, $A[3]/H)$ act diagonally on $A^2$ $($respectively, $(A/H)^2)$.
\end{lemma}

\begin{proof}
Consider the isogeny $f_0 \colon A^2 \to A^2$ given by $f_0(x,y)=(x+2y,x-y)$, whose kernel is the diagonal copy of $A[3]$ in $A^2[3]$.
The automorphisms $\sigma$, $\tau$ in \eqref{eq:S3} and  $g \in G_\circ$ 
descend along $f_0$ to
\[\bar{\sigma}(x,y)=(-x-y, x), \quad \bar{\tau}(x,y)=(x+y, -y), \quad \bar{g}=g. 
\]
Note that the group $\langle \bar{\sigma}, \bar{\tau} \rangle=\langle \bar{\sigma}^2, \bar{\sigma}^2 \bar{\tau}\rangle=\langle \sigma, -\tau\rangle$ acts via the standard action of $S_3$ up to a sign.
Hence, the action of $S_3 \times A[3] \rtimes G_\circ = S_3 \times \ker(f_0)\rtimes G_\circ$ descends along $f_0$ to the action of $S_3 \times G_\circ$ since
\[\left\langle \bar{\sigma}, \bar{\tau}, \overline{A[3]}, \overline{-\id}, \bar{g} \right\rangle
=\left\langle \bar{\sigma}^2, \bar{\sigma}^2 \bar{\tau}, -\id, g\right\rangle
=\langle \sigma, \tau,  -\id, g\rangle.\]

Further, for any $(\alpha, \beta) \in A^2[3]$, the translation $\tau_{(\alpha,\beta)}$ descends along $f_0$ to $\bar{\tau}_{(\alpha,\beta)}=\tau_{(\alpha-\beta, \alpha - \beta)}$.
In particular, the anti-diagonal $H^- \new{\coloneqq \{(\alpha, - \alpha)\}}\subset H^2 \subset A^{2}[3]$ descends along $f_0$ to the diagonal $H \new{\coloneqq \{(\alpha, \alpha)\}} \subset H^2 \subset A^{2}[3]$.
We conclude that
\begin{equation*}\pushQED{\qed}
\frac{(A/H)^2}{S_3\times (A[3]/H) \rtimes G_\circ} \simeq \frac{A^2}{S_3\times (H^{-} \times A[3]) \rtimes G_\circ} \simeq \frac{A^2}{S_3\times H \rtimes G_\circ}.
\qedhere \popQED
	\end{equation*}
\renewcommand{\qed}{} 
\end{proof}

Let $\xi_3$ be a primitive third root of unity, and let $E$ be an elliptic curve with complex multiplication $\xi_3 \curvearrowright E \colon x \mapsto \xi_3 \cdot x$. Denote by $g_3 \colon E^2 \to E^2$ the diagonal automorphism $g_3(x_1,x_2)=(\xi_3 x_1, \xi^{-1}_3 x_2)$.

\begin{lemma}\label{lem:K2SG}
Let $G'$ be a finite symplectic group acting diagonally on $E^2$, and set $G \coloneqq  \langle \Pi_{g_3}, g_3, G' \rangle$. 
The group $G'$ acts on the K3 surface \new{$S \sim_{\bir} E^2/\langle g_3 \rangle$}, and the following orbifolds are birational: 
\[\K_2(E^2, G) \sim_{\bir}  S(G')^{[2]}_{\id}\]
\end{lemma}
\begin{proof} We follow closely \cite[Proof of Theorem 4.2]{kawatani2009birational}.
As in \eqref{eq:birationalKA2}, there exists a birational map 
\begin{equation*} 
\K_2(E^2, G) \sim_{\bir}  E^4/(S_3 \times G).
\end{equation*} Consider the isogeny $f_1 \colon E^4 \to E^4$ given by \[f_1(x_1,x_2, x_3, x_4)=\left(\xi^{2}_3 x_1-x_3, \xi^{2}_3 x_2-x_4, -\xi_3 x_1+x_3, -\xi_3 x_2+x_4\right),\] whose kernel is the diagonal copy of $\Pi_{g_3}$ in $E^4[3]$.  
The automorphisms $\sigma$, $\tau$ and $g_3$ descend along $f_1$ to $\bar{\sigma}$, $\bar{\tau}$ and $\bar{g}_3$ such that
\begin{align}
    \bar{g}_3\bar{\sigma}(x_4,x_1, x_2, x_3) & =\left(\xi_3 x_4, \xi^2_3 x_1, x_2, x_3\right), \nonumber\\
    \bar{g}_3\bar{\sigma}^2(x_4,x_1, x_2, x_3) & =\left( x_4, x_1, \xi_3 x_2, \xi^2_3 x_3\right), \label{eq:descent}\\
    \bar{\tau} \bar{\sigma}(x_4, x_1,x_2, x_3)& =(x_2,x_3,x_4,x_1). \nonumber
\end{align} 
In particular, we obtain
\begin{equation} \label{eq:K3bir}
E^4/\langle \bar{\sigma}, \bar{\tau}, \bar{g}_3 \rangle = E^4/\langle \bar{g}_3\bar{\sigma}, \bar{g}_3\bar{\sigma}^2, \bar{\tau}\bar{\sigma}  \rangle \simeq \left((E^2/g_3) \times (E^2/g_3)\right)/\bar{\tau} \bar{\sigma} \sim_{\bir} S^{[2]}.
\end{equation}

An element $g'\in G'$ is of the form \[g'(x_1,x_2,x_3, x_4)=(cx_1+a, dx_2 + b, cx_3+a, dx_2+b)\] for some $c,d \in \mathbb{C}$ and $a,b \in E[3]$, and it descends along $f_1$ to \[\bar{g}'(x_4, x_1, x_2, x_3) =\left(dx_4+\bar{b}, cx_1+\bar{a}, dx_2+\bar{b}, cx_1+\bar{a}\right),\] with $\bar{a}=\xi^2_3a-a$ and $\bar{b}=\xi^2_3b-b$. Since $\xi_3 \bar{a}=\bar{a}$ and $\xi_3 \bar{b}=\bar{b}$, the morphism $\bar{g}'$ commutes with all the automorphisms $\bar{\sigma}$, $\bar{\tau}$ and $\bar{g}_3$. 

We conclude that
	\begin{equation*}\pushQED{\qed}
\K_2(E^2, G) \sim_{\bir} E^4/(S_3 \times G) \simeq E^4/(\langle \bar{\sigma}, \bar{\tau}, \bar{g}_3 \rangle \times G') \sim_{\bir} S^{[2]}/G'\sim_{\bir} S(G')^{[2]}_{\id}.
\qedhere \popQED
	\end{equation*}
\renewcommand{\qed}{}     
\end{proof}

\begin{lemma}\label{lem:K2SII}
The following orbifolds are birational: \vspace{0.5em}
\begin{itemize} \setlength\itemsep{0.5em}
\item 
   $\K_2(E^2, C_3) \sim_{\bir} S(C^2_3)_{-\id}^{[2]}$ with $C_3=\langle g_3\rangle$ and $S \sim_{\bir} E^2/\langle g_3 \rangle$,
\item $\K_2(E^2, C^2_3)  \sim_{\bir} S_{\alpha}(C_3)_{-\id}^{[2]}$ with $C_3^2 = \langle g_3, \tau_{\alpha}\rangle$ and $S_{\alpha} \sim_{\bir} E^2/\langle g_3, \tau_{\alpha}\rangle$,
\item $\K_2(E^2, C^2_3 \rtimes C_6) \sim_{\bir} S_{\alpha}(C_3 \rtimes S_3)_{\theta}^{[2]}$ 
  with $C^2_3 \rtimes C_6= \langle g_3, -\id, \tau_{\alpha}, \tau_{\beta}\rangle$, $g_3(\beta) \neq \beta$ and $\theta=(-\id, \id)$ acting on $C_3 \rtimes S_3$.
\end{itemize}
  \end{lemma}

\begin{proof} 
  Consider the isogeny $f_2 \colon E^4 \to E^4$ given by
  \[f_2(x_1,x_2, x_3, x_4)=\left( x_1+x_3, x_2+x_4, \xi_3 x_1+\xi_3^2x_3, \xi_3 x_2+\xi_3^2 x_4\right),\] whose kernel is the anti-diagonal copy of $\Pi_{g_3}$ in $E^4[3]$; \textit{i.e.},  
 \[\ker(f_2)=\{(a,b,-a,-b) \in E^4\, | \, \xi_3(a)=a, \xi_3(b)=b\}\subseteq E^4[3].\]
 The automorphisms $\sigma, \tau$ in \eqref{eq:S3} and $g_3$ lift along $f_2$ to the automorphisms $\tilde{\sigma}$, $\tilde{\tau}$ and $\tilde{g}_3$ such that 
 \begin{align*}
    \tilde{g}_3\tilde{\sigma}(x_4,x_1, x_2, x_3) & =\left(\xi_3 x_4, \xi^2_3 x_1, x_2, x_3\right), \nonumber\\
    \tilde{g}_3 \tilde{\sigma}^2(x_4,x_1, x_2, x_3) & =\left( x_4, x_1, \xi_3 x_2, \xi^2_3 x_3\right), \label{eq:descent}\\
    \tilde{\tau}\tilde{\sigma}^2(x_4, x_1,x_2, x_3)& =(x_2,x_3,x_4,x_1). \nonumber
\end{align*} 
Thus, the group \( \langle \tilde{\sigma}, \tilde{g}_3\rangle \simeq C_3^2 \) acts 
on $E^2_{x_4,x_1} \times E^2_{x_2,x_3}$ as $\langle g_3 \rangle \times \langle g_3 \rangle$, while the group \(\langle \ker(f_2), \tilde{\tau}\tilde{\sigma}^2\rangle \simeq C^2_3 \rtimes C_2 \) acts on \(E^4/\langle \tilde{\sigma}, \tilde{g}_3\rangle \sim_{\bir} S^2\) as the group $\langle C_3^2, C_2 \rangle$ in \cref{defb:Fujiki variety} with $\theta=-\id$.

From the short exact sequence
 \[1 
 \lra C_3^2 = \langle \tilde{\sigma}, \tilde{g}_3\rangle  
 \lra (C^3_3 \rtimes C_2) \times C_3  = \langle \ker(f_2), \tilde{\sigma}, \tilde{\tau}, \tilde{g}_3\rangle 
 \lra C^2_3 \rtimes C_2 = \left\langle \ker(f_2), \tilde{\tau}\tilde{\sigma}^2\right\rangle 
 \lra 1,\]
we obtain that
\begin{align*}
    \K_2(E^2, C_3) \sim_{\bir} & E^4/(S_3 \times C_3) 
    = E^4/ \langle \sigma, \tau, g_3 \rangle
    \simeq E^4/\langle \ker(f_2), \tilde{\sigma}, \tilde{\tau}, \tilde{g}_3\rangle = E^4/((C^3_3 \rtimes C_2) \times C_3) \\
    & \simeq \left((E^2/g_3) \times (E^2/g_3)\right)/ (C^2_3 \rtimes C_2) 
    \sim_{\bir}  S^2/ (C^2_3 \rtimes C_2) 
    \sim_{\bir}S(C^2_3)_{-\id}^{[2]}.
\end{align*}

A 3-torsion point $\alpha$ in the diagonal $E^2[3] \subset E^4[3]$ lifts along $f_2$ to its opposite $-\alpha$, up to a translation in $\ker(f_2)$. If $\alpha =(a,b)$ is a nonzero translation in $\Pi_{g_3}$, then $\langle \ker(f_2), \tilde{\tau}_{\alpha}\rangle$ is generated by three translations
\[\tau_1 \coloneqq (a,0,0,b), \quad \tau_2  \coloneqq (0,b,a,0), \quad \tau_3 \in \ker(f_2) \setminus \langle (a,-b,-a,b) \rangle.\] 
From the short exact sequence 
 \[1 \lra C_3^4 = \langle \tilde{\sigma}, \tilde{g}_3, \tau_1, \tau_2 \rangle 
 \lra \langle \ker(f_2), \tilde{\sigma}, \tilde{\tau}, \tilde{g}_3, \tilde{\tau}_{\alpha}\rangle = \left(C^3_3 \rtimes C_2\right) \times C^2_3 
 \lra S_3 = \left\langle \tau_3, \tilde{\tau}\tilde{\sigma}^2\right\rangle \lra 1,\]
we obtain that
\begin{align*}
    \K_2(E^2, C^2_3) \sim_{\bir} & E^4/\left(S_3 \times C^2_3\right) \simeq E^4/\left(\left(C^3_3 \rtimes C_2\right) \times C^2_3\right) \\ 
    & \simeq \left(E^2/\langle g_3, \tau_{(b,a)} \rangle\right)^2/ S_3 \sim_{\bir}  S_{\alpha}^2/ S_3 \sim_{\bir}S(C_3)_{-\id}^{[2]}.
\end{align*}
Quotienting further by $S_3=\langle \tau_{\beta}, -\id \rangle$ with $g_3(\beta)=\beta + \alpha$, we also obtain 
\begin{equation*}\pushQED{\qed}
\K_2\left(E^2, C^2_3 \rtimes C_6\right) \sim_{\bir} S_{\alpha}\left(C_3 \rtimes S_3\right)_{\theta}^{[2]}.\qedhere \popQED
\end{equation*}
\renewcommand{\qed}{}     
\end{proof}

\subsection{Birational orbifolds in dimension 6}
The following \cref{lem:quotientK3} was communicated to the authors by Menet. 
By construction, we have a birational map 
\begin{equation}\label{eq:birationalKA3}
    \K_{3}(A, G) \sim_{\bir}A^{(4)}_{0}/G \simeq A^3/(S_4 \times G),
\end{equation}
 where $G$ acts diagonally on $A^3$ and the action of $S_4$ on $A^2$ is given by 
\[\sigma_{12}(x,y,z)=(y,x,z), \quad 
\sigma_{13}(x,y,z)=(z,y,x), \quad \sigma_{14}(x,y,z)=(-x-y-z,y,z).\]

\begin{lemma}\label{lem:quotientK3}
Let $H\simeq C_2^k \subseteq A[2]$ for $0 \leq k \leq 4$, and set $G \coloneqq A[2]/H \times \langle -\id \rangle$. \new{The group $H$ acts by translation on $A$, and it induces an action on the corresponding Kummer surface $T \sim_{\bir} A/\langle -\id \rangle$.}
Then the following orbifolds are birational: 
\[\K_3(A/H, G) \sim_{\bir}  T(H)^{[3]}_{\id}.\]
\end{lemma}

\begin{proof} 
Consider the isogeny $f_3 \colon A^3 \to A^3$ given by $f_3(x,y,z)=(x+y,x+z,y+z)$, whose kernel is the diagonal copy of $A[2]$ in $A^3[2]$.
The automorphisms $\sigma_{12}$, $\sigma_{13}$, $\sigma_{14}$ and $-\id$ descend along $f_3$ to, respectively,  the permutations $(23)$, $(13) \in S_3$ of the factors of $A^3$, and
\[\bar{\sigma}_{14}(x,y,z)=(-y,-x,z)=(-\id, -\id, \id)(12)(x,y,z), \quad \overline{-\id} =(-\id, -\id, -\id).\]
Hence, the action of $S_4 \times A[2] \times \langle -\id \rangle= S_4 \times \ker(f_3) \times \langle -\id \rangle$ descends along $f_3$ to the action of $S_3 \times \langle-\id \rangle^3$, and
\[ \frac{A^3}{S_4 \times A[2] \times \langle -\id \rangle}
\simeq 
\frac{A^3}{S_3 \times \langle -\id \rangle^3} \sim_{\bir} T^{[3]}.\] 
Further, for any $(\alpha, \beta, \gamma) \in A^3[2]$, the translation $\tau_{(\alpha,\beta,\gamma)}$ of $A^3$ descends along $f_3$ to \[\bar{\tau}_{(\alpha,\beta,\gamma)}=(\tau_\alpha, \tau_\alpha, \id)(\tau_\beta, \id, \tau_\beta)(\id, \tau_\gamma, \tau_\gamma).\]
In particular, the action of $H^3 \subseteq A^{3}[2]$ descends along $f_3$ to \new{the action of $H^3 \subset \langle H, S_3\rangle$} as in \cref{defb:Fujiki variety} with $\theta=\id$.
We conclude that 
\begin{equation*}\pushQED{\qed}
\frac{(A/H)^3}{S_4 \times A[2]/H \times \langle -\id \rangle} \simeq \left(\frac{A^3}{S_4 \times A[2] \times \langle -\id \rangle}\right) \bigg/ H^3
\simeq 
\frac{A^3}{\langle H, S_3 \rangle \times \langle -\id \rangle^3}  \sim_{\bir} T(H)_{\id}^{[3]}.\qedhere \popQED
\end{equation*}
\renewcommand{\qed}{} 
\end{proof}


\newcommand{\etalchar}[1]{$^{#1}$}

\end{document}